%% file: Log1Mot_final.tex
\documentclass[11pt, a4paper]{amsart}

\input{include}

\def\lSm{\mathbf{lSm}}
\def\SmlSm{\mathbf{SmlSm}}
\def\Sm{\mathbf{Sm}}
\def\sM{\mathcal{M}}
\newcounter{elno}

\begin{document}

\input{defin}

\title{Derived Log Albanese sheaves}   
\author{Federico Binda, Alberto Merici and Shuji Saito}

\address{Dipartimento di Matematica ``Federigo Enriques'',  Universit\`a degli Studi di Milano\\ Via Cesare Saldini 50, 20133 Milano, Italy}
\email{federico.binda@unimi.it}

\address{Department of Mathematics, University of Oslo, Niels Henrik Abels hus, Moltke Moes vei 35, 0851 Oslo, Norway}
\email{allbertm@math.uio.no}

\address{Graduate School of Mathematical Sciences, University of Tokyo, 3-8-1 Komaba, Tokyo 153-8941, Japan}
\email{sshuji@msb.biglobe.ne.jp}

\thanks{F.B.\ is supported by the PRIN ``Geometric, Algebraic and Analytic Methods in Arithmetic''.}
\thanks{A.M.\ is supported by the Swiss National Science Foundation (SNSF), project 200020\_178729}
\begin{abstract}We define higher pro-Albanese functors for every effective log motive over a field $k$ of characteristic zero, and we compute them for every smooth log smooth scheme $X=(\underline{X}, \partial X)$. The result involves an inverse system of the coherent cohomology of the underlying scheme as well as a pro-group scheme $\mathrm{Alb}^{\log}(X)$ that extends Serre's semi-abelian Albanese variety of $\underline{X}-|\partial X|$. This generalizes the higher Albanese sheaves of Ayoub, Barbieri-Viale and Kahn and is related to an old question of Grothendieck.  

\end{abstract}
\maketitle
\tableofcontents

\def\shujirem#1{\begin{color}{red}{#1}\end{color}}
\def\shuji#1{\begin{color}{green}{#1}\end{color}}

\section{Introduction}
Let $k$ be a perfect field and let $X$ be a smooth, quasi-projective and geometrically connected $k$-scheme. A very classical tool in the study of the geometry of $X$ is given by the \emph{Albanese variety} $\Alb_X$ of $X$, the universal Abelian variety receiving a map from $X$ (up to the choice of a base point). When $X$ is a smooth curve, the Albanese variety coincides with the Jacobian variety  ${\rm Jac}(X)$ of $X$, and essentially every invariant of $X$ can be recovered from it. 
In higher dimension, the Albanese variety is still an important tool for gathering information about the Chow group $\CH_0(X)$ of zero cycles of $X$. Extending the Albanese map by linearity, there is in fact a well-defined morphism (now independent on the choice of a base point) 
\begin{equation}\label{eq:intro_alb} {a}_X\colon \CH_0(X)^0 \to \Alb_X(k).\end{equation}
Much is known, at least conjecturally, about this map. If $X$ is proper over an algebraically closed field, a famous theorem of Rojtman \cite{Rojtman} asserts that ${a}_X$ is an isomorphism on torsion subgroups (at least modulo $p$-torsion in characteristic $p>0$, which was later fixed by Milne \cite{MilneTorsion}). If $k$ is finite, the kernel of $a_X$ can be explicitly determined by geometric class field theory \cite{KatoSaitoUnCFT}. If $k$ is the algebraic closure of a finite field, then a theorem of Kato and Saito (see again \cite{KatoSaitoUnCFT}) asserts that ${a}_X$ is in fact an isomorphism, a statement that is conjectured to be true even when $k=\overline{\mathbb{Q}}$ as consequence of the Bloch-Beilinson conjectures. This is far from being true over the complex numbers, as shown by Mumford. 

When $X$ is no longer proper, both sides of \eqref{eq:intro_alb} need to be modified. It is already clear from the case of curves \cite{SerreGACC} that one can consider a more general class of commutative algebraic groups as target of a map from $X$, including Abelian varieties, tori, and their extensions, i.e. semi-Abelian varieties. Serre \cite{SerreChevalley1} (see also \cite{FaltingsWustholz}) showed that the problem of finding a universal map to a semi-Abelian variety has always a solution. The corresponding universal object is now known as Serre's Albanese variety: it agrees with
the usual Albanese variety if $X$ is proper. 

Using Serre's semi-Abelian Albanese variety it is possible to extend the Albanese morphism to every smooth quasi-projective variety\footnote{at least after  inverting the exponential characteristic of the ground field, in an appropriate sense.}. As observed by Spie\ss\ and Szamuely \cite{SS}, every semi-Abelian variety, seen as \'etale sheaf on the big  site $\Sm(k)$, has a natural structure of \'etale sheaf with transfers, i.e. it enjoys an extra functoriality with respect to the category of finite correspondences $\Cor(k)$ introduced by Suslin and Voevodsky. Since every map from an affine space to a torus or to an Abelian variety is constant, such sheaves are moreover ${\A}^1$-homotopy invariant. 

These two facts are essentially enough to show that the assignment $X\mapsto \mathbf{Alb}_X$ (here $\mathbf{Alb}_X$ is the non-connected algebraic group whose neutral component is exactly Serre's semi-Abelian Albanese) can be promoted by left Kan extension to a motivic ``realization'' functor
\begin{equation}\label{eq:LAlb_intro_DM}
     L\Alb \colon \mathcal{DM}^{\rm eff}_{\et}(k, \Q) \longrightarrow \cD(\HI_{\leq 1,\et}(k,\Q))\end{equation}
defined on the $\infty$-category of Voevodsky's effective motives $\mathcal{DM}^{\rm eff}_{\et}(k, \Q)$, i.e. the full subcategory of the derived $\infty$-category of \'etale sheaves with transfers $\cD(\Shv_{\et}^{\rm tr}(k, \Q))$ whose objects are $\mathbb{A}^1$-local complexes, taking values in the derived $\infty$-category of the Abelian category $\HI_{\leq 1,\et}(k,\Q)$ of $1$-motivic sheaves with rational coefficients: this is the full subcategory of \'etale sheaves with transfers  generated under colimits by lattices (i.e. \'etale sheaves $L$ such that $L(\overline{k}) \cong \Z^r$) and semi-abelian varieties (see \cite[Prop. 1.3.8]{Ayoub-BV}). It is naturally a full subcategory of the abelian category of homotopy invariant sheaves with transfers. 
This result can in fact be refined to integral coefficients by considering a more exotic (and probably not fully faithful) functor \[R\iota\colon \cD(\HI_{\leq 1,\et})\to \mathcal{DM}^{\rm eff}_{\Nis}(k,\Z),\]
and by constructing $L\Alb$ as its left adjoint (this functor is in fact much more mysterious, and we will not take it in consideration). This result, due to Ayoub and Barbieri-Viale \cite[Thm. 2.4.1]{Ayoub-BV} (extending Barbieri-Viale and Kahn \cite{BVKahn} to non necessarily geometric motives) has several consequences. First, it provides  a construction of an Albanese map for arbitrary motives (in particular, for every separated $k$-scheme of finite type, not necessarily smooth or proper), giving for example vast generalizations of the theorem of Rojtman \cite[13]{BVKahn}. Second, the Albanese functor is now a \emph{derived} functor: it has higher homotopy groups $L_i \Alb(M) = \pi_i(L\Alb(M))$ for every $M\in \mathcal{DM}^{\rm eff}_{\et}(k, \Q)$,  encoding   information such as the N\'eron-Severi group  of a variety (see \cite[Thm. 9.2.3]{BVKahn}).  
Moreover, the functor $L\Alb$ in  \eqref{eq:LAlb_intro_DM} can be identified with the left adjoint of the derived functor $i_{\leq 1}^{\DM}$ of the natural embedding  $\HI_{\leq 1,\et}(k,\Q) \subset \Shv_{\et}^{\rm tr}(k, \Q)$. One can show that $i_{\leq 1}^{\DM}$ is fully faithful, and its essential image coincides with the stable $\infty$-category $\mathcal{DM}^{\rm eff}_{\leq 1}(k,\Q)$ generated by the motives of curves. If we restrict ourselves to compact objects, $\cD(\HI_{\leq 1,\et}(k,\Q))^\omega$ coincides with the (bounded) derived category of the Abelian category of Deligne $1$-motives introduced in \cite{DeligneHodgeIII}. In fact, the properties of \eqref{eq:LAlb_intro_DM} are essential in the ``motivic'' proof of Deligne's conjectures on 1-motives, see \cite[Part 4]{BVKahn} and \cite{Vologodsky}.
\medskip 

Our goal in this paper is to extend the  picture sketched above in order to include a more general kind of algebraic groups in the definition of the Albanese variety. Thanks to Chevalley's structure theorem, any connected commutative algebraic group over a perfect field can be written as an extension of an Abelian variety by an affine smooth group scheme, which splits as a product of a torus by a unipotent commutative group.  As observed by Serre, however, the problem of finding a universal map from a smooth variety $X$ to an arbitrary commutative algebraic group does \emph{not} have  a solution in general (namely, when $X$ is not proper), whence the classical restriction to semi-Abelian varieties. A solution does, however, exist, if a bound on the dimension of the tangent spaces of the groups is imposed.
Let us assume that $k$ has characteristic zero (and keep this assumption throughout the rest of the Introduction, see Proposition \ref{prop:AlbOmega-AlbChow-char0}). Faltings and W\"ustholz \cite{FaltingsWustholz} realized that when $X$ admits a smooth compactification $\ol{X}$ with normal crossing boundary $D$, it is possible to use any finite dimensional subspace of the vector space $H^0(X, \Omega^1_{X})$ to give such a bound. A natural choice is to use for $n\geq 1$ the subspaces $H^0(\ol{X}, \Omega^1_{\ol{X}}(nD))$ of regular $1$-forms on $X$ having poles of order at most $n$ along $D$. The resulting universal object $\Alb_{(\ol{X}, nD)}$ depends on the pair $\mathfrak{X}^{(n)}:=(\ol{X}, nD)$ in a functorial way. This gives a generalized Albanese morphism
\[ a_{\mathfrak{X}}\tensor \Q \colon \Q_{tr}(X) \to \mathbf{Alb}_{\mathfrak{X}^{(n)}} \tensor_{\Z} \Q, \]
which is a surjective morphism of   \'etale sheaves with transfers with rational coefficients (here $\Q_{tr}(X)$ denotes the \'etale sheaf of $\Q$-vector spaces represented by $X$). The generalized Albanese $\mathbf{Alb}_{\mathfrak{X}^{(n)}}$ is an extension of Serre's semi-Abelian Albanese of $X$ (independent on the choice of the compactification $\ol{X}$) by a unipotent group. If $X$ is a curve, $\mathbf{Alb}_{\mathfrak{X}^{(n)}} = \Jac(\ol{X}, nD)$ is exactly the generalised Jacobian variety of Rosenlicht and Serre \cite{SerreGACC}, and in higher dimension it is the generalised Albanese with modulus considered in \cite{BS}, \cite{BK} (see also \cite{RussellKyoto}, \cite{RussellANT}).

By varying $n$, we get a pro-object in the category of commutative algebraic groups up to isogeny $``\lim_n"\mathbf{Alb}_{\mathfrak{X}^{(n)}}$, which satisfies an obvious universal property, see Prop. \ref{prop:Alb-Omega-PST}.

In fact, we can give a finer result. Let $\mathbf{RSC}_{\et, \leq 1}(k, \Q)$ be the full (abelian) subcategory of the category of \'etale sheaves with transfers $\Shv_{\et}^{\rm tr}(k, \Q)$ generated under  colimits by  commutative connected $k$-group schemes of finite type and lattices. Note that we clearly have $\mathbf{HI}_{\et, \leq 1}(k, \Q)\subset \mathbf{RSC}_{\et, \leq 1}(k, \Q)$.  Write $\Comp(X)$ for the category of normal compactifications $\overline{X}$ of $X$ such that the complement $\ol{X}-X$ is the support of an effective Cartier divisor. 
\begin{thm}\label{thm:Alb_intro}(see Thm. \ref{thm:Alb}) Assume that the characteristic of $k$ is zero.
The embedding $\RSC_{\et,\leq 1}(k,\Q) \subseteq \Shv_{\et}^{\tr}(k,\Q)$ has a pro-left adjoint:
	\begin{equation}\label{eq:thm:Alb_intro}
	\Alb\colon \Shv_{\et}^{\tr}(k,\Q)\to \pro\RSC_{\et,\leq 1}(k,\Q), \end{equation}
induced by colimit from \[
\Q_{\tr}(X)\mapsto  ``\lim_{n}"\bAlb_{\kX^{(n)}}
\]
for any choice of $\kX\in \Comp(X)$ smooth. 
\end{thm}
It is natural to ask for a derived version of the above Theorem, in the spirit of the result of Ayoub, Barbieri-Viale and Kahn. However, since unipotent group schemes are $\A^1$-contractible, \eqref{eq:thm:Alb_intro} cannot be extended to $\mathcal{DM}^{\rm eff}_{\et}(k, \Q)$ in a non-trivial way, i.e. without simply collapsing to the subcategory $\mathbf{HI}_{\et, \leq 1}(k, \Q)$, recovering \eqref{eq:LAlb_intro_DM}. 

Our solution to this difficulty is to extend the construction to a framework in which $\A^1$-contractibility is no longer a problem. This is achieved by passing from the world of algebraic geometry to the world of \emph{logarithmic algebraic geometry}, in the sense of Fontaine, Illusie, Kato and others. Over a field $k$ (seen as log scheme with trivial log structure), we can, roughly speaking, replace schemes with log pairs $X:=(\underline{X},\partial X)$, where $\underline{X}$ is the \emph{underlying} $k$-scheme and $\partial X$ is a log structure supported on a Cartier divisor (the so-called compactifying log structure associated to the open embedding of schemes $\underline{X}-|\partial X| \hookrightarrow \underline{X}$).

For $X$ a smooth log-smooth scheme (i.e. a log scheme such that the underlying scheme is smooth and the log structure is supported on a normal crossing divisor,  see \ref{ssec:loggeom}), we write $X^\circ:= \ul{X}-|\partial X|$ and $\mathbf{Comp}(X)$ for the category of normal compactifications $\overline{X}$ of $\underline{X}$ such that $\overline{X}-X^{\circ}=|\partial X|+D$, for $D$ an effective Cartier divisor. We can then consider the pro-algebraic group  \[\bAlb_X \coloneqq``\lim"_{n} \bAlb_{(\overline{X},|\partial X|_{\rm red} + nD)}\] as an invariant of the log scheme $X$. 
Note that we recover the previous pro-Albanese in the case where  
$X=(\underline{X}, \triv)$, i.e. the scheme $\underline{X}$ seen as log scheme with trivial log structure.

In order to 
exploit this formalism, we need another observation. Any commutative group scheme $G$ (not just semi-Abelian varieties) gives rise to an \'etale sheaf with transfers (still denoted $G$) on $\Sm(k)$. As such, it belongs to the subcategory $\RSC_{\et}(k, \Z)$ of $\Shv_{\et}^{\rm tr}(k, \Z)$ of \emph{reciprocity sheaves}. Its objects satisfy the property that each section $a\in F(X)$ for any $X\in \Sm(k)$ ``has bounded ramification'', i.e. that the corresponding map $a\colon \Z_{tr}(X) \to F$ factors through a quotient $h_0(\mathfrak{X})$ associated to a pair $\mathfrak{X} = (\ol{X}, D)$ where $\ol{X}$ is a proper compactification of $X$ and $D$ is an effective Cartier divisor such that $X=\ol{X}-|D|$ (we refer to such a pair as a Cartier compactification of $X$). Thanks to \cite{shujilog}, every reciprocity sheaf $F$ is \emph{logarithmic}, i.e. it can be extended in a unique way to a functor $\mathcal{L}og(F)$ defined on the category $\SmlSm(k)$ of smooth log smooth log schemes over $k$ (see also \cite{BindaMerici} for an alternative construction).
In fact, we have that (with rational coefficients)
\[\mathcal{L}og(F) \in \logCI_{d\et} \subset \Shv_{d\et}^{\rm ltr}(k, \Q) \]
where $\Shv_{d\et}^{\rm ltr}(k, \Q)$ is the category of \emph{dividing \'etale sheaves with log transfers} introduced in \cite[Section 3]{BPO} (see also \cite{BPOCras}), i.e. sheaves for a certain Grothendieck topology on the category  $\lSm(k)$ of log smooth log schemes over $k$, equipped with an extra transfer structure with respect to an extension of Voevodsky's category of finite correspondences. The topology is generated by \'etale covers of the underlying schemes together with admissible blow-ups with center contained in the locus where the log structure is non-trivial. The category $\logCI_{d\et}$ is the Grothendieck abelian category \cite[Thm. 5.7]{BindaMerici} of strictly $\bcube:=(\P^1, \infty)$-invariant sheaves (here $(\P^1, \infty)$ denotes the log scheme $\P^1$ with compactifying log structure given by the open embedding $\A^1\hookrightarrow\P^1$).
Again by \cite[Thm. 5.7]{BindaMerici}, it is the heart of a t-structure, called the homotopy t-structure, on the $\infty$-category of effective log motives $\logDM(k, \Q)$ \cite{BPO}, i.e. the full subcategory of the derived $\infty$-category $\cD(\Shv_{d\et}^{\rm ltr}(k, \Q))$ consisting of $\bcube$-local complexes.
By Saito's theorem, we have actually a fully faithful embedding (see \eqref{eq;shuji-functor2})
\begin{equation}
    \omega_{\leq 1}^{\CI} \colon \RSC_{\et, \leq 1}(k, \Q) \hookrightarrow \logCI_{d\et} \subset \Shv_{d\et}^{\rm ltr}(k, \Q), 
\end{equation}
and passing to the derived $\infty$-categories, a functor
\[ \omega_{\leq 1}^{\logDM} \colon \cD(\RSC_{\et, \leq 1}(k, \Q)) \to \cD(\Shv_{d\et}^{\rm ltr}(k, \Q)) \xrightarrow{L_{\bcube}} \logDM(k, \Q),
\]
where $L_{\bcube}$ is the localization functor.

If we put these  facts together, we see that {for each $(\underline{X},\partial X)\in \SmlSm(k)$}, each smooth compactification {$(\overline{X},|\partial X| + D) \in \Comp(X)$} and $n\geq 1$, we can construct a strictly $\bcube$-invariant sheaf {$\omega_{\leq 1}^{\CI}(\mathbf{Alb}_{(\overline{X},|\partial X|_{\rm red} + nD)})$} defined on the category of log smooth log schemes over $k$. This extends to the motivic category in the following way.

\begin{thm}[Theorems \ref{thm:derivedAlbanese} and \ref{thm;fully-faithfyl}, Proposition \ref{prop;comparison-usual-LAlb}]\label{thm:main-introduction} Assume that the characteristic of $k$ is zero. 
	The functor $\omega^{\lDM}_{\leq 1}$  has a pro-left adjoint, the \textit{log motivic Albanese functor}:
	\[L\Alb^{\log}\colon \logDM(k, \Q) \to \Pro \cD(\RSC_{\et, \leq 1}(k, \Q)), \]
where $\Pro$ means the pro-$\infty$-category, which fits in a commutative diagram:
\[
	\begin{tikzcd}
		&\cD(\HI_{\leq 1}(k,\Q))\ar[r,"j"]&\Pro\cD(\RSC_{\leq 1}(k,\Q))\\
		\mathcal{DM}^{\rm eff}_{\leq 1}(k,\Q)\ar[r,hook]\ar[ur,bend left=10,"L\Alb_{\leq 1}"]&\mathcal{DM}^{\rm eff}(k,\Q) \ar[r,"\omega^*"]
		&\logDM(k,\Q),\ar[u,swap,"L\Alb^{\log}"]
	\end{tikzcd}
	\]
where $L\Alb_{\leq 1}$ is the (restriction of the) functor $L\Alb$ of Ayoub, Barbieri-Viale and Kahn \eqref{eq:LAlb_intro_DM}, $\omega^*$ is the natural comparison functor 
\[\omega^*\colon \mathcal{DM}^{\rm eff}(k,\Q) \to \logDM(k,\Q),\]
which is fully faithful by \cite[Thm.8.2.16]{BPO}. 
Moreover, the functor $\omega^{\lDM}_{\leq 1}$ is fully faithful and its essential image is the full stable $\infty$-subcategory of $\logDM(k,\Q)$ generated by $\omega^*\mathcal{DM}^{\rm eff}_{\leq 1}(k,\Q)$ and $\G_a[n]$.
	\end{thm}
The proof of the above theorem is fairly technical, and requires new ingredients compared to the argument given in \cite{Ayoub-BV} (among which some very explicit computations). We would like to stress that the formalism of stable $\infty$-categories is essential to generalize the usual construction of derived functors (via resolutions) to pro-adjoint functors between derived categories, as we explain in Appendix \ref{sec:appendix}. 

Theorem \ref{thm:main-introduction} in particular asserts that for all $X\in \Sm(k)$ and $G$ a commutative algebraic group (with rational coefficients), we have an equivalence
	\begin{align*}
	R\Gamma_{\et}(X,G)&\simeq \Map_{\logDM(k,\Lambda)}(M(X,\triv),\Log(G))\\
	&\simeq \Map_{\Pro\cD(\RSC_{\et,\leq 1})}(L\Alb^{\log}(X,\triv),G).
	\end{align*}
In particular, the pro-object $L\Alb^{\log}(X,\triv)$ represents the \'etale cohomology of $G$. A similar object was considered in a letter from Grothendieck to Serre \cite[August 9, 1960]{GScorr}, where for every smooth variety $X$, a complex of pro-agebraic groups $J_*$ was constructed by the use of local cohomology, such that $\Hom(J_*,G)$ computes the Zariski cohomology of $G$. The reference to the existence of this object was suggested to us by B. To\"en, we thank him for this.

For $X\in \SmlSm(k)$, we determine the homotopy groups $\pi_iL\Alb^{\log}(X)$ completely:
\begin{thm}[Theorem \ref{thm;compute-LAlbX}]\label{thm;compute-LAlbX-intro}
		Let $X\in \SmlSm(k)$ geometrically connected and $(\overline{X},D)$ a Cartier compactification of $X$. Then we have that\[
		\pi_iL\Alb^{\log}(X)\cong \begin{cases}
			``\lim"(H^i(\overline{X},\cO_{\overline{X}}(nD))^\vee\otimes_k\G_a) &\textrm{for }2\leq i\leq \dim(X)\\
			\substack{``\lim"\Bigl((H^1(\overline{X},\cO_{\overline{X}}(nD))/H^1(\overline{X},\cO_{\overline{X}}))^\vee\otimes_k\G_a\Bigr) \\ \oplus \NS^*(\underline{X}-|\partial X|)\otimes_{\Z}\Q}&\textrm{for }i=1\\
			\Alb^{\log}(X)&\textrm{for }i=0\\
			0&\textrm{otherwise}.
		\end{cases}
		\]
where $\NS^*(\underline{X}- |\partial X|)$ is the dual torus to the N\'eron-Severi group of $\underline{X}-|\partial X|$, and for $V$ a $k$-vector space, $V^\vee$ denotes its linear dual.
\end{thm}

As an application, we can identify the compact objects of $\logDMone(k,\Q)$: our result generalizes \cite{BVKahn} on Deligne $1$-motives to the case of \'etale Laumon $1$-motives:
\begin{thm}[Theorem \ref{thm;Laumon-are-motivic}]\label{thm:Laumon_intro}
	Let $\cM_{1,\et}^{a}$ be the category of \'etale Laumon $1$-motives (see \ref{def:eff-1-mot} and \ref{def:1-mot}). Then the functor $\omega^{\lDM}_{\leq 1}$ preserves compact objects and it induces an equivalence\[
	\cD^b(\cM_{1,\et}^{a}\otimes \Q)\xrightarrow{\sim} \mathbf{log}{\mathcal{DM}}^{\rm eff}_{\leq 1,\rm gm}(k,\Q)
	\] 
	where the right-hand side is the $\infty$-subcategory of compact objects of $\mathbf{log}{\mathcal{DM}}^{\rm eff}_{\leq 1}(k,\Q)$
	\end{thm}
Finally, we remark that it is essential to use the \'etale topology with rational coefficients from the beginning in Theorem \ref{thm:Alb_intro} for the following reason: the map $\Z_{\tr}(X)\to \bAlb_{\kX}$ is surjective in the \'etale topology (and not in the Nisnevich topology), which forces us to work in the category of \'etale reciprocity sheaves (a similar issue occurs in Ayoub and Barbieri-Viale, where they are forced to work in $\HI_{\et}$). On the other hand, it is not known in general that the category of \'etale reciprocity sheaves is closed under quotients, and the reason is the following: if $X_x^h$ is an Henselian local scheme, then its generic point is again Henselian local. On the other hand, if $X_x^{sh}$ is a Henselian, its generic point is not going to be be strictly Henselian anymore. This forces us to use rational coefficients to exploit Voevodsky's theorem and conclude that  $\RSC_{\et}(k,\Q)$ is an abelian category.
\subsection{Outline}
We now give a brief outline of the contents of the various sections of this paper.

In Section \ref{sec:recall}, we give a quick reminder of the  theory of reciprocity sheaves  and modulus sheaves with transfers as developed in \cite{MotModulusI}, \cite{MotModulusII}, \cite{KSY-RecII} and \cite{SaitoPurity}. We also give a quick recollection of the material in \cite{BPO} and \cite{BindaMerici} on logarithmic motives and we prove some basic result with rational coefficients.
In Section \ref{sec:Albanese}, we construct the Albanese map with modulus as a universal object in the category of reciprocity sheaves and compare it with the usual Albanese map and the Albanese group scheme of \cite{BK}.
 In Section \ref{sec:1-mot-sheaves-with-modulus} we introduce the categories of $n$-reciprocity sheaves by a suitable modification of the techniques of \cite{Ayoub-BV}. We prove that the category of $0$-reciprocity sheaves agrees with the category of $0$-motivic sheaves of \cite{Ayoub-BV}. We show the existence of a pro-left adjoint $\Alb^{\log}$ of the fully faithful embedding of $1$-motivic sheaves into the category of dividing \'etale sheaves with log transfers, or ``logarithmic sheaves'' for short (again with rational coefficients).

Section \ref{sec:derived} is the most technical one: we prove that the category of logarithmic sheaves admits enough $BC$-admissible objects (in the sense of the Appendix \ref{sec:appendix}) and deduce the existence of a pro-left derived functor $L\Alb^{\log}$. The techniques in particular are fairly different from the corresponding one in \cite{Ayoub-BV}, although the general structure of the proof is similar. Next, we prove that the functor $L\Alb^{\log}$ factors through $\logDM$ and that on $\mathcal{DM}_{\leq 1}^{\rm eff}$ it agrees with the motivic Albanese map of \cite{Ayoub-BV} (note that this result is optimal, see Remark \ref{rmk:main_theorem_optimal}). After that,  we compute $L\Alb^{\log}(\G_a)$ thanks to an explicit resolution (the Breen-Deligne resolution of the algebraic group $\G_a$), deducing the full faithfulness of the inclusion $\cD(\RSC_{\leq 1, \et})(k,\Q)\to \logDM(k,\Q)$.

In Section \ref{sec:computations}, we perform several computations, and we identify precisely $L\Alb(X)$ for $X\in \SmlSm(k)$, proving Theorem \ref{thm;compute-LAlbX-intro}. We also pose some questions about the behaviour of the higher pro-Albanese sheaves in some special geometric situations.
In Section \ref{sec:Laumon} we consider the category of \'etale Laumon $1$-motives, and prove that they are motivic in the sense that their bounded derived category agrees with the category of compact objects in the category of logarithmic 1-motives, as explained in Theorem \ref{thm:Laumon_intro}. 

Finally, in the Appendix \ref{sec:appendix} we introduce the notion of $BC$-admissible objects in a stable $\infty$-category and generalise the notion of a derived functor to pro-adjuntions between derived stable $\infty$-categories which are not in general induced by Quillen adjunctions.


\subsection*{Acknowledgements}
The authors are deeply grateful to Joseph Ayoub for the many insights and suggestions. F.B.~wishes to thank Luca Barbieri-Viale for several useful conversations. This project started while the first and the last named author where visiting the Institut Mittag-Leffler in Djursholm, Sweden, during the special trimester program ``Algebro-Geometric and Homotopical Methods'', in the Spring 2017. Another part of this project was written while the first author was a Postdoc at the University of Regensburg, Germany, supported by the SFB 1085 ``Higher Invariants''. The authors are grateful to all the institutions for their support during various stages of this work. {Finally, the authors would like to thank the referee for a thorough reading of the paper and providing helpful comments which filled some gaps in the arguments and led to an improved presentation.}

\begin{warning}\label{warning}
	\begin{enumerate}
		\item[(i)]In the whole paper, we will commit the following abuse of notation: for $G$ a smooth commutative group scheme, we still write $G$ for the associated \'etale sheaf with transfers. For a ring $\Lambda$, we often write $G\in \Shv_{\et}(X,\Lambda)$ for the sheaf $G\otimes_\Z\Lambda$. Notice that if $\Lambda$ is torsion free, the functor $\_\otimes_\Z\Lambda$ is exact, hence if \[\mathbf{1}\to H\to G\to Q\to \mathbf{1}\] is an exact sequence of commutative algebraic groups, then \[0\to H\otimes_\Z \Lambda\to G\otimes_\Z \Lambda\to Q\otimes_\Z \Lambda\to 0\] is an exact sequence of \'etale sheaves with transfers. 
		\item[(ii)] We will use the following two notations:\begin{itemize}
			\item[$\pro$:] the pro-category of an abelian category, which is itself abelian
			\item[$\Pro$:] the pro-$\infty$-category of an $\infty$-category.
			\end{itemize}
		\end{enumerate}
\end{warning}

\section{Reciprocity sheaves and logarithmic motives with rational coefficients}\label{sec:recall}
We work over a fixed ground field $k$, which is assumed to be perfect. Let $\Lambda$ be a (commutative) ring of coefficients.
 In this section, we recall the main results on reciprocity sheaves and logarithmic motives and we state some general results on the categories with rational coefficients. 

Let $\Sm(k)$ be the category of separated smooth schemes of finite type over $k$, and let $\Cor(k)$ be the additive category of finite correspondences. It has the same objects as $\Sm(k)$, and for $X, Y \in \Sm(k)$, the hom group $\Cor(X,Y)$ is the free abelian group on the set of integral closed subschemes of $X\times Y$ which are finite and surjective over a connected component of $X$ (see \cite[Def. 1.1]{MVW}). We denote by $\PSh^{\tr}(k,\Lambda)$ the category of additive presheaves of $\Lambda$-modules on $\Cor(k)$, whose objects are called \emph{presheaves with transfers}. For $X\in \Sm(k)$, we let $\Lambda_{\tr}(X)=\Cor(-,X)\tensor_\Z \Lambda$ be the representable object. For $\tau$ the Nisnevich or the \'etale topology, we let $\Shv^{\tr}_{\tau}(k,\Lambda)\subseteq \PSh^{\tr}(k,\Lambda)$ be the category of $\tau$-sheaves with transfers and we let\[
a_{\tau}^V\colon \PSh^{\tr}(k,\Lambda)\to \Shv^{\tr}_{\tau}(k,\Lambda)
\]
be Voevodsky's $\tau$-sheafification functor: it is induced by the classical sheafification functor defined on the category of presheaves of $\Lambda$-modules without transfers.
Let $\HI\subseteq \PSh^{\tr}(k,\Lambda)$ be the category of $\A^1$-invariant presheaves, i.e. objects $F$ such that the projection $X\times \A^1\to X$ induces an isomorphism $F(X\times \mathbb{A}^1)\xrightarrow{\simeq} F(X)$ for every $X\in \Sm(k)$. Set  $\HI_\tau=\HI\cap \Shv^{\tr}_{\tau}(k,\Lambda)\subseteq \Shv^{\tr}_{\tau}(k,\Lambda)$.

We recall the following result:
\begin{prop}[\cite{MVW}, Cor.~14.22, Prop.~14.23]\label{prop:et-Nis-Voe} Let $\Lambda$ be a $\Q$-algebra. Then for every $F\in \PSh^{\tr}(k,\Lambda)$ we have $a_{\rm Nis} F = a_{\rm \acute{e}t} F$. Moreover, for all smooth $X$ and $n>0$ we have \[H^{n}_{\rm Nis}(X, F) = H^{n}_{\rm \acute{e}t}(X, F).\]
\end{prop}

\subsection{The \texorpdfstring{$\infty$}{infinity}-category of logarithmic motives}
We recall the construction of the $\infty$-category of logarithmic motives of \cite{BPO} and some properties. The standard reference for log schemes is \cite{ogu}. 
We denote by ${\mathbf{lSm}}(k)$ the category of fine and saturated (fs for short) log smooth log schemes over $\Spec(k)$,  considered as a log scheme with trivial log structure. 

\subsubsection{Log geometry}\label{ssec:loggeom} For $X\in {\mathbf{lSm}}(k)$, we write $\ul{X}\in \mathbf{Sch}(k)$ for the underlying $k$-scheme. We also write $\partial X$ for the (closed) subset of $\ul{X}$ where the log structure of $X$ is not trivial. Let $\SmlSm(k)$ be the full subcategory of $\lSm(k)$ having for objects $X\in \lSm(k)$ such that $\ul{X}$ is smooth over $k$. By e.g.\ \cite[A.5.10]{BPO}, if $X\in \SmlSm(k)$, then $\partial X$ is a strict normal crossing divisor  on $\ul{X}$ and the log scheme $X$  is isomorphic to $(\ul{X}, \partial {X})$, i.e. to the compactifying log structure associated to the open embedding $(\ul{X}\setminus \partial X) \to \ul{X}$. If $X, Y\in \SmlSm(k)$, we will write $X\times Y$ for the fiber product of $X$ and $Y$ over $k$ computed in the category of fine and saturated log schemes: it exists by \cite[Cor. III.2.1.6]{ogu} and it is again an object of $\lSm(k)$ using \cite[Cor. IV.3.1.11]{ogu}. Since $k$ has trivial log structure, the underlying scheme $\ul{X\times Y}$ agrees with $\ul{X}\times_k \ul{Y}$ and the support of $\partial (X\times Y)$ is $|\partial X|\times \ul{Y} \cup \ul{X}\times|\partial Y|$, in particular $X\times Y \in \SmlSm(k)$. See \cite[\S III.2.1]{ogu} for more details.

A morphism $f\colon X\to Y$ of fs log schemes is called \emph{strict} if the log structure on $X$ is the pullback log structure from $Y$. Geometrically, if both $X$ and $Y$ are objects in $\SmlSm(k)$, this amounts to require that there is an equality $\partial X = f^*(\partial Y)$ as reduced normal crossing divisors on $X$. For $\tau$ a Grothendieck topology on $\Sch(k)$, the $\emph{strict}$ topology $s\tau$ on $\SmlSm(k)$ is the Grothendieck topology generated by covers $\{e_i\colon X_i\to X\}$ such that $\underline{e_i}\colon \underline{X_i}\to \underline{X}$ is a $\tau$-cover and each $e_i$ is strict. Recall from \cite[3.1.4]{BPO} that a cartesian square of fs log schemes
\[
\begin{tikzcd}
    Y' \ar[r, "g'"]\ar[d, "f'"] & Y \ar[d, "f"]\\
    X' \ar[r, "g"] & X
\end{tikzcd}
\]
is a \emph{dividing distinguished square} (or \emph{elementary dividing square}) if $Y'=X'=\emptyset$ and $f$ is a log modifications, in the sense of F.\ Kato \cite{FKato} (see \cite[A.11.9]{BPO} for more details on log modifications). The collection of dividing distinguished squares forms a cd structure on $\SmlSm(k)$, called the \emph{dividing cd structure}. For $\tau$ a Grothendieck topology on $\Sch(k)$, the $\emph{dividing}$ topology $d\tau$ on $\SmlSm(k)$ is the topology on $\SmlSm(k)$ generated by the strict topology $s\tau$ and the dividing cd structure.

From now until the end of the section, we will consider $\tau\in \{\Nis,\et\}$.

\subsubsection{Correspondences and transfers}Following \cite{BPO}, we denote by $\lCor(k)$ the category of finite log correspondences over $k$. It is a variant of the Suslin--Voevodsky category of finite correspondences $\Cor(k)$. It has the same objects as $\SmlSm(k)$\footnote{Notice that this notation conflicts with the notation of \cite{BPO} where the objects were the same as $\lSm(k)$, although the categories of sheaves are the same in light of \cite[Lemma 4.7.2]{BPO}}, and morphisms are given by the free abelian subgroup
\[ \lCor(X,Y) \subseteq  \Cor(X- \partial X, Y- \partial Y)\]
generated by elementary correspondences  $V^o\subset (X- \partial X) \times (Y- \partial Y)$ such that the closure $V\subset \ul{X}\times \ul{Y}$ is finite and surjective over (a component of) $\ul{X}$ and such that there exists a morphism of log schemes $V^N \to Y$, where $V^N$ is the fs log scheme whose underlying scheme is the normalization of $V$ and whose log structure is given by the inverse image log structure along the composition $\underline{V^N} \to \ul{X}\times \ul{Y} \to \ul{X}$. See \cite[2.1]{BPO} for more details, and for the proof that this definition gives indeed a category. 

Additive presheaves (of $\Lambda$-modules) on the category $\lCor(k)$ will be called \emph{presheaves (of $\Lambda$-modules) with log transfers}. Write $\PShltr(k, \Lambda)$ for the resulting category. As usual, for $X\in \lCor(k)$ we denote by $\Lambda_{\ltr}(X)$ the representable presheaf $\lCor(-, X)\tensor_\Z \Lambda$. 
As in \cite{BindaMerici}, we let $\widetilde{\SmlSm}(k)$ be the category of fs log smooth $k$-schemes $X$ which are essentially smooth over $k$, i.e. $X$ is a limit $\lim_{i \in I} X_i$ over a filtered set $I$, 
where $X_i \in \SmlSm(k)$ and all transition maps are strict \'etale (i.e. they are strict maps of log schemes such that the underlying maps $f_{ij}\colon \ul{X}_i\to \ul{X}_j$ are {affine and} \'etale). For $X\in \SmlSm(k)$ and $x\in \underline{X}$, we put
\begin{equation}\label{eq;henselization}
X_x^h = (\underline{X}_x^h,\partial X_x^h) \in \widetilde{\SmlSm}(k)
\end{equation}
where $(\underline{X})_{x}^h$ denotes the henselization of $X$ at $x$ and $(\partial X)_{x}^h$ denotes the pullback of $\partial X$ along the henselization map. For $F\in \PShltr(k,\Lambda)$ and $X \in \widetilde{\SmlSm}(k)$ such that $X=\lim_{i \in I} X_i$ for $X_i\in \SmlSm(k)$ we put as usual 
$F(X) := \colim_{i \in I} F(X_i)$.

We denote by $\mathbf{Shv}_{d\tau}^{\rm ltr}(k, \Lambda)\subset \PShltr(k, \Lambda)$ the subcategory of $d\tau$-sheaves. By \cite[Prop. 4.5.4]{BPO} and \cite[Thm. 4.5.7]{BPO}, the inclusion $\mathbf{Shv}_{d\tau}^{\rm ltr}(k, \Lambda)\subset \PShltr(k, \Lambda)$ admits an exact left adjoint $a_{d\tau}$ (see \cite[Prop. 4.2.10]{BPO}), and  the category $\mathbf{Shv}_{d\tau}^{\rm ltr}(k, \Lambda)$ is a Grothendieck abelian category (\cite[Prop. 4.2.12]{BPO}). For $\tau\in \{\Nis,\et\}$, \cite[
Theorem 5.1.8]{BPO} implies that for $F\in \PShltr(k, \Lambda)$ and $X\in \lSm(k)$, 
\begin{equation}\label{eq;colimdiv}
H^i_{d\tau}(X, a_{d\tau} F) = \colim_{Y \in X_{\rm div}^{\Sm}} H^i_{s\tau}(Y, a_{s\tau}F).
\end{equation}
where $X_{\rm div}^{\Sm}$ is the filtered category of log modifications $Y\to X$ such that $Y\in \SmlSm(k)$. 
The following statement can be shown by imitating the proof of \cite[Prop.~14.23]{MVW} using \eqref{eq;colimdiv}.
\begin{prop}\label{prop:rat-coeff-nisequaletale} Let $\Lambda$ be a $\Q$-algebra and let $F$ be an object of $\PSh^{\ltr}(k, \Lambda)$. Then  there is a natural isomorphism \[
	H^{n}_{\dNis}(X, a_{\dNis}F) = H^{n}_{\mathrm{d\acute{e}t}}(X, a_{\mathrm{d\acute{e}t}}F),\]
	for all $X\in \SmlSm(k)$ and $n\geq 0$.
\end{prop}
Finally the monoidal structure of  $\SmlSm(k)$ induces a monoidal structure on $\Shv^{\rm ltr}_{d\tau}(k,\Lambda)$, and recall from  \cite[(4.3.4)]{BPO} that the functor $\omega^{\log}:X\mapsto X^{\circ}$ induces an adjunction
\begin{equation}\label{omegaadjunction}
	\begin{tikzcd}
		\Shv^{\rm ltr}_{d\tau}(k,\Lambda)\arrow[rr,shift left=1.5ex,"\omega_\sharp^{\log} "]&& \Shv^{\rm tr}_{\tau}(k,\Lambda)\arrow[ll,"\omega_{\log}^*"]
	\end{tikzcd}
\end{equation} 
where for $Y\in \Cor(k)$, $\omega_{\sharp}^{\log} F(Y) = F(Y,\textrm{triv})$ and for $X\in \lCor(k)$, $\omega^*_{\log}F(X)=F(\underline{X}-|\partial X|)$. Moreover, since $\omega^{\log}$ is monoidal by construction (see \ref{ssec:loggeom}), $\omega_\sharp^{\log}$ is monoidal.
We will need later the following immediate result. 
\begin{prop}\label{prop;internal-hom-omega}
	For all $A\in \Shv_{\tau}^{\tr}$ and $B\in \Shv_{d\tau}^{\ltr}$, we have that
\[	\underline{\Hom}_{\Shv_{d\tau}^{\ltr}(k, \Lambda)}(B, \omega^*_{\log}A)\cong \omega^*_{\log}\underline{\Hom}_{\Shv_{\tau}^{\tr}(k, \Lambda)}(\omega_\sharp^{\log} B,A).
	\]	
\end{prop}
\subsubsection{Log motives}\label{ssec;logDM} In light of Proposition \ref{prop:rat-coeff-nisequaletale}, from now until the end of the section, we will consider one of the following situations:
 \begin{itemize}
 	\item $\tau$ is the Nisnevich topology 
 	\item $\tau$ is the \'etale topology and $\Lambda$ is a $\Q$-algebra.
\end{itemize}

Let $\cD(\mathbf{Shv}_{d\tau}^{\rm ltr}(k, \Lambda))$ be the derived stable $\infty$-category of the Grothendieck abelian category $\mathbf{Shv}_{d\tau}^{\rm ltr}(k, \Lambda)$ as in \cite[Section 1.3.5]{HA}: it is equivalent to the underlying $\infty$-category of the model category $\Cpx(\PSh^{\rm ltr}(k, \Lambda))$ with the $d\tau$-local model structure used in \cite{BPO} and \cite{BindaMerici}. 

The adjunction $(\omega^{\log}_\sharp, \omega_{\log}^*)$ of \eqref{omegaadjunction} 
induces the following adjunction of $\infty$-categories of sheaves (see \cite[4.3.4]{BPO}):
\begin{equation}\label{eq:adjunctionomega-derived1} \begin{tikzcd}
		L\omega_\sharp^{\log}\colon \cD(\Shv^{\rm ltr}_{d\tau}(k,\Lambda))\arrow[r,shift left=.5ex ]\arrow[r, description,leftarrow, shift right=.5ex ] & \cD(\Shv^{\rm tr}_{\tau}(k,\Lambda)): R\omega^*_{\log}.
	\end{tikzcd}
\end{equation}
Finally (see \cite[Section 5.2]{BPO}), let $\bcube:=(\P^1,\infty)$. Notice that $\omega^{\log}(\bcube)=\A^1$.
\begin{defn} The stable $\infty$-category $\mathbf{log}\mathcal{DM}^{\textrm{eff}}(k,\Lambda)$ is the localization of the stable $\infty$-category $\cD(\mathbf{Shv}_{\tau}^{\rm ltr}(k, \Lambda))$ with respect to the class of maps
\[
(a_{d\tau}\Lambda(\bcube\times X))[n]\to (a_{d\tau}\Lambda(X))[n]
\]
for all $X\in \lSm(k)$ and $n\in \Z$. 
We let\[
L_{(d\tau,\bcube)}\colon \cD(\mathbf{Shv}_{d\tau}^{\rm ltr}(k, \Lambda))\to \mathbf{log}\mathcal{DM}^{\textrm{eff}}(k,\Lambda)
\]
be the localization functor. For $X\in \SmlSm(k)$, we will let $M(X)=L_{d\tau,\bcube}(\Lambda_{\ltr}(X))$.
\end{defn}
The interested reader can verify that this is equivalent to the underlying $\infty$-category of the model category $\Cpx(\PSh^{\rm ltr}(k, \Lambda))$ with the $(\bcube,d\tau)$-local model structure of \cite[Def. 5.2.1]{BPO} and \cite[Def. 2.9]{BindaMerici}. The derived (triangulated) category of effective log motives $\lDM(k, \Lambda)$ is by definition the homotopy category of $\logDM(k,\Lambda)$.

We recall the following result, which follows naturally from \cite[Thm. 5.7]{BindaMerici}:
\begin{thm}\label{thm:t-structure}
The standard $t$-structure of $\cD(\Shv^{\rm ltr}(k, \Lambda)$ induces an accessible $t$-structure on $\logDM(k,\Lambda)$ compatible with filtered colimits in the sense of \cite[Def. 1.3.5.20]{HA}, called the \emph{homotopy $t$-structure.}
\end{thm}

We denote by $\mathbf{logCI}_{d\tau}$ its heart\footnote{In \cite{BindaMerici}, it is denoted by $\CIltr$}, which is then identified with the category of strictly $\bcube$-invariant $d\tau$-sheaves and it is a Grothendieck abelian category. The inclusion
\[
	i^{\rm ltr}\colon \mathbf{logCI}_{d\tau}\hookrightarrow \Shv^{\rm ltr}_{d\tau}(k, \Lambda)
\]
admits both a left adjoint $h_0^{\ltr}\colon F\mapsto \pi_0(L_{(d\tau,\bcube)}(F[0]))$ and a right adjoint $h^0_{\ltr}$  (see \cite[Proposition 5.8]{BindaMerici}), in particular it is exact and $\logCI_{d\tau}$ inherits a monoidal structure from $\Shv^{\rm ltr}_{d\tau}(k, \Lambda)$ given by:
\begin{equation}\label{eq:logCI-monoidal}
	F\otimes_{\logCI} G :=h_0^{\ltr}\big(i^{\ltr}(F)\otimes_{\Shv_{d\tau}^{\ltr}}i^{\ltr}(G)\big).
\end{equation}

\subsection{Comparison with Voevodsky motives}
In this subsection, we assume that $k$ admits resolution of singularities (see e.g. \cite[Def. 7.6.3]{BPO} for a precise definition). This assumption is always satisfied if $\ch(k)=0$.

\begin{defn}\label{defn;cartier-comp}
Let $X\in \Sm(k)$. A \emph{smooth Cartier compactification} or simply a \emph{Cartier compactification} of $X$ is a pair $(\overline{X},D)$ where $\overline{X}\in \Sm(k)$ is proper and $D\subseteq \overline{X}$ is an effective Cartier divisor with simple normal crossing such that $\overline{X}-|D| \cong X$.
\end{defn}
Note that if $k$ admits resolution of singularities, every $X\in \Sm(k)$ admits a (smooth) Cartier compactification. This definition is slightly different from the one used in \cite{MotModulusI} and \cite{MotModulusII}, where the total space $\overline{X}$ is not required to be smooth over $k$, but simply normal. Under our assumption on $k$, this difference is irrelevant.

By \cite[Prop. 8.2.12]{BPO}, the adjunction of \eqref{eq:adjunctionomega-derived1} descends to an adjunction:
\begin{equation}\label{eq;adjunctionomega-motivic}
\begin{tikzcd}
	L^{\bcube}\omega_\sharp^{\log}\colon \logDM(k,\Lambda) \arrow[r,shift left=.5ex ]\arrow[r, description,leftarrow, shift right=.5ex ] & \mathcal{DM}^{\rm eff}(k,\Lambda) : R^\bcube\omega^*_{\log},
\end{tikzcd}
\end{equation}
where the right-hand side is the $\infty$-category of Voevodsky motives. By \cite[Thm. 8.2.16 and Thm. 8.2.17]{BPO}, the functor $R^\bcube\omega^*_{\log}$ is fully faithful and for $X\in \Sm(k)$ and $(\overline{X},D)$ a Cartier compactification we have a natural equivalence
\[R^\bcube\omega^*_{\log}M(X) \simeq M(\overline{X},\partial X)\] 
with $\partial X$ supported on $|D|$. In particular, the essential image of $R^\bcube\omega^*_{\log}$ is the full subcategory spanned by $M(X)$ with $X\in \SmlSm(k)$ and $\underline{X}$ proper. Finally, by \cite[Prop. 5.12]{BindaMerici} it is $t$-exact with respect to the  homotopy $t$-structure of \ref{thm:t-structure} on $\logDM(k,\Lambda)$ and the Morel-Voevodsky $t$-structure on  $\mathcal{DM}^{\rm eff}(k,\Lambda)$. In particular, we have  a fully faithful functor (still denoted by $\omega^*_{\log}$) $\HI_{\tau}\to\mathbf{logCI}_{d\tau}$ between the hearts that commutes with the inclusions. The following result will be crucial in the proof of Proposition \ref{prop:key-prop-Alb-local-admissible}:

\begin{lemma}\label{lem:uneful-tensor-hi-closed}
	The functor $\omega^*_{\log} \colon \HI_{\tau}\to\mathbf{logCI}_{d\tau}$ admits a right adjoint (in particular it commutes with all colimits) and is monoidal with respect to the structure \eqref{eq:logCI-monoidal}.
	\begin{proof}
 By \cite[Proposition 8.2.12]{BPO}, the functor $R^\bcube\omega_{\log}^*$ has a right adjoint $R^\bcube\omega^{\log}_*$, hence since $R^\bcube\omega_{\log}^*$ is $t$-excact by \cite[Proposition 5.12]{BindaMerici}, the functor induced on the hearts are still adjoints by \cite[Prop 1.3.17-(iii)]{BBD}.
	In particular, the functor $\omega^*_{\log}$ commutes with all colimits, so to conclude it is enough to show that for $X,Y\in \Sm(k)$ we have\[
		\omega^*_{\log}(h_0^{\A^1}(X\times Y)) = \omega^*_{\log}h_0^{\A^1}(X)\otimes_{\mathbf{logCI}} \omega^*_{\log}h_0^{\A^1}(Y),
		\]
		where $h_0^{\A^1}:\Shv_{\tau}^{\tr}\to \HI_{\tau}$ is left adjoint to the inclusion (the $0$-th Suslin homology sheaf).
		Thanks to \cite[Proposition 8.2.4]{BPO}, for any choice of a smooth Cartier compactification $X\subseteq \overline{X}$ and $Y\subseteq \overline{Y}$, 
 by putting as log structure $\partial X$ and $\partial Y$ associated with the simple normal crossing divisor $\overline{X}-X$ and $\overline{Y}-Y$, we have equivalences in $\logDM$:\[
		R^\bcube \omega^*_{\log}M^{\A^1}(X) = M(\overline{X},\partial X)\quad \textrm{and}\quad R^\bcube\omega^*_{\log}M^{\A^1}(Y) = M(\overline{Y},\partial Y).
		\]
		By taking $\pi_0$, \cite[Proposition 5.12]{BindaMerici} implies that\[
		\omega^*h_0^{\A^1}(X)=h_0^{\ltr}(\overline{X},\partial X) \quad \textrm{and}\quad 	\omega^*h_0^{\A^1}(Y)=h_0^{\ltr}(\overline{Y},\partial Y).
		\] 
		Hence, we have that\[
		\omega^*h_0^{\A^1}(X)\otimes_{\mathbf{logCI}} \omega^*h_0^{\A^1}(Y)=h_0^{\ltr}(\overline{X},\partial X)\otimes_{\mathbf{logCI}}h_0^{\ltr}(\overline{Y},\partial Y) = h_0^{\ltr}((\overline{X},\partial X)\times (\overline{Y},\partial Y)).
		\]
		Finally, since the underlying scheme of $(\overline{X},\partial X)\times (\overline{Y},\partial Y)$ is $\overline{X}\times\overline{Y}$, which is proper, and the subscheme where the log structure is trivial is $X\times Y$, we have that the log scheme $(\overline{X},\partial X)\times (\overline{Y},\partial Y)$ is a Cartier compactification of $X\times Y$, hence again by \cite[Proposition 8.2.4]{BPO} we have\[
		h_0^{\ltr}((\overline{X},\partial X)\times (\overline{Y},\partial Y))\cong \omega^*h_0^{\A^1}(X\times Y),
		\]
		which concludes the proof.
	\end{proof}
\end{lemma}

\subsection{The abelian category of reciprocity sheaves}\label{reciprocitysheaves}
We recall the construction of the abelian category of reciprocity sheaves via modulus sheaves of \cite{MotModulusI} and \cite{MotModulusII} as done in \cite{KSY-RecII} and some properties. 

A pair $\kX = (X,D)$ where $X$ is a proper scheme of finite type over $k$ and $D$ is an effective Cartier divisor on $X$ is called a \emph{proper modulus pair} if $X - |D| \in \Sm(k)$.  
Let $\kX=(X,D_X)$, $\sY=(Y,D_Y)$ be proper modulus pairs and $\Gamma \in \Cor(X-|D_X|,Y-|D_Y|)$ be a prime correspondence. Let $\overline{\Gamma} \subseteq X \times Y$ be the closure of $\Gamma$, and let $\overline{\Gamma}^N\to X\times Y$ be the normalization. We say that $\Gamma$ is \emph{admissible} if $(D_X)_{\overline{\Gamma}^N}\geq (D_Y)_{\overline{\Gamma}^N}$ as Weil divisors, where $(E)_{\overline{\Gamma}^N}$ denotes the pullback of the  divisor $E$ to the normalization $\overline{\Gamma}^N$. By \cite[Proposition 1.2.7]{MotModulusI}, proper modulus pairs and admissible correspondences define an additive category, denoted $\MCor(k)$. {For $\kX=(X,D)$ and $n\geq 0$, we let $\kX^{(n)}:=(X,nD)$.}

We denote by $\MPST(k,\Lambda)$ or simply $\MPST$ the category of additive presheaves of $\Lambda$-modules on $\MCor(k)$, whose objects are called \emph{proper modulus presheaves with trasnfers}. For $\kX \in \MCor(k)$, we let $\Lambda_{tr}(\kX)=\MCor(-,\kX)\tensor_\Z \Lambda \in \MPST$ be the representable object.

	\begin{defn}\label{defn;comp-modulus}
		For $X\in \Sm(k)$, we let $\Comp(X)$ be the cofiltered category given by modulus pairs $(\overline{X},D)$ given by Cartier compactifications of $X$ (see \cite[Lemma 1.8.2]{MotModulusI} and Definition \ref{defn;cartier-comp}).
	\end{defn}

There is a functor:\[
\omega: \MCor(k)\to \Cor(k)\quad (X,D)\mapsto X-|D|,
\]
which induces adjoint functors (cf. \cite[Pr. 2.2.1]{MotModulusI}):\[\begin{tikzcd}
	\omega_!: \MPST(k,\Lambda)\ar[r,shift left = .5ex]\arrow[r,shift left=.5ex ]\arrow[r, description,leftarrow, shift right=.5ex ] & \PSh^{\tr}(k,\Lambda): \omega^*
\end{tikzcd}
\]
where $\omega^*$ is fully faithful. For $\kX=(X,D)\in \MCor(k)$, we have\[
\omega^*F(\kX) = F(\omega(\kX)) = F(X-|D|).
\]
The functor $\omega_!$ is given by left Kan extension, so that  for $X\in \Sm(k)$ and any choice of $\kX\in \Comp(X)$, we have 
\begin{equation}\label{eq;omega-shriek}
\omega_!F (X) = \colim_{\kY\in \Comp(X)} F(\kY)\xrightarrow{\simeq} \colim_{n} F(\kX^{(n)}).
\end{equation}
where the displayed isomorphism follows from \cite[Lemma 1.27 (1)]{SaitoPurity} (with $X_{\infty}^+=0$), which implies that we have an isomorphism in $\pro\MPST$: \[``\lim_{n}"\Z_{\tr}(\kX^{(n)})\cong ``\lim"_{\kY\in \mathbf{Comp}(X)}\Z_{\tr}(\kY).\]
As in the logarithmic case, let $\bcube:=(\P^1,\infty)\in \MCor(k)$ and for any $\kX=(X,D)\in \MCor(k)$ let (see \cite{MotModulusI})\[
\kX\otimes \bcube := (X\times \P^1,X\times \infty + D\times \P^1).
\]
We say $F \in \MPST$ is $\bcube$-invariant if for any $\kX \in \MCor(k)$, the projection $p:\kX\otimes \bcube\to \kX$ induces an isomorphism \[p^* : F(\kX) \to F(\kX \otimes \bcube).\]
We let $\CI$ be the full subcategory of $\MPST$ consisting of all $\bcube$-invariant objects. By \cite[Lemma 2.1.2]{KSY-RecII}, it is a Serre subcategory of $\MPST$ and that the inclusion functor $i^\bcube : \CI \to \MPST$ has a left adjoint $h_0^\bcube$ and a right adjoint $h^0_\bcube$ given for $F \in \MPST$ and $\kX \in \MCor(k)$ by
\begin{align*} 
	&h_0^\bcube(F)(\kX)
	=\Coker(i_0^* - i_1^* : F(\kX \otimes \bcube) \to F(\kX)),
	\\
	\notag
	&h^0_\bcube(F)(\kX)=\Hom(h_0^\bcube(\kX), F),
\end{align*}
where for $a\in k$ the section $i_a: \kX\to \kX\otimes \bcube$ is induced by the map $k[t]\to k[t]/(t-a)\cong k$.
We write $\RSC(k,\Lambda)\subseteq \PSh^{\tr}(k,\Lambda)$ for the essential image of $\CI$ under $\omega_!$. It is an abelian subcategory of $\PSh^{\tr}(k,\Lambda)$. 
\begin{remark}\label{rmk;reciprocity-only-(n)}
By \eqref{eq;omega-shriek}, for $F\in \PSh^{\tr}(k,\Lambda)$ the following conditions are equivalent:
\begin{enumerate}
	\item[(i)] $F\in \RSC(k,\Lambda)$,
	\item[(ii)] for every $X\in \Sm(k)$ and every section $a:\Z_{\tr}(X)\to F$, there exists $\kY\in \Comp(X)$ such that $a$ factors through $\Z_{\tr}(X)\to \omega_!h_0^\bcube(\kY)$,
	\item[(iii)] for every $X\in \Sm(k)$ and every section $a:\Z_{\tr}(X)\to F$, for any choice of $\kX\in \Comp(X)$ there exists $n$ such that $a$ factors through $\Z_{\tr}(X)\to \omega_!h_0^\bcube(\kX^{(n)})$.
\end{enumerate}
\end{remark}
For $\tau$ the Nisnevich or the \'etale topology, we let $\RSC_{\tau}(k,\Lambda):=\RSC(k,\Lambda)\cap \Shv^{\tr}_{\tau}(k,\Lambda)$.
The objects of $\RSC(k,\Lambda)$ (resp. $\RSC_{\tau}(k,\Lambda)$) are called reciprocity presheaves (resp. $\tau$-reciprocity sheaves) of $\Lambda$-modules. 
By \cite[Thm. 0.1]{SaitoPurity}, the Nisnevich sheafification restricts to a functor\[
a_{\Nis}^V\colon \RSC(k,\Lambda)\to \RSC_{\Nis}(k,\Lambda),
\]
which makes $\RSC_{\Nis}(k,\Lambda)$ a Grothendieck abelian category (see \cite[Corollary 2.4.2]{KSY-RecII}). Notice in particular that $\RSC_{\Nis}(k,\Lambda)$ is closed under sub-objects and quotients in $\Shv^{\tr}_{\Nis}(k,\Lambda)$, and that the inclusion functor $i\colon \RSC_{\Nis}(k,\Lambda) \to \Shv^{\tr}_{\Nis}(k,\Lambda)$ is exact. As in \cite[Theorem 2.4.3 (1)]{KSY-RecII}, we denote by $\rho$ the right adjoint to the inclusion $i$. 

By \cite[Theorem 0.2]{SaitoPurity}, each $F\in\RSC_{\Nis}(k,\Lambda)$ satisfies \emph{global injectivity}, i.e. for every $X\in \Sm$ connected with generic point $\eta_X$, the restriction map gives an injective map:\[
F(X)\hookrightarrow F(\eta_X).
\]
By Proposition \ref{prop:et-Nis-Voe}, if $\Lambda$ is a $\Q$-algebra the \'etale sheafification coincides with the Nisnevich sheafification, hence it restricts to a functor
\begin{equation}\label{eq;et-Nis-RSC}
	\begin{tikzcd}
		\RSC(k,\Lambda)\ar[r,"a_{\Nis}^V"]\ar[rr,bend right=10,swap,"a_{\et}^V"] &\RSC_{\Nis}(k,\Lambda) &\RSC_{\et}(k,\Lambda)\ar[l,swap,"\simeq"],
	\end{tikzcd}	
\end{equation}
in particular, if $\Lambda$ is a $\Q$-algebra, $\RSC_{\et}(k,\Lambda)$ is a Grothendieck abelian category and every $F\in \RSC_{\et}(k,\Lambda)$ satisfies global injectivity.

We have two important examples of reciprocity sheaves:
\begin{enumerate}
\item Let $\sG^*$ be the category of smooth commutative $k$-group schemes {(i.e. $G\in \sG^*$ is a group scheme such that the connected component of the identity $G^0$ is a smooth commutative algebraic group and $\pi_0(G)$ is finitely generated, see \cite[1.3]{Ayoub-BV} for an analogous definition on semi-abelian group schemes)}. It is classical that the corresponding \'etale sheaf has a unique structure of sheaf with transfers (see \cite[Lemma 3.2]{SS}) and by \cite[Theorem 4.4]{KSY} it has reciprocity in the sense of \cite[Definition 2.1.3]{KSY}.
This  defines a functor $\sG^*\to \RSC_{\et}(k,\Z)$, generalizing the functor $\sG^*_{\rm sab} \to \HI_{\et}(k,\Z)$ considered in \cite{BVKahn}. 
\item
For $\kC=(C,C_{\infty})\in \MCor(k)$ with $C$ geometrically connected and $\dim(C)=1$, let $\uPic(C,C_{\infty})$ denote the relative Picard group scheme. We would like to underline that $C_{\infty}$ is not supposed to be reduced. By \cite[Thm. 1.1]{RulYama} combined with \cite[Lem 2.2.2]{KSY-RecII}, we have that
\begin{equation}\label{eq;pic}
	\omega_!h_0^\bcube(\Z_{tr}(\kC)) = \uPic(C,C_\infty).
\end{equation}
\end{enumerate}
We end this subsection recalling the following result:
\begin{prop}\label{prop;forget-transfer-fully-faithful} If $\Lambda$ is a $\Q$-algebra, the ``forgetting transfers" functor $\RSC_{\et}(k,\Lambda)\to \Sh_{\et}(k,\Lambda)$ is  fully faithful and exact. 
	\begin{proof}The argument for  homotopy invariant sheaves with transfers given in \cite[3.9]{BVKahn} works here as well, replacing the reference to Voevodsky's purity theorem for homotopy invariant sheaves with the global  injectivity provided by \cite[Theorem 0.2]{SaitoPurity}.\end{proof}\end{prop}

\subsection{Reciprocity sheaves and logarithmic motives}\label{RSC-and-logRec}

In this subsection, we continue to assume that $k$ satisfies resolution of singularities. 
Let $\tau$ and $\Lambda$ be as in \ref{ssec;logDM}.

By \cite{shujilog}, there exists a fully faithful and exact functor
\[\mathcal{L}og: \RSC_{\Nis}(k,\Lambda) \to \mathbf{logCI}_{\dNis}(k,\Lambda)\]
such that $\omega_{\sharp}^{\log}\mathcal{L}og\simeq id$. 
If $\Lambda$ is a $\Q$-algebra, by Proposition \ref{prop:rat-coeff-nisequaletale}, we have similarly a fully faithful and exact functor:
\begin{equation}\label{eq;shuji-functor2}
	\omega^{\CI}_{\log}:\RSC_{\et}(k,\Lambda)\overset{\eqref{eq;et-Nis-RSC}}{\cong }\RSC_{\Nis}(k,\Lambda)
	\overset{\mathcal{L}og}{\rightarrow} \mathbf{logCI}_{\dNis}(k,\Lambda)\cong \mathbf{logCI}_{\textrm{d\'et}}(k,\Lambda)\hookrightarrow \mathbf{Shv}_{\textrm{d\'et}}^{\mathrm{ltr}}(k, \Lambda).
\end{equation}

	\begin{remark}\label{rmk;omegaCI-commutes-with-colimits}
The functor $\Log$ commutes with all colimits: to see this, let $\{F_i\}\subseteq \RSC_{\Nis}(k, \Lambda)$ be an inductive system, then we have a natural map
\begin{equation}\label{eq:colim-log-map}
\colim^{\logCI_{\dNis}}\Log(F_i)\to \Log (\colim^{\RSC_{\Nis}}F_i)
\end{equation}
Since $\omega_{\sharp}^{\log}(k, \Lambda)\colon \logCI_{\dNis}\to \Shv^{\tr}_{\Nis}(k, \Lambda)$ is faithful, hence conservative (this follows from the results of \cite{BindaMerici}, see \cite[Prop. 0.1]{BindaMericiErratum}), it is enough to show that \eqref{eq:colim-log-map} is an isomorphism after applying $\omega_\sharp^{\log}$. We have that 
\begin{align*}
\omega_{\sharp}^{\log}\colim^{\logCI_{\dNis}}\Log(F_i)&\overset{(*1)}{\cong} \colim^{\Shv^{\tr}_{\Nis}}\omega_{\sharp}^{\log}\Log(F_i) \overset{(*2)}{\cong} \colim^{\Shv^{\tr}_{\Nis}}F_i \overset{(*3)}{\cong} \colim^{\RSC_{\Nis}}F_i \\ &\overset{(*4)}{\cong} \omega_{\sharp}^{\log}\Log(\colim^{\RSC_{\Nis}}F_i)
\end{align*}
where $(*1)$ comes from the fact that $\omega_{\sharp}^{\log}$ has a right adjoint, so it preserves colimits, $(*2)$ and $(*4)$ follow from the fact that $\omega_{\sharp}^{\log}\Log=id_{\RSC_{\Nis}}$, $(*3)$ from the fact that the inclusion $\RSC_{\Nis}(k, \Lambda)\hookrightarrow \Shv^{\tr}_{\Nis}(k,\Lambda)$ has a right adjoint, so it preserves colimits. In particular, this implies that the composition \eqref{eq;shuji-functor2} preserves all colimits.
\end{remark}
 \begin{remark}
Let $F\in \mathbf{logCI}_{d\tau}(k, \Lambda)$ such that $\omega_{\sharp}^{\log}F\in \HI_{\tau}(k, \Lambda)$, we have that \[
\omega^*_{\log}\omega_{\sharp}^{\log}F \cong \Log(\omega_{\sharp}^{\log}F) \in \mathbf{logCI}_{\tau}(k, \Lambda),
\] hence we have that the natural map
$\eta\colon F \to \omega^{*}_{\log}\omega_{\sharp}^{\log}F$ is a map in $\mathbf{logCI}_{d\tau}(k, \Lambda)$ such that $\omega_{\sharp}^{\log}(\eta)$ is an isomorphism. Since $\omega_{\sharp}^{\log}$ is faithful and exact, it is conservative. We conclude that:
\begin{equation}\label{eq;only-one-homotopy-invariant}
	F\cong \omega^*_{\log}\omega_{\sharp}^{\log}F\cong \Log(\omega_{\sharp}^{\log}F).
\end{equation}
Notice that \eqref{eq;only-one-homotopy-invariant} strictly depends on the fact that $\omega^*_{\log}\omega_{\sharp}^{\log} F\in \mathbf{logCI}_{d\tau}$, which is not true unless $\omega_{\sharp}^{\log}F\in \HI_{\tau}$.
\end{remark}

\begin{remark}\label{rmk:Ga} 
Recall that the functor $\omega^{\log}$ from \eqref{omegaadjunction} admits a left adjoint $\lambda^{\log}\colon \Sm(k)\to \SmlSm(k)$ such that 
$\lambda^{\log}(Y) = (Y,\triv)$ for $Y\in \Sm(k)$,
which itself has a left adjoint which associates the underlying scheme $\ul{X}$ to $X\in \SmlSm(k)$. Moreover, these functors preserve transfers.
 In particular, we have a left exact functor 
\[\lambda_{\sharp}^{\log}\colon \PSh^{\tr}(k,\Z)\to \PSh^{\ltr}(k,\Z)\]
such that $\lambda^{\log}_\sharp F(X) = F(\ul{X})$ for $X\in \SmlSm(k)$ and it is a left adjoint to $\omega_\sharp^{\log}$.
Moreover, as observed in \cite[4.3, p.64]{BPO}, the functor $\lambda_\sharp^{\log}$ sends $\tau$-sheaves to $d\tau$-sheaves, which implies that we have an adjunction:\begin{equation}\label{omegaadjunction2}
	\begin{tikzcd}
	\Shv^{\rm tr}_{\tau}(k,\Lambda)	\arrow[rr,shift left=1.5ex,"\omega^\sharp_{\log} "]&& \Shv^{\rm ltr}_{d\tau}(k,\Lambda)\arrow[ll,"\omega_\sharp^{\log}"]
	\end{tikzcd}
\end{equation} 
where $\omega^\sharp_{\log}$ is the restriction of $\lambda_\sharp^{\log}$ under
the forgetful functor $\Sh_{\tau}(k,\Lambda)\to \PSh(k,\Lambda)$.
Since the latter functor commutes with limits as a right adjoint, both $\omega^\sharp_{\log}$ and $\omega_\sharp^{\log}$ are exact, so they derive trivially. 
Hence, \eqref{omegaadjunction2} induces the following adjunction of $\infty$-categories of sheaves:
\begin{equation}\label{eq:adjunctionomega-derived1-2} \begin{tikzcd}
		\cD(\omega^\sharp_{\log}) \colon \cD(\Shv^{\rm tr}_{\tau}(k,\Lambda) )\arrow[r,shift left=.5ex ]\arrow[r, description,leftarrow, shift right=.5ex ] & \cD(\Shv^{\rm ltr}_{d\tau}(k,\Lambda)): \cD(\omega_\sharp^{\log}).
	\end{tikzcd}
\end{equation}
We remark that by the construction of $\mathcal{L}og$ in \cite[\S6]{shujilog} and 
\cite[Cor. 6.8(1)]{RulSaito}, we have
\begin{equation}\label{omegaCIGa}
	\Log\G_a(X)=
	 \Gamma(\ul{X},\sO_{X}) = \omega^\sharp_{\log} \G_a(X).
\end{equation} 
\end{remark}

\section{Categories of rational maps and universal problems}\label{sec:Albanese}

\subsection{Commutative groups schemes and torsors under them}\label{ssec;groups}
We recall some well-known facts on commutative group schemes over a perfect field $k$ and we fix some notation. 

Let $\sG^*$ be again the category of smooth commutative $k$-group schemes, locally of finite
type over $k$ {and such that $\pi_0(G)$ is finitely generated} (for short, a commutative $k$-group scheme). Write $\sG$ for the subcategory of smooth commutative algebraic $k$-groups (i.e. objects of $\sG^*$ which are of finite type). Given $G\in \sG^*$, let $G^0$ be the connected component of the identity in $G$.
Recall (see \cite[Definition 1.1.2]{BVKahn} or \cite[Proposition 5.1.4]{DemazureGabriel}) the following definition. 
\begin{defn} A group scheme $L\in \sG^*$ is called \textit{discrete} if $L^0 = \Spec(k)$ and the abelian group $L(\ol{k})$ is finitely generated (equivalently, if $L$ is \'etale over $k$). A discrete $k$-group scheme $L\in \sG^*$ is called a \textit{lattice} if $L(\ol{k})$ is torsion free. 
\end{defn}
As in \cite{BVKahn}, we denote by ${}^t\mathcal{M}_0$  the subcategory of $\sG^*$ consisting of discrete $k$-group schemes. By \cite[Lemma 1.1.3]{BVKahn} it is a Serre subcategory of $\sG^*$, hence it is an Abelian category. We denote by $\cM_0$ the full subcategory of lattices. 
By \cite[Proposition 5.1.8]{DemazureGabriel}, for $G\in \sG^*$, there is an exact sequence
\[ 0\to G^0 \to G\to \pi_0(G)\to0\]
where $\pi_0(G)$ is an \'etale $k$-group, which is universal for homomorphisms from $G$ to discrete groups. The fibers of $G\to \pi_0(G)$ are the irreducible components of $G$.
\begin{defn}Let $G\in \sG$ be a smooth commutative algebraic $k$-group. By \textit{$k$-torsor under $G$ or for $G$} we mean a $k$-scheme $P$, locally of finite type over $k$, equipped with an action $P\times G\to P$ such that the induced morphism $(s,g) \mapsto (s, sg)\colon P\times G\to P\times P$ is an isomorphism. Write $\sP_G$ for the category of \textit{$k$-torsors under $G$}: morphisms between torsors are $G$-equivariant $k$-morphisms.
\end{defn}
\subsubsection{}\label{ssec:torsor-construction}

Given a $k$-torsor $P$ under $G\in \sG$, we can construct a commutative $k$-group scheme $P_G = \coprod_{n\in \Z} P^{\vee n}\in \sG^*$ using the sum of torsors $\vee_G$ (see \cite[III.4.8.b]{MilneEtCoh}) following  \cite[1.2]{Ramachandran}. It fits in a short exact sequence 
\[ 0 \to G \to P_G \xrightarrow{a_P} \underline{\Z}\to 0, \]
presenting $P_G$ as extension of $\underline{\Z}$ by the group $G$. Moreover, we can identify the torsor $P$ with the fiber of the section $1\in\underline{\Z}$ along the map $a_P$, so that we have a natural inclusion $P\hookrightarrow P_G$.

\subsubsection{}Write $\sP$ for the category whose objects are pairs $(P,G)$, where $P\in \sP_G$ for a  $G\in \sG$ a smooth commutative connected algebraic group. A morphism in $\sP$ is the datum of a pair of morphisms $(f^1, f^0)\colon (P,G)\to (P', G')$, where $f^0\colon G\to G'$ is a $k$-morphism of algebraic groups and $f^1\colon P\to P'$ is $f^0$-equivariant. If $X$ is a $k$-scheme, we write $\XP$ for the comma category over $X$: its objects are triples $(u, P, G)$, where $(P,G)\in \sP$ and $u\colon X\to P$ is a $k$-morphisms. Morphisms in $\XP$ are defined in the obvious way. 

\begin{defn}\label{defn:fibtors}A \textit{fibration to torsors} is the datum, for each $X\in \Sm$, of a full category $\sM_X$ of $\XP$, contravariantly functorial in $X$. 
Similarly,  a fibration to torsors \textit{for proper modulus pairs} is the datum, for each $\kX =(\ol{X}, X_\infty)\in \MCor$, of a full subcategory $\sM_{\kX}$ of $\XP$, where $X = \ol{X}\setminus X_{\infty} \in \Sm$, contravariantly functorial in $\kX$ for maps in $\MSm^{\rm fin}$ (see \cite[Definition 1.3.3.(2)]{MotModulusI}). 
	The initial object (if it exists) of $\sM_X$ is called the \textit{$\sM$-Albanese torsor of $X$}. By definition, it is the datum of a smooth commutative connected algebraic group $\Alb^0_{\sM}(X)$, a $k$-torsor $\Alb^1_{\sM}(X)$ under $\Alb^0_{\sM}(X)$ and a $k$-morphism $X\to \Alb^1_{\sM}(X)$ which is universal for maps in $\sM_X$. The algebraic group $\Alb^0_{\sM}(X)$ is called the $\sM$-Albanese variety of $X$. Similarly, if $\sM_{-}$ is a fibration to torsors for proper modulus pairs, the initial object (if it exists) of $\sM_{\kX}$ is called the \textit{$\sM$-Albanese torsor of $\kX$}. The corresponding algebraic group, $\Alb^0_{\sM}(\kX)$ will be called the $\sM$-Albanese variety of $\kX$.
\end{defn}

\begin{example}\label{ex:SAb-case}For any $X$, let $\mathbf{SAb}_X$ be the full subcategory of $\XP$ consisting of maps to torsors $P$ under semi-Abelian varieties. In this case, the existence of an initial object for $\mathbf{SAb}_X$ was proven by Serre \cite{SerreChevalley1} in the case the base field $k$ is algebraically closed. In \cite[Appendix A]{WittenbergAlbaneseTorsors}, a Galois descent argument is used to show that the Albanese variety, the Albanese torsor and the universal map \[X\to \Alb^{(1)}_{\mathbf{SAb}}(X)\] always exist, without any assumption on $k$. 
	
	If $X$ is smooth and proper over $k$, the semi-abelian variety $\Alb^0_{\mathbf{SAb}}(X)$ is in fact an Abelian variety, and coincides with the classical Albanese variety of $X$, dual (as abelian variety) to the Picard variety $\mathrm{Pic}_X^{0, red}$.
\end{example}
\subsection{A universal construction}\label{ssec:AlbaneseOmega} We will discuss a number of situations in which the $\sM$-Albanese torsor of a proper modulus pair $\kX$ exists. In fact, we will consider different fibrations to torsors $\sM_{-}$ for proper modulus pairs, giving sufficient conditions for the initial object to exist. In the end, all the fibrations that we consider will turn out to be equivalent, giving then a \textit{unique notion of Albanese torsor} for a proper modulus pair $\kX$. We recall the following result due to Serre \cite{SerreChevalley1}: 
\begin{thm}\label{thm:SerreChevalley1}
	Let $X\in \Sm(k)$ and $\sC$ be a subcategory of $X\backslash \sG$ satisfying the following conditions:
	\begin{enumerate}
		\item[$\mathrm{(I)}$]
		If $u_i: X\to G_i$ are in $\sC$ for $i=1,2$, then $u_1\times u_2:X \to G_1\times G_2$ is in $\sC$.
		\item[$\mathrm{(II)}$]
		For a homomorphism $f: H\to G$ in $\sG$ such that $\Ker(f)$ is finite and 
		$v:X\to H$ such that $u={f\circ v}: X\to G$ is in $\sC$, $v:X\to H$ is in $\sC$.
	\end{enumerate}
Then:
		\begin{itemize}
			\item[(1)]
			A morphism $u: X \to G$ in $X\backslash \sG$ is universal if and only if it is maximal in the sense of [Def.2, Ser60] and for any maximal morphism $v:X\to H$ in $\sC$, we have
			$\dim H\leq \dim G$. 
			\item[(2)]
			There exists a universal object in $\sC$ if and only if there exists $N>0$ such that
			for any maximal morphism $u:X\to G$ in $\sC$, we have $\dim(G)\leq N$.
		\end{itemize}
\end{thm}

\begin{defn}\label{defn:COmega}Let $\kX=(\ol{X}, X_\infty)$ be a proper geometrically integral modulus pair, and write $U(\kX) = U(\ol{X}, X_\infty)$ for the $k$-vector space $H^0(\ol{X}, \Omega^1_{\ol{X}, cl}(X_\infty))$, where $\Omega^1_{\ol{X}, cl}(X_{\infty}) := \Omega^1_{\ol{X}, cl}\cap \Omega^1_{\ol{X}}(X_{\infty})$ denotes the subsheaf of closed forms.
	Let  $\sM^\Omega_{\kX}$ be the full subcategory of $\XP$ consisting of triples $(u,P,G)$ with the following property: Let $k\subset \ol{k}$ be an algebraic closure of $k$ and let $u_{\ol{k}}\colon {X}_{\ol{k}}\to P_{\ol{k}} \cong G_{\ol{k}}$ be the base change of $u$ to $\ol{k}$. 
	Then $(u,P,G)\in \sM^\Omega_\kX$ if and only if $(u_{\ol{k}})^* \Omega(G_{\ol{k}}) \subseteq U(\ol{X}_{\ol{k}}, X_{\infty, \ol{k}})$, where $\Omega(G_{\ol{k}})$ denotes the space of invariant differential forms on $G_{\ol{k}}$. The assignment $\kX\mapsto \sM^\Omega_{\kX}$ defines a fibration to torsors for proper modulus pairs in the sense of Definition \ref{defn:fibtors}
\end{defn}
\begin{remark}\label{rmk:COmegakbarenough} It follows immediately that $(u,P,G) \in \XP$ belongs to $\sM^\Omega_{\kX}$  if and only if \[u_{L}^*\Omega(G_{L})\subseteq U(\ol{X}_{L}, X_{\infty, L}) \]
	for any algebraically closed field $L\subset k$. Moreover, $\sM^\Omega_{\kX}$ satisfies condition (I) of Theorem \ref{thm:SerreChevalley1} and, if $\ch(k)=0$, it satisfies (II) too since $\Omega(G_{\ol{k}})\to\Omega(H_{\ol{k}})$ is surjective for $f:H \to G$ as in (II). 
\end{remark}
\begin{thm}\label{thm:existenceAlbaneseOmega}If $\ch(k)=0$, then for any $\kX \in \MCor$, the $\sM^\Omega_{\kX}$-Albanese torsor of $\kX$ exists.
	\begin{proof}Suppose that $k=\ol{k}$ is algebraically closed. Since any $k$-torsor under an algebraic group $G\in \sG$ is trivial in this case, the category $\sM^\Omega_{\kX}$ is equivalent to the category of morphisms $u: X \to G$ in $X\backslash \sG$ satisfying the condition $u^*(\Omega(G))\subseteq U(\kX)$. Morphisms in $\sM^\Omega_{\kX}$ are $k$-morphisms $f\colon G\to G'$ of \textit{torsors} commuting with the structural morphisms $u\colon X\to G$ and $u'\colon X\to G'$ (where we view $G$ and $G'$ acting on themselves). Although the morphism $f$ is not a homomorphism of algebraic groups in general, it can be written as $f = f_0 + \tau$, where $f_0$ is a group homomorphism  and $\tau$ is a translation.
	Following \cite{SerreChevalley1}, in order to check whether an initial object for  $\sM^\Omega_{\kX}$ exists using Theorem \ref{thm:SerreChevalley1}, it is enough to restrict to the subcategory   $\sM^{\Omega,g}_{\kX}$ of morphisms which are \textit{generating} (see \cite[Definition 1]{SerreChevalley1}).  
	For them, one has the following simple 
		\begin{lemma}[Lemma 6, p.198,  \cite{FaltingsWustholz}]\label{lem:FW-pullback-injective-forms}Let $u\colon X\to G \in \sM^\Omega_{\kX}$ and suppose that $u$ is generating in the sense of \cite[Definition 1]{SerreChevalley1}. Then the pullback map
			\[u^*\colon H^0(G, \Omega^1_G)^{\rm inv} = \Omega(G)\to H^0(X, \Omega^1_X)\]
			is injective.\end{lemma}
		By e.g.~\cite[Prop.~III.16]{SerreGACC}, the dimension of any $G\in \sG$ agrees with the dimension of the $k$-vector space $\Omega(G)$ of invariant differential forms.  The previous Lemma implies that for any $u\colon X\to G \in \sM^{\Omega,g}_{\kX}$, one has $\dim G \leq \dim U(\kX)$. 
By Theorem \ref{thm:SerreChevalley1}, this concludes the case $k=\ol{k}$.

		Suppose now that $k$ is any perfect field and let $\ol{k}$ be an algebraic closure of $k$. Write $\Xkbar$ for the base change $\ol{X}\tensor_k \ol{k}$ and $\kX_{\ol{k}}$ for the pair $(\Xkbar, (X_\infty)_{\ol{k}})$. According to the above argument, the category $\sM^{\Omega}_{\kX_{\ol{k}}}$ admits a universal object, 
		\[ {\rm alb}^{\Omega}_{\kX_{\ol{k}}} \colon X_{\ol{k}}\to \Alb^{\Omega}_{\kX_{\ol{k}}}.\]
		The descent to the base field $k$ can be done following  the proof of Serre \cite[V.22]{SerreGACC} in the case of generalized Jacobians of curves to get a triple $({\rm alb}_{\kX}^{\Omega}, \Alb^{\Omega,(1)}_\kX, \Alb^{\Omega, (0)}_\kX)$  defined over $k$.
	\end{proof} 
\end{thm}
\begin{remark}
	It should be possible to closely follow the construction of \cite{ESV} of the universal regular quotient of the Chow group of zero cycles, that works over any field, {to remove the hypothesis on the characteristic of $k$}. This was used in \cite{BK} to construct the universal regular quotient of the Kerz-Saito Chow group of zero cycles with modulus. Since the applications we have in mind in the later sections will require the characteristic zero assumption, we have decided to not pursue this goal here.
\end{remark}
\subsection{Cutting curves}\label{ssec:CuttingCurvesII} Assume now that $\kX = (\ol{X},X_\infty)\in \MCor$ is such that $\ol{X}$ is smooth over $k$. A finite morphism $\nu\colon \ol{C} \to \ol{X}$, with $\ol{C}$ a normal and geometrically integral curve, is \textit{admissible} for $\kX$ if $\nu(\ol{C}) \not\subseteq X_\infty$. In this case, write $C_\infty$ for the effective Cartier divisor $\nu^* X_\infty$ on $\ol{C}$. Write $C$ for the open subset $\ol{C}\setminus C_\infty$.  We will use the following Lemma, taken from \cite{BS}.
\begin{lemma}[Lemma 10.14, \cite{BS}]\label{lem:restriction-injective}
	Let $\gamma$  be  the restriction map
	\[ \gamma\colon H^0(X, \Omega^1_X) \to \prod_{\nu\colon \ol{C}\to \ol{X}} H^0(C, \Omega^1_C)/ H^0(\ol{C}, \Omega^1_{\ol{C}}(\nu^* X_\infty)),\]
	where the product runs over the set of admissible curves $\nu\colon \ol{C}\to \ol{X}$.	Then the  kernel of $\gamma$ agrees with $H^0(\ol{X}, \Omega^1_{\ol{X}}(X_\infty))$. \end{lemma}

If $u\colon X\to P$ is a $k$-morphism from $X$ to a torsor $P$ for an algebraic group $G\in \sG$, we get by composition a morphism
\[u_C \colon C:=\ol{C}\times_{\ol{X}} X \to X \xrightarrow{u} P. \]
Write $\nu^*(u, P, G)$ for the corresponding object in $ C \backslash \sP$.
\begin{lemma}\label{lem:cuttingCurvesOmega}A triple $(u, P, G)\in \XP$ belongs to $\sM^\Omega_{\kX}$ if and only if for any admissible curve $\nu\colon \ol{C}\to \ol{X}$, we have $\nu^*(u,P,G)\in \sM^{\Omega}_{(\ol{C}, C_\infty)}$.\end{lemma}
\begin{proof} The necessity of the condition is clear. According to  Definition  \ref{defn:COmega} and Remark \ref{rmk:COmegakbarenough}, the statement of the Lemma can be checked over an algebraic closure $\ol{k}$ of $k$, so that we can assume $k = \ol{k}$. As above, the category $\sM^\Omega_{\kX}$ is equivalent to the category of morphisms $\psi\colon X\to G$ from $X$ to algebraic groups satisfying the condition $\psi^*(\Omega(G))\subseteq U(\kX)$. 
	Let now $\psi\colon X\to G$ be a $k$-morphism, and let $\omega\in \Omega(G)$. We have to show that $\psi^*(\omega)\in H^0(\ol{X}, \Omega^1_{\ol{X}}(X_\infty))$ (note that $\psi^*(\omega)$ is automatically closed), i.e. that $\eta:=\psi^*(\omega)$ has poles along $|X_\infty|$ of order bounded by the multiplicity of $X_\infty$, assuming that this condition is satisfied after restriction admissible to curves. But this is precisely the content of Lemma \ref{lem:restriction-injective}.
\end{proof}

\subsection{The universal regular quotient of the Chow group of zero cycles} 
We start by recalling the definition of the Kerz-Saito Chow group of $0$-cycles with modulus (see \cite{KerzSaitoDuke}).
For an integral scheme $\overline{C}$ over $k$ and for $E$ a closed subscheme of $\overline{C}$, we set
\begin{align*}
	G(\ol{C},E)
	&=\bigcap_{x\in E}\mathrm{Ker}\bigl(\cO_{\overline{C},x}^{\times}\to \cO_{E,x}^{\times} \bigr) \\&= \varinjlim_{E\subset U\subset \overline{C}}\Gamma(U, \ker(\cO_{\overline{C}}^\times \to \cO_{E}^\times)),
\end{align*}
where  $U$ runs over the set of open subsets of $\ol{C}$ containing $E$ (the intersection taking place in the function field $k(\ol{C})^\times$). We say that a rational function $f\in G(\ol{C},E)$ satisfies the modulus condition with respect to $E$.

Let $\kX = (\ol{X}, X_\infty)\in {\bf MCor}$ be a proper modulus pair and write $X$ for the complement $\ol{X}\setminus{|X_\infty|}$. Let  $Z_0(X)$ be the free abelian group on the set of closed points of $X$. Let $\ol{C}$ be an integral normal curve over $k$ and
let $\varphi_{\ol{C}}\colon \ol{C}\to \ol{X}$ be a finite morphism such that $\varphi_{\ol{C}}(\ol{C})\not \subset X_\infty$ (so $\ol{C}$ is admissible in the sense of \ref{ssec:CuttingCurvesII}).   The push forward of cycles along the restriction of $\varphi_{\ol{C}}$ to $C = \ol{C}\times_{\ol{X}}X$ gives a well defined group homomorphism
\[
\tau_{\ol{C}}\colon G(\ol{C},\varphi_{\ol{C}}^*(X_\infty))\to Z_0(X),
\]
sending a function $f$ to the push forward of the divisor ${\rm div}_{\ol{C}}(f)$.
\begin{defn}[Kerz-Saito]\label{def:DefChowMod-Definition}
	We define the Chow group $\CH_0(\kX) = \CH_0(\ol{X}|X_\infty)$ of 0-cycles of $\ol{X}$ with modulus $X_\infty$ as the cokernel of the homomorphism 
	\begin{equation}\label{eq;def:DefChowMod-Definition}
	\tau\colon G(\ol{X},X_\infty)\to Z_0(X)\;\text{ with  }
	G(\ol{X},X_\infty)=\bigoplus_{\varphi_{\ol{C}}\colon \ol{C}\to \ol{X}}G(\ol{C},\varphi_{\ol{C}}^*(X_\infty)) 
\end{equation}
where the sum runs over the set of finite morphisms $\varphi_{\ol{C}}\colon \ol{C}\to X$  from admissible curves.
\end{defn}
\begin{remark}\label{rem:CH-dim-1}
	Let $\pi_\kX\colon \Z_0(X)\to \CH_0(\kX)$. If $\dim(X)=1$, then by construction the map $\tau$ is injective, so $G(\ol{X},X_{\infty})=\Ker(\pi_\kX)$ and $\CH_0(\kX)\simeq \Pic(\ol{X},X_\infty)$. In particular, if $X_\infty'\geq X_\infty$ we have an isomorphism
	\begin{equation}\label{eq:trick-divisible}
	\frac{G(\ol{X},X'_{\infty})}{G(\ol{X},X_{\infty})} \cong \ker(\Pic(X,X'_\infty)\to \Pic(X,X_\infty))
	\end{equation}
\end{remark}
\begin{defn}\label{defn:CChow} Let $X\in \Sm(k)$ geometrically connected and $\kX \in \Comp(X)$ (see Definition \ref{defn;comp-modulus}). 
		Let $\CH_0(\kX)$ 
		be the Chow group of $0$-cycles with modulus of $\kX$. Let $\sM^{\CH}_{\kX}$ be the full subcategory of $\XP$ consisting of triples $(u, P, G)$ with the following property. For any algebraically closed field $L\supset k$, write $u_L\colon Z_0(X)^0 \to P_L(L)\cong G_L(L)$ for the induced morphism on zero-cycles of degree zero. Then $(u,P,G) \in \sM^{\CH}_{\kX}$  if and only if  $u_L$ factors as
	\[ \xymatrix{ Z_0(X_L)^0\ar[d]\ar[rd]^{u_L} &\\
		\CH_0(\kX_L)^0\ar[r]& G_L(L), 
	}
	\]
	where $\CH_0(\kX_L)^0$ is the image of $Z_0(X_L)^0$. 
\end{defn}
\begin{remark}\label{rmk:reciprocity-inplies-cond-II}
	Since $G$ is a reciprocity sheaf by \cite[Cor.3.2.5]{KSY-RecII}, for $u\in G(X)$ there always exists a modulus $\kX\in \Comp(X)$ such that the induced map $\Z_{\tr}(X)\to G$ factors through $h_0(\kX)=\omega_!h_0^\bcube(\kX)$ (cf. Remark \ref{rmk;reciprocity-only-(n)}(ii)). 
 By taking sections over an algebraically closed field $L\supset k$ we get that $u_L$ factors as follows:\[
	\begin{tikzcd}
		\Z_{\tr}(X)(L)\cong Z_0(X_L)\ar[d]\ar[dr,"u_L"]\\
		h_0(\kX)(L)\underset{(*)}{\cong} \CH_0(\kX_L)\ar[r] &G_L(L),
	\end{tikzcd}
\]
where the isomorphism $(*)$ follows from \cite[Remark 2.2.3]{KSY-RecII}.
In particular, for all $\kX'\to \kX\in \Comp(X)$, the diagram above factors further through the map $\CH_0(\kX'_L)\to \CH_0(\kX_L)$.
\end{remark}
\begin{lemma}\label{lem:MCH-satisfies-I-and-II}
	If $\ch(k)=0$, then $\sM^{\CH}_{\kX}$ satisfies the conditions $\mathrm{(I)}$ and $\mathrm{(II)}$ of Theorem \ref{thm:SerreChevalley1}.
\end{lemma}
\begin{proof}
	The assertion is obvious for (I). We prove it for (II).
	Take a homomorphism $f: H\to G$ in $\sG$ such that $\Ker(f)$ is finite and 
	$v:X\to H$ such that $u=f\cdot v: X\to G$ is in $\sM^{\CH}_{\kX}$.
	We want to prove $v:X\to H$ is in $\sM^{\CH}_{\kX}$, equivalently
	\[\Ker(\pi_{L,\kX}: Z_0(X_L)^0\to \CH_0(\kX_L)^0)\subset \Ker(v_L: Z_0(X_L)^0 \to H_L(L))\]
	for every algebraically closed field $L\supset k$.
	Since $H$ is a reciprocity sheaf, by Remark \ref{rmk:reciprocity-inplies-cond-II} there exists $\kX'=(\ol{X},X'_\infty)$ such that $v_L$ factors through $\CH_0(\kX'_L)^0$, so $\Ker(\pi_{L,\kX'})\subset \Ker(v_L)$. Since we can always factor through maps $\kX''\to \kX\in \Comp(X)$ as observed in Remark \ref{rmk:reciprocity-inplies-cond-II}, we can choose $X'_\infty\geq X_{\infty}$.
	This implies that $v_L$ induces a map 
\[ \ol{v}_L: \frac{G(\ol{X}_L,(X_\infty)_L)}{G(\ol{X}_L,(X'_\infty)_L)} \to H(L),\]
	where the left hand side is equal to 
	\[ \bigoplus_{\varphi_{\ol{C}}\colon \ol{C}\to \ol{X}_L}
	\frac{G(\ol{C},C_\infty)}{G(\ol{C},C'_\infty) } 
\;\text{ with } C_\infty=\varphi_{\ol{C}}^*((X_\infty)_L),\; 
C'_\infty=\varphi_{\ol{C}}^*((X'_\infty)_L) ,\] 
where $\varphi_{\ol{C}}\colon \ol{C}\to \ol{X}_L$ are as in \eqref{eq;def:DefChowMod-Definition} with $\ol{X}$ replaced by $\ol{X}_L$. 
Moreover, the image of $\ol{v}_L$ lies in $\Ker(f)(L)$, which is finite by the assumption. We will show that the group 
$\frac{G(\ol{C},C_\infty)}{G(\ol{C},C'_\infty) }$ is divisible, which will imply that the map $\ol{v}_L$ is the zero map. As observed in Remark \ref{rem:CH-dim-1}, we have
\[\frac{G(\ol{C},C_\infty)}{G(\ol{C},C'_\infty))}= \Ker(\Pic(\ol{C},C'_\infty))\to \Pic(\ol{C}_L,C_\infty)),\] 
	and the latter equals to $(L_+)^r\times (L^\times)^s$, where $r,s$ are some non-negative integers. In particular, the above group is divisible thanks to the assumption $\ch(k)=0$, so $\ol{u}_L$ is the zero map.
\end{proof}
\begin{prop}\label{prop:AlbaneseChow} Assume $\ch(k)=0$ and that $\kX$ is a smooth and proper modulus pair over $k$. Then the $\sM^{\CH}_{\kX}$-Albanese torsor $({\rm alb}_{\kX}^{\CH}, \Alb^{\CH,(1)}_\kX, \Alb^{\CH, (0)}_\kX)$ of $\kX$ exists. If $\dim \kX =1$, it agrees with the Rosenlich-Serre generalized Jacobian $(\varphi_{X_\infty}, \Jac_{(\ol{X}, X_\infty)}^{(1)}, \Jac_{(\ol{X}, X_\infty)}^{(0)})$ of \cite[V.4.20]{SerreGACC}.
\end{prop}
\begin{proof} If $\dim \kX=1$, this is precisely the content of \cite[V.Theorem 1]{SerreGACC}, and the very definition of modulus for a rational map and local symbols. For the general case, as in the proof of Theorem \ref{thm:existenceAlbaneseOmega}, it is enough to show the existence in the case $k=\ol{k}$ is algebraically closed (the descent argument is identical). Similarly, we can restrict to the category $\sM^{\CH,g}_{\kX}$ of morphisms which are generating. Then by Theorem \ref{thm:SerreChevalley1} and Lemma \ref{lem:MCH-satisfies-I-and-II}, it is then enough to show that there exists a uniform bound on the dimensions of the groups appearing in $\sM^{\CH,g}_{\kX}$. We do this by showing  that $\sM^{\CH}_{\kX}$ is a (full) subcategory of $\sM^{\Omega}_{\kX}$. The required bound will be then provided by Lemma \ref{lem:FW-pullback-injective-forms}.
	
	Suppose then that that $u\colon X\to G$ is a $k$-morphism from $X = \ol{X}\setminus |X_\infty|$ to a commutative connected algebraic group $G$ such that the induced map on zero cycles factors through $\CH_0(\kX)$. Let $\varphi\colon \ol{C}\to \ol{X}$ be a finite morphism from normal integral curve $\ol{C}$ such that $\varphi(\ol{C})\not\subset X_\infty$. Put $C= \ol{C}\times_{\ol{X}} X$ and $C_\infty = \varphi^*(X_\infty)$. Let $u_C\colon C\to G$ be the composition $u\circ \varphi$. Then we have
	\[ u_C( \div_{\ol{C}}(f)) = u ( \varphi_* (\div_{\ol{C}}(f))) = 0 \text{ in } G(k) \text{ for any } f\in G(\ol{C}, C_\infty).\]
	In particular, the divisor $C_\infty$ is a modulus in  the sense of Rosenlicht-Serre for the rational map (still denoted $u_C$) $u_C\colon \ol{C}\dashrightarrow G$. Therefore we have then a factorization (\cite[V, Theorem 2]{SerreGACC})
	\[ \xymatrix{ C \ar[d]^a\ar[dr]^{u_C} &\\
		\Jac_{\ol{C}, C_\infty}\ar[r]_-{\tilde{u}_C} & G},\]
	where $a\colon C\to\Jac_{\ol{C}, C_\infty}$ is the universal map from $C$ to its generalized Jacobian (with respect to a chosen $k$-rational point). But now we have
	\[ u_C^*(\Omega(G)) = a^*(\tilde{u}_C^* (\Omega(G)))\subseteq a^*(\Omega(\Jac_{\ol{C}, C_\infty})) = H^0(\ol{C}, \Omega^1_{\ol{C}}\tensor_{\ol{C}} \cO_{\ol{C}}(C_\infty)),\]
	where the last equality follows from \cite[V, Proposition 5]{SerreGACC}. This implies that $(u_C, G)\in \sM^{\Omega}_{(\ol{C}, C_\infty)}$ for any admissible curve and so, by Lemma \ref{lem:cuttingCurvesOmega}, we deduce that $(u,G)\in \sM^{\Omega}_{\kX}$.
\end{proof}
\subsubsection{} 
Assume $\ch(k)=0$. Let $\kX$ be as above. From the proof of Proposition \ref{prop:AlbaneseChow}, we deduce immediately the existence of a natural surjective map of torsors
\[\rho^{\Omega, \CH}_{\kX}\colon  \Alb^{\Omega, (1)}_{\kX}\to \Alb^{ \CH, (1)}_{\kX} \]
equivariant with respect to a surjective homomorphism of algebraic $k$ groups $\Alb^{\Omega, (0)}_{\kX}\to \Alb^{ \CH, (0)}_{\kX}$
such that ${\rm alb}^{\CH}_{\kX}=\rho^{\Omega, \CH}_{\kX}\circ  {\rm alb}^{\Omega}_{\kX}$.  We will see below that those maps are isomorphisms.
To further relate the Chow groups of zero cycles with modulus of a pair $\kX$ with the $\sM^\Omega_{\kX}$ Albanese construction of \ref{ssec:AlbaneseOmega}, we also recall the following
\begin{prop}[{\cite[III, Proposition 10]{SerreGACC} or \cite[Proposition 4.3.1]{KSY}}]\label{prop:SerreModulusOmega}
	Let $G$ be a commutative algebraic group over a field $K$ of characteristic zero. Let $\ol{C}$ be a proper normal curve over $K$, $C$ an open dense subscheme of $\ol{C}$ and $\psi\colon C\to G$ a $K$-morphism. Let $D$ be an effective divisor on $\ol{C}$ supported on $\ol{C} \setminus C$ such that
	\[ \psi^* (\Omega(G)) \subset H^0(\ol{C}, \Omega^1_{\ol{C}}\tensor_{\ol{C}} \cO_{\ol{C}}(D)). \]
	Then we have $\div_{\ol{C}}(f)^*\psi=0$ in $G(K)$ for any 
	$f\in G(\ol{C},D)$.
\end{prop}
\begin{prop}\label{prop:AlbOmega-AlbChow-char0}Let $k$ be a field of characteristic zero. Then the   $\sM^{\CH}_{\kX}$-Albanese torsor of $\kX$ agrees with the $\sM^{\Omega}_{\kX}$-Albanese torsor $({\rm alb}_{\kX}^{\Omega}, \Alb^{\Omega,(1)}_\kX, \Alb^{\Omega, (0)}_\kX)$ of Theorem \ref{thm:existenceAlbaneseOmega}. 
	\begin{proof} It is enough to show that the two categories  $\sM^{\CH}_{\kX}$ and $\sM^{\Omega}_{\kX}$ have the same objects (as they are both full subcategories of $\XP$). According to Definitions \ref{defn:CChow} and \ref{defn:COmega}, it is enough to show the statement under the assumption that $k=\ol{k}$ is algebraically closed. 
We already know thanks to the proof of Proposition \ref{prop:AlbaneseChow} that $\sM^{\CH}_{\kX}$ is a full subcategory of $\sM^{\Omega}_{\kX}$
(this does not require $k$ to be of characteristic $0$).
To prove the other inclusion, let $u\colon X\to G\in \sM^{\Omega}_{\kX}$. By functoriality, for any $\kX$-admissible  morphism $\varphi\colon \ol{C}\to \ol{X}$ from a normal integral curve, the composition $u_C\colon \ol{C}\times_{\ol{X}} X\to X\to G$ satisfies $u^*_C(\Omega(G)) \subseteq H^0(\ol{C}, \Omega^1_{\ol{C}}\tensor_{\ol{C}} \cO_{\ol{C}}(C_\infty))$, where $C_\infty$ denotes as before the pullback $\varphi^*(X_\infty)$. By Proposition \ref{prop:SerreModulusOmega} above, we have $u_C(\div_{\ol{C}}(f))= u( \varphi_{*} \div_{\ol{C}}(f)) = 0$ in $G(k)$ for any $f\in G(\ol{C}, C_\infty)$, so that the map induced by $u$ on the group of zero cycles of degree zero $Z_0(X)^0$ factors through $\CH_0(\kX)^0$. The same argument applies to any base-change to  $L\supset k$ algebraically closed, so that $(u,G)\in \sM^{\CH}_{\kX}$ as required.
	\end{proof}
\end{prop}
\begin{remark}\label{rmk:alb-surjective} Suppose $\ch(k)=0$ and that $k=\overline{k}$ is algebraically closed and let $\kX$ be a smooth proper integral modulus pair. Then the morphism 
	\[\rho_{\kX}\colon \CH_0(\kX)^0 \to \Alb_{\kX}(k)\]
	to the Albanese variety $\Alb_{\kX} = \Alb^{\CH, (0)}_\kX$ of $\kX$ induced by ${\rm alb}_{\kX}$ is surjective and regular. We can reformulate the universal property in $\sM^{\CH}_{\kX}$ by saying that $\Alb_{\kX}$ is the universal regular quotient of the Chow group of zero cycles with modulus. As such it agrees, a posteriori, with the Albanese variety $\Alb(\ol{X}|X_\infty)$ of \cite[Theorem 1.1]{BK}. If $k=\mathbb{C}$ and $|X_\infty|$ is a strict normal crossing divisor, it agrees with the generalized Jacobian $J^{d}_{\ol{X}|X_\infty}$ (for $d=\dim {\ol{X}}$) studied in \cite[10.2]{BS}. Despite the fact that it is defined starting from a different modulus condition on algebraic cycles, it agrees also with the Albanese variety with modulus $\Alb(\ol{X}, X_\infty)$ defined by Russell \cite{RussellANT}. This is a consequence of Lemma \ref{lem:cuttingCurvesOmega} and the fact that both {$\Alb(\ol{X}, X_\infty)$ and $\Alb^{\Omega,(0)}(\kX)$} agree with the Rosenlicht-Serre generalized Jacobian in the one-dimensional case\footnote{An independent (and explicit) proof of the fact that over $\mathbb{C}$ the generalized Jacobian $J^{d}_{\ol{X}|X_\infty}$ agrees with $\Alb(\ol{X}, X_\infty)$ defined by Russell has been given by T.~Yamazaki \cite{YamazakiLetter}, using Hodge-theoretic methods.}. 
\end{remark}
\subsection{The Albanese scheme with modulus} \label{ssec:AlbaneseScheme} 
Assume $\ch(k)=0$.
We  can now extend the construction of Ramachandran \cite{Ramachandran} to the modulus setting.  Let $X\in \Sm(k)$ be geometrically integral and $\kX\in \Comp(X)$. Thanks to Theorem \ref{thm:existenceAlbaneseOmega}, we have a map, defined over $k$, 
\begin{equation}\label{eq:univ-map-alb-torsor}{\rm alb}_{\kX}^{\Omega, (1)}\colon X \to \Alb^{\Omega, (1)}_{\kX}\end{equation}
universal for morphisms from $X$ to torsors under commutative algebraic groups in $\sM^{\Omega}_{\kX}$ (in the following, we shall say ``to torsors under commutative algebraic groups with modulus $\kX$''). Let $\mathbf{Alb}^{\Omega}_{\kX}\in \sG^*$ be the $k$-group scheme $\coprod_{n\in \Z}  (\Alb^{\Omega, (1)}_{\kX})^{\tensor n}$ constructed in \ref{ssec:torsor-construction}. The universal map \eqref{eq:univ-map-alb-torsor} composed with the natural inclusion of $\Alb^{\Omega, (1)}_{\kX}$ in $\mathbf{Alb}^{\Omega}_{\kX}$ gives then a canonical morphism 
\[ a_{\kX}\colon X\to \mathbf{Alb}^{\Omega}_{\kX}\]
which is now universal for morphisms to $k$-group schemes in the appropriate sense. By construction, the $k$-group scheme $\mathbf{Alb}^{\Omega}_{\kX}$ is an extension
\begin{equation}\label{eq:alb0-es}
	0\to  \Alb^{\Omega, (0)}_{\kX} \to \mathbf{Alb}^{\Omega}_{\kX} \to \underline{\Z}\to 0.
\end{equation}
If $X$ has a $k$-rational point, the extension is split, i.e. we have an isomorphism $\mathbf{Alb}^{\Omega}_{\kX}  \cong \Alb^{\Omega, (0)}_{\kX} \times \underline{\Z}$. This happens in particular when $k$ is algebraically closed, and corresponds to the fact that we can trivialize the torsor $\Alb^{\Omega, (1)}_{\kX} \cong \Alb^{\Omega,(0)}_{\kX}$. 
Recall now the following Proposition, which follows from \cite[Theorem 4.1.1]{KSY} (while the transfer structure follows from \cite[Proof of Lemma 3.2]{SS})
\begin{prop}\label{prop:Alb-Omega-PST} The $k$-group scheme $\mathbf{Alb}^{\Omega}_{\kX}$, regarded as \'etale sheaf on $\Sm(k)$,  has a canonical structure of sheaf with transfers, and as such it has reciprocity in the sense of \cite[Definition 2.1.3]{KSY}.
\end{prop}
Thanks to the Proposition, there is a unique map of preshaves with transfers
\[a_{\kX}\colon \Z_{tr}(X) \to \mathbf{Alb}^{\Omega}_{\kX}\]
extending the map $a_{\kX}\colon X\to \mathbf{Alb}^{\Omega}_{\kX}$ defined above. 

\subsubsection{}When $\kX$ is moreover smooth over $k$, we can apply Proposition \ref{prop:AlbaneseChow} to get a map, defined over $k$, ${\rm alb}_{\kX}^{\CH, (1)}\colon X \to \Alb^{\CH, (1)}_{\kX}$, universal for morphisms to torsors in $\sM^{\CH}_{\kX}$. We can repeat the constructions of the previous point to get yet another $k$-group scheme $\mathbf{Alb}^{\CH}_{\kX}\in \sG^*$ together with a canonical morphism
\[a_{\kX}^{\CH}\colon X\to \mathbf{Alb}_{\kX}^{\CH}.\]
The group  $\mathbf{Alb}_{\kX}^{\CH}$ has by the same argument of Proposition \ref{prop:Alb-Omega-PST} a canonical structure of \'etale sheaf with transfers, with reciprocity in the sense of \cite[Definition 2.1.3]{KSY}. This gives us a unique map of presheaves with transfers
\[a_{\kX}^{\CH}\colon \Z_{tr}(X)\to \mathbf{Alb}_{\kX}^{\CH}\]
extending $a_{\kX}^{\CH}$. By Proposition \ref{prop:AlbOmega-AlbChow-char0}, we canonically identify $\mathbf{Alb}_{\kX}^{\CH}$ with $\mathbf{Alb}_{\kX}^{\Omega}$. 

\subsection{The maximal semi-abelian quotient} Let $X\in \Sm(k)$ be geometrically connected. 
Serre's Albanese map
of Example \ref{ex:SAb-case} can be extended to a unique map of presheaves with transfers:
\begin{equation}\label{eq:SerreAlb}\Z_{tr}(X)\to \mathbf{Alb}_{X}\end{equation}
where $\mathbf{Alb}_X$ is the semi-abelian Albanese scheme of $X$, defined by \cite{Ramachandran} or \cite[Lemma 3.2]{SS} using the same recipe of Section \ref{ssec:AlbaneseScheme}. Since $\mathbf{Alb}_{X}$ is semi-abelian, it is a homotopy invariant \'etale sheaf with transfers. Thus, taking sections over any field $L\supset k$, we have a factorization of the map \eqref{eq:SerreAlb}  through
\[h_0^{\mathbb{A}^1}(X_L) \to \mathbf{Alb}_{X}(L),\]
where $h_0^{\mathbb{A}^1}(X_L) $ denotes the zeroth Suslin homology group of $X_L= X\tensor_k L$. If $X$ admits a $k$-point, the  scheme  $\mathbf{Alb}_{X}$ decomposes as $\underline{\Z}\times \Alb^{(0)}_{X}$, where $\Alb^{(0)}_{X}$ denotes Serre's semi-abelian Albanese variety of $X$. In particular, we get for any $L\supset k$ algebraically closed an induced (surjective) map on the degree zero part
\begin{equation}\label{eq:SerreAlb-H0} h_0^{\mathbb{A}^1}(X_L)^0 \to {\Alb}_{X_L}^{(0)}(L) = \Alb_{X_L}(L).\end{equation}
Now assume $\ch(k)=0$.
Let $\kX\in \Comp(X)$. By e.g. \cite[Proposition 2.6]{tor-div-rec}, there is a natural surjection $\CH_0(\kX)\to h_0^{\mathbb{A}^1}(X)$, which can be composed with \eqref{eq:SerreAlb-H0} to give a surjective homomorphism
\[ \CH_0(\kX_L)^0\to \Alb_{X_L}(L).\]
By Definition \ref{defn:CChow}, we have then that the object $({\rm alb}_X, \Alb^{(1)}_{X}, \Alb^{(0)}_{X} )$ belongs to $\sM_{\kX}^{\CH}$. By Proposition \ref{prop:AlbaneseChow}, the universal property of $\Alb^{\CH, (0)}_{\kX}$ gives a  unique surjection 
\[\Alb^{\CH, (0)}_{\kX}\to \Alb^{(0)}_{X} \]
(and similarly for $\Alb^{\CH, (1)}_{\kX}$ and $\Alb^{(1)}_{X}$), which factors through the semi-abelian quotient $\Alb^{\CH, (0)}_{\mathbf{SAb}, \kX}$ of $\Alb^{\CH, (0)}_{\kX}$. It is straightforward to show that  if $X$ has a $k$-rational point, the algebraic groups   $\Alb^{\CH, (0)}_{\mathbf{SAb}, \kX}$ and $\Alb_X^{(0)}$ are isomorphic. 
We have therefore  the following 
\begin{prop}\label{prop:max-semiab-quotient}Let $\kX$ be as above, and suppose that $X$ has a $k$-rational point. Then the semi-abelian part of $\Alb_{\kX}^{\CH, (0)}$ agrees with Serre's semi-abelian Albanese variety of $X$. 
\end{prop}
\subsection{Universal problem for presheaves with transfers} \label{ssec:univ-prop}

We suppose $\ch(k)=0$ and continue with the notations of \ref{ssec:AlbaneseScheme}. As observed in Remark \ref{rmk;reciprocity-only-(n)} (ii), for any {$F \in \RSC(k,\Z)$} and for every section $g : \Z_{tr}(X) \to F$, there exists $\kX\in \Comp(X)$ such that g factors through $\omega_!h_0^{\bcube}(\kX)$. In this case, we say that $g$ \textit{has modulus} $\kX$. We apply this to the case $\mathbf{Alb}^{\CH}_{\kX}$.
\begin{prop}\label{prop:UniversalPropertyPST}
	Let $X\in \Sm(k)$ be geometrically connected and $\kX\in \Comp(X)$.
	Then the canonical map $a_{\kX}\colon \Z_{tr}(X) \to \mathbf{Alb}_{\kX}^{\CH}$ factors through $\omega_!h_0^{\bcube}(\kX)$ and it is universal with respect to this property: for any smooth commutative $k$-group scheme $G$, seen as \'etale reciprocity sheaf, and for any section $g\colon \Z_{tr}(X)\to G$ with modulus $\kX$, there is a unique morphism $\tilde{g}\colon \mathbf{Alb}^{\CH}_{\kX} \to G$ in $\PST$ such that \[
	\begin{tikzcd}
		\Z_{tr}(X)\ar[rr,"g"]\ar[dr,swap,"a_{\kX}"]&&G\\
		&\mathbf{Alb}^{\CH}_{\kX}.\ar[ur,swap,"\tilde{g}"]	
	\end{tikzcd}
	\]
	\begin{proof}
		We first prove that $a_{\kX}\colon \Z_{tr}(X) \to \mathbf{Alb}_{\kX}^{\CH}$ factors through $\omega_!h_0^{\bcube}(\kX)$. We have to show that for any smooth $k$-scheme $S$, the map $\Z_{tr}(X)(S) \to \mathbf{Alb}_{\kX}^{\CH}(S)$ factors through $h_0(\kX)(S)$ with $h_0(\kX)=\omega_!h_0^\bcube(\kX)$ (cf. Remark \ref{rmk;reciprocity-only-(n)}(ii)). 
Since $\mathbf{Alb}_{\kX}^{\CH}$, as any commutative $k$-group scheme, satisfies global injectivity, it is enough to check the factorization after passing to the function field $k(S)$ of $S$, and in fact even to its algebraic closure. Let then $K\supset k$ be an algebraically closed field, and look at the map $\Z_{\tr}(X)(K) \to \mathbf{Alb}_{\kX_K}^{\CH}(K)$. Let $\pi_0(X)$ be the spectrum of the integral closure of $k$ in $\Gamma(X, \cO_X)$. The assignment $X\mapsto \pi_0(X)$ is universal for morphisms from $X$ into \'etale $k$-schemes. 
		Since $X$ is geometrically integral by assumption, we have $\Z_{\tr}(\pi_0(X)) = \underline{\Z}$ (as \'etale sheaves), and the map $\Z_{\tr}(X_K)\to \mathbf{Alb}_{\kX_K}^{\CH}$ induces then a map \[\Z_{\tr}(X_K)^0\to \Alb^{\CH,(0)}_{\kX_K},\] where $\Z_{\tr}(X_K)^0$ denotes the kernel of $\Z_{\tr}(X_K) \to \Z_{\tr}(\pi_0(X_K)).$ We can then identify $\Z_{tr}(X_K)^0(K)$ with the group of $0$-cycles of degree zero $Z_0(X_K)^0$ of $X_K$. Since $\kX$ is a proper modulus pair, we have $\CH_0(\kX_K)^0 = h_0(\kX_K)^0(K)$ thanks to \cite[Remark 2.2.3]{KSY-RecII},
		and thus the claim follows from Proposition \ref{prop:AlbaneseChow}. The same argument proves the universal property as well.
	\end{proof}
\end{prop}
From now on, we write simply $\mathbf{Alb}_{\kX}$ for the $k$-group scheme $\mathbf{Alb}^{\CH}_\kX$. We end this section with the following result, which will be crucial for the construction of the category of $1$-reciprocity sheaves. 
\begin{lemma}\label{lem:Alb-map-surjective}Under the assumptions of Proposition \ref{prop:UniversalPropertyPST}, the map\[ a_{\kX}\otimes_{\Z}\Q\colon \Q_{tr}(X) \to \mathbf{Alb}_{\kX}\otimes_{\Z}\Q
	\] 
	is a surjective morphism of \'etale sheaves with transfers with rational coefficients.
	\begin{proof}
		Let ${\rm Im}(a_{\kX}) \subset \mathbf{Alb}_{\kX}$  be the image of $a_{\kX}$ in $\PST$.
		Since $\bAlb_{\kX}$ is a smooth $k$-group scheme, we have $\bAlb_{\kX}\in \RSC_{\Nis}$ by \cite[Cor.3.2.5]{KSY-RecII}.
		Let $C = \mathbf{Alb}_{\kX}/{\rm Im}(a_{\kX})\in {\RSC(k,\Z)}$. We have to show that $a_{\et}^V (C\otimes_{\Z} \Q)=0$. By \eqref{eq;et-Nis-RSC}, we have that $a_{\et}^V (C\otimes_{\Z} \Q)\in \RSC_{\et}(k,\Q)$, in particular it satisfies global injectivity by \cite[Theorem 0.2]{SaitoPurity}, i.e. for any $Y\in \Sm$ geometrically connected with function field $k(Y)$ with algebraic closure $\ol{k(Y)}$ there is an injective map \[a_{\et}^V (C\otimes_{\Z}\Q)(Y)\hookrightarrow a_{\et}^V (C\otimes_{\Z}\Q)(k(Y))\hookrightarrow (C\otimes_{\Z}\Q)(\ol{k(Y)}) = C(\ol{k(Y)})\otimes_{\Z}\Q.\] 
		
		To complete the proof, it is then enough to show that $C(\ol{k(Y)})=0$. This follows from the fact that for any $K\supset k$ algebraically closed, the map  \[\Z_{tr}(X\otimes_k K)(K) \to \mathbf{Alb}_{\kX}(K)\] is surjective, since $\CH_0(\kX_K)^0\to  \Alb^{(0)}_{\kX}(K)$ is surjective.
	\end{proof}
\end{lemma}
\section{The Albanese functors}\label{sec:1-mot-sheaves-with-modulus}

\subsection{\texorpdfstring{$n$}{n}-reciprocity sheaves} 
For any $n\geq 0$, let $\Cor(k)_{\leq n}$ be the category of finite correspondences on smooth $k$-schemes of dimension $\leq n$, 
and let $\MCor(k)_{\leq n}$ be the category of modulus correspondences on smooth proper modulus pairs $\kX = (\ol{X}, X_\infty)$ with $\dim (\ol{X})\leq n$.
We let $\MPST(k_{\leq n},\Lambda)$ be the category of additive presheaves of $\Lambda$-modules on $\MCor(k)_{\leq n}$.



The natural inclusions of subcategory of objects of dimension $\leq n$ give rise to a standard string of adjoint functors between the category of presheaves
\begin{equation}\label{eq;sigma}(\sigma_{n, !}, \sigma_n^*), \quad \sigma_{n, !}\colon \MPST(k_{\leq n},\Lambda) \leftrightarrows \MPST{(k,\Lambda)}\colon \sigma_n^*.\end{equation}
Here, we follow the convention of \cite{MotModulusI} for the left Kan extension of the restriction functor $\sigma_n^*$. Note that this is different from the one adopted in \cite{Ayoub-BV}. 


\begin{remark}\label{rmk:comp-ABV} Let \begin{equation}\label{eq;sigma-Voe}(\sigma^{V}_{n, !}, \sigma_n^{V,*}), \quad (\sigma^{V}_{n, !}\colon \PST(k_{\leq n},\Lambda) \leftrightarrows \PST{(k,\Lambda)}\colon\sigma_n^{V,*})
	\end{equation}
	be the analogous adjoint functors from \cite{Ayoub-BV}. 
Since the functor $\omega$ clearly restricts to a functor $\MCor(k)_{\leq n}\to \Cor(k)_{\leq n}$, for all $F\in \PST(k,\Lambda)$ we have that \[\sigma_n^*\omega^*F \cong \omega^*\sigma_n^{V,*}F\textrm{ in }\MPST(k_{\leq n},\Lambda).\]
By adjunction, we conclude that for all $F\in \MPST(k_{\leq n},\Lambda)$, \[\sigma_{n,!}^{V}\omega_! F \cong \omega_!\sigma_{n,!} F\textrm{ in }\PST(k,\Lambda).\] For $F\in \MPST(k,\Lambda)$ and $X\in \Cor(k)_{\leq n}$, for any modulus pair $\kX\in \Comp(X)$, we have $\kX\in \MCor(k)_{\leq n}$, hence $\sigma_n^*F(\kX)=F(\kX)$\[
	\omega_!\sigma_n^*F (X) = \colim_{\kX\in \Comp(X)}\sigma_n^*F(\kX) = \omega_!F(X)= \sigma_n^{V,*}(\omega_!F)(X).
	\]
	Finally, for every modulus pair $\kX$ and every $\alpha\colon Y\to \omega(\kX)\in \Cor(k)$, by \cite[Theorem 1.6.2]{MotModulusII} there exists a proper modulus pair $\kY'\in \Comp(Y)$ and $\alpha'\colon \kY'\to \kX\in \MCor(k)$ such that $\alpha=\omega(\alpha')$. In particular, since $\kY\in \MCor(k)_{\leq n}$, the system $\{F(Y)\}$ for $Y\in \Cor(k)_{\leq n}$ running over the maps $Y \to \omega(X)$ is cofinal in the system $\{F(\omega(\kY))\}$ for $\kY\in \MCor(k)_{\leq n}$ running over the maps $\kY \to \kX$. Hence we have that
    \begin{align*}
	\omega^*\sigma_{n,!}^V F(\kX) = \sigma_{n,!}^V F(\omega(\kX))  & = \colim_{\substack{(\omega(\kX)\to Y)\\ Y\in \Cor(k)_{\leq n}}} F(Y) \\ & = \colim_{\substack{(\kX\to \kY')\\Y\in \Cor(k)_{\leq n}}} F(\omega(\kY'))\\ & = \colim_{\substack{(\kX\to \kY)\\\kY\in \MCor(k)_{\leq n}}} F(\omega(\kY)) = \sigma_{n,!} \omega^*F(\kX).
	\end{align*}
\end{remark}

\begin{remark}\label{rmk;commutecolimits}
	The functors $\sigma_{n,!}$ commute with colimits of presheaves since they are left adjoint. The functor $\sigma_n^*$ also commutes with colimits of presheaves, since they are computed section-wise and $\sigma_n^*F(\kX) = F(\kX)$ for all $\kX\in \MCor(k)_{\leq n}$.
	In particular for all diagrams $\{F_i\}$ in $\MPST(k,\Lambda)$, $\sigma_{n,!}\sigma_n^*\colim F_i\cong \colim \sigma_{n,!}\sigma_n^* F_i$
\end{remark}
The following lemma is mutuated from \cite[Lemma 1.1.16]{Ayoub-BV}:
\begin{lemma}\label{lem:unit-maps-sigma-invertible}
	The unit map $\id \xrightarrow{\sim} \sigma_{n}^* \sigma_{n,!}$ is invertible.
\end{lemma}

\begin{defn} We say that $F\in \MPST(k,\Lambda)$ is \emph{$n$-generated} (resp. \emph{strongly $n$-generated}) if the counit map $\sigma_{n,!}\sigma_{n }^* F\to F$ is surjective (resp. an isomorphism).
\end{defn}
\begin{remark} If $F$ is \textit{$n$-generated} (resp. \textit{strongly $n$-generated}), then $\omega_!F$ is $n$-generated (resp. strongly $n$-generated) in the sense of \cite{Ayoub-BV}. Indeed, the functor $\omega_!$ is exact and $\omega_!\sigma_{n,!}\sigma_{n }^* F = \sigma_{n,!}^{V}\sigma_{n }^{*,V} \omega_! F$ by Remark \ref{rmk:comp-ABV}.
\end{remark}

For example, if $\kX = (\ol{X}, X_\infty)$ with {$\dim(\ol{X}) \leq n$,} then $\Z_{tr}(\kX)$ is strongly $n$-generated. 
The proof of the following Lemma is a diagram chase.
\begin{lemma}\label{lem:n-generated-closed-under-ext-and-quotients}Quotients and extensions of (strongly) $n$-generated sheaves are again (strongly) $n$-generated.
\end{lemma}
\begin{defn}\label{def;hetrec}
Let $F\in \CI(k,\Lambda)$. Following \cite[Definition 1.1.20]{Ayoub-BV}, we say that $F$ is an \textit{$n$-modulus presheaf} 
if the natural map
	\[\hetrec (\sigma_{n,!}\sigma_{n }^* F) \to \hetrec F =a_{\et}^V\omega_{\CI} F\]
	is an isomorphism of \'etale sheaves with transfers. Here, for any $G\in \MPST(k,\Lambda)$, we denote by $\hetrec(G)$ 
the \'etale sheaf with transfers $a_{\et}^V\omega_{\CI} h_0^{\bcube}G$, where $a_{\et}^V\colon \PST(k,\Lambda)\to \Shv^{\tr}_{\et}(k,\Lambda)$ is Voevodksy's \'etale sheafification functor and $\omega_{\CI}$ is the composition $\omega_! \circ i^{\bcube}$, where $i^{\bcube}$ is the inclusion $\CI(k,\Lambda)\to \MPST(k,\Lambda)$, which has a right adjoint. Notice that the functor $\hetrec$ is a composition of left adjoints, hence it commutes with all colimits.
	
	We write $\CI_{\leq n}(k,\Lambda)$ for the full subcategory of $n$-modulus presheaves.
\end{defn}
The following is identical to \cite[Remark 1.1.21]{Ayoub-BV}:
\begin{lemma}\label{lem:h0-nstrgen-is-motivic} Let $F\in \MPST(k,\Lambda)$ be strongly $n$-generated. Then $h_{0}^{\bcube}(F)$ is an $n$-modulus sheaf.
\end{lemma}
\begin{remark}
Notice that in this case, differently from \cite[Remark 1.1.21]{Ayoub-BV}, if $F$ is an $n$-modulus presheaf then it is not automatic that $F$ is the $h_{0}^{\bcube}$ of a strongly $n$-generated sheaf: we only know that the map \[
h_{0}^{\bcube}\sigma_{n,!}\sigma_n^*F \to F
\]
is an isomorphism after applying $a_{\et}^V\omega_{\CI}$.
\end{remark}


\begin{defn}
	We define the category of \emph{$n$-reciprocity sheaves} $\RSC_{\et,\leq n}(k,\Lambda)$ as the essential image of $\CI_{\leq n}(k,\Lambda)$ via the functor $a_{\et}^V\omega_{\CI}$. 
\end{defn}
The following result is immediate:

\begin{lemma}\label{lem;concrete-n-reciprocity}
Let $F\in \Shv^{\tr}_{\et}(k,\Lambda)$. The following are equivalent
\begin{itemize}
\item[(i)]
$F$ is an $n$-reciprocity sheaf;
\item[(ii)]
$F\cong \hetrec \sigma_{n,!}\sigma_{n }^* G$ for some $G\in \MPST(k,\Lambda)$.
\end{itemize}
Moreover, if the above conditions hold, we can take $G\in \CI_{\leq n}(k,\Lambda)$ in (ii).
\end{lemma} 

\begin{remark}
	If $F\in \RSC_{\et,\leq n}(k,\Lambda)$, then $F$ is an $n$-generated \'etale sheaf in the sense of \cite{Ayoub-BV}. Indeed, for $F=\hetrec \sigma_{n,!}\sigma_{n }^* G$ with $G\in  \CI_{\leq n}(k,\Lambda)$, then there is a surjective map in $\Shv^{\tr}_{\et}(k,\Lambda)$
	\[
	a_{\et}^V\sigma^{V}_{n,!}\sigma_{n }^{V,*}(F) = a_{\et}^V\omega_! \sigma_{n,!}\sigma_{n }^* G\twoheadrightarrow a_{\et}^{V}\omega_{\CI} \hetcube\sigma_{n,!}\sigma_{n }^* G \cong \hetrec \sigma_{n,!}\sigma_{n }^* G = F.
	\]
\end{remark}

Recall that the category $\HI_{\et,\leq n}(k,\Lambda)$ of \cite[Definition 1.2.20]{Ayoub-BV}, is the full subcategory of $\HI_{\et}(k,\Lambda)$ of objects $F$ such that 
$a^V_{\et}h_0^{\A^1}\sigma^V_{n,!}\sigma_{n}^{V,*}F \to a^V_{\et}h_0^{\A^1}F=F$ is an isomorphism. 

\begin{prop}\label{prop;HI-n-sub-RSC-n} If $F\in \HI_{\et,\leq n}(k,\Lambda)$, then $F\in \RSC_{\et,\leq n}(k,\Lambda)$.
\begin{proof}
Take $F\in \HI_{\et,\leq n}(k,\Lambda)$.
By Remark \ref{rmk:comp-ABV} we have
\begin{equation}\label{eq1;prop;HI-n-sub-RSC-n}
	\hetrec \sigma_{n,!}\sigma_{n }^* \omega^*F = \hetrec\omega^*\sigma^V_{n,!}\sigma_{n}^{V,*}F.
	\end{equation}

	Notice that $h_{0}^{\bcube}\omega^* = \omega^*h_{0}^{\A^1}$, 
so by 
full faithfulness of $\omega^*$ we conclude that\[
	\hetrec\omega^*\sigma^V_{n,!}\sigma_{n}^{V,*}F = a^V_{\et}\omega_!h_0^{\bcube}\omega^*\sigma^V_{n,!}\sigma_{n}^{V,*}F =  
a^V_{\et}\omega_!\omega^*h_0^{\A^1}\sigma^V_{n,!}\sigma_{n}^{V,*}F = a^V_{\et}h_0^{\A^1}\sigma^V_{n,!}\sigma_{n}^{V,*}F = F.
	\]
In view of \eqref{eq1;prop;HI-n-sub-RSC-n} and Lemma \ref{lem;concrete-n-reciprocity}, this implies $F\in \RSC_{\et,\leq n}(k,\Lambda)$.
\end{proof}
\end{prop}

The following lemma is analogue to \cite[Lemma 1.1.23]{Ayoub-BV}
\begin{lemma}\label{lem;isonmotivic}
	For any $G \in \CI(k,\Lambda)$, the natural map 
\begin{equation}\label{eq;isonmotivic}\sigma_{n }^{V,*}\hetrec(\sigma_{n,!}\sigma_{n }^* G) \to \sigma_{n }^{V,*}a^V_{\et}\omega_{\CI}G
	\end{equation}
induced by the counit map $\sigma_{n,!}\sigma_{n }^*G \to G$ is an isomorphism in $\Shv^{\tr}_{\et}(k_{\leq n},\Lambda)$
\end{lemma}

\begin{corollary}[cfr. {\cite[Corollary 1.1.26]{Ayoub-BV}}]\label{cor:Key-corollary}
Let $F\in \RSC(k,\Lambda)$ such that $F=\omega_{\CI}G$ with $G\in \CI(k,\Lambda)$ and consider the natural map 
\[\hetrec (\sigma_{n,!}\sigma_{n }^* G)= 
a_{\et}^V \omega_{\CI}h_{0}^{\bcube}\sigma_{n,!}\sigma_{n }^* G
\to a_{\et}^V \omega_{\CI}h_{0}^{\bcube} G =a_{\et}^V \omega_{\CI} G =a_{\et}^VF\]
induced by the counit map $\sigma_{n,!}\sigma_{n }^*G\to G$.
Let $N$ be the kernel of the above map. 
If $N$ is an $n$-generated \'etale sheaf in the sense of \cite{Ayoub-BV}, then it is zero.
\end{corollary}
\begin{proof}
Since the functor $\sigma_{n }^{V,*}$ is exact, we have by the definition of $N$ an exact sequence of \'etale sheaves:
\[ 0\to \sigma_{n }^{V,*}(N) \to \sigma_{n }^{V,*}\hetrec(\sigma_{n,!}\sigma_{n }^* G) \to \sigma_{n }^{V,*}a_{\et}^VF, \]
hence by Lemma \ref{lem;isonmotivic} we conclude $\sigma_{n }^{V,*}(N) = 0$. Since $N$ is $n$-generated, we have a surjective map of \'etale sheaves with transfers $\sigma_{n,!}^{V} \sigma_{n }^{V,*}(N)  \to N$, showing that $N =0$. \end{proof}

\subsection{\texorpdfstring{$0$}{Zero}-reciprocity sheaves}\label{ssec:zero-mot-sheaves}

We specialize the general results of the previous section to the case $n=0$. By definition, the objects of the category $\MCor_{\leq 0}$ of smooth modulus pairs of dimension $\leq 0$ are the finite \'etale extensions $\ell\supset k$ (with empty modulus divisor). 
The essential image of the restriction of the functor $\omega\colon \MCor\to \Cor$ to $\MCor_{\leq 0}$ induces an equivalence of categories 
\[\omega_{|{\leq 0}} \colon \MCor_{\leq 0} \simeq \Cor_{\leq 0}\]
whose inverse is given by \[\lambda_{|\leq 0}\colon\Spec(\ell)\mapsto (\Spec(\ell),\emptyset).\] 
and induces an equivalence of categories
\begin{equation}\label{eq;mpst-equals-pst}
\omega_{|\leq 0,!}\colon \ulMPST(k_{\leq 0},\Lambda) \simeq \PST(k_{\leq 0},\Lambda)\colon \lambda_{|\leq 0,!}.
\end{equation}
Moreover, for any $\kX\in \MCor$ with $X=\omega(\kX)$, we have 
\[\MCor(\kX, (\Spec(\ell), \emptyset)) \cong \Cor(\omega(\kX), \Spec(\ell)) \cong \Cor(\pi_0(X), \Spec (\ell)) \cong \mathbb{Z}^{\pi_0(|X\tensor_k \ell|)},\]
where $\pi_0(X)$ is the spectrum of the integral closure of $k$ in $\Gamma(X,\mathcal{O}_X)$ and $\pi_0(|Y|)$ for  a scheme $Y$ denotes the set of connected components of the underlying topological space $|Y|$. The first isomorphism follows from the fact that $\Spec(\ell)\in \Cor^{\rm prop}$ (see \cite[Lemma 1.5.1]{MotModulusI}),
while the second and the third are classical (see \cite[Lecture 1]{MVW} and \cite[1.2.1]{Ayoub-BV}). From this we get an adjunction
\begin{equation} \label{eq:pi0sigma0} \lambda_{|\leq 0}\circ\pi_0\circ\omega\colon \MCor   \leftrightarrows \MCor_{\leq 0}\colon \sigma_{0} .
\end{equation}
We let $\Pi_0$ denote $\lambda_{|\leq 0}\circ\pi_0\circ\omega$.    
Passing to the categories of presheaves, we have that the functor $\sigma_{0,!}$ of \eqref{eq;sigma} has a left adjoint:
\begin{equation} \label{eq:pi0sigma0adjunction}  \Pi_{0,!}\colon \MPST(k,\Lambda) \leftrightarrows \MPST(k_{\leq 0},\Lambda)  \colon \Pi_0^*=\sigma_{0,!} .
\end{equation}
In particular, $\sigma_{0,!}$ is given explicitly by    
\begin{equation}\label{eq;explicit-sigma-0}
	\sigma_{0,!} (F)(\kX) = F(\Pi_0(\kX)) = F(\pi_0(\omega(\kX)),\emptyset)\;\text{ for } F\in \MPST(k_{\leq 0},\Lambda).
\end{equation}
The following Corollary shows that the category of $0$-reciprocity sheaves is simply equivalent to the category of $0$-motivic sheaves in the sense of Ayoub--Barbieri-Viale.
\begin{cor}\label{cor:CI_0-equiv-MEST0} Let $G\in \MPST(k,\Lambda)$, then $\sigma_{0,!}\sigma_0^*G\in \CI(k,\Lambda)$. In particular, every $0$-reciprocity sheaf is strongly $0$-generated in the sense of \cite{Ayoub-BV}, and we have equivalences
	\begin{equation}\label{eq;0-gen-equivalent}
	\HI_{\et,\leq 0}(k,\Lambda)\simeq \RSC_{\et,\leq 0}(k,\Lambda)\simeq \Sh_{\et}^{\rm tr}(k_{\leq 0},\Lambda).
	\end{equation}
	\begin{proof} 
		Let $\kX\in \MCor$ and $X=\omega(\kX)$. By\eqref{eq;explicit-sigma-0}, we have that\[
		\sigma_{0,!}\sigma_0^*G(\kX\otimes \P^1) = \sigma_0^*G(\pi_0(\omega(\kX\otimes \P^1)),\emptyset).
		\]
		On the other hand, we have that $\omega(\kX\otimes \P^1)=X\times \A^1$ and as observed in the proof of \cite[Lemma 1.2.2]{Ayoub-BV}, we have that $\pi_0(X\times \A^1)=\pi_0(X)$, hence\[
		\sigma_{0,!}\sigma_0^*G(\kX\otimes \P^1)=\sigma_0^*G(\pi_0(\omega(\kX\otimes \P^1)),\emptyset)=\sigma_0^*G(\pi_0(\omega(\kX)),\emptyset)=\sigma_{0,!}\sigma_0^*G(\kX).
		\]
		Hence $\sigma_{0,!}\sigma_0^*G\in \CI(k,\Lambda)$ proving the first assertion. 
Let $F\in \RSC_{\et,\leq 0}(k,\Lambda)$ and let $F =\hetrec\sigma_{0,!}\sigma_0^*G'$ with 
$G'\in \CI(k,\Lambda)$ be as in Lemma \ref{lem;concrete-n-reciprocity}(ii). 
Then $\sigma_{0,!}\sigma_0^*G'\in \CI(k,\Lambda)$,
so
\[F = a_{\et}^V\omega_{\CI}h_0^{\bcube}\sigma_{0,!}\sigma_0^*G' \simeq  a_{\et}^V\omega_!\sigma_{0,!}\sigma_0^*G' = a_{\et}^V\sigma_{0,!}^V\sigma_0^{V,*}\omega_!G'.
		\]
		This implies that $F\in \Sh_{\et}^{\rm tr}(k_{\leq 0},\Lambda)$. On the other hand, by Proposition \ref{prop;HI-n-sub-RSC-n} we have\[
		\begin{tikzcd}
		\HI_{\et,\leq 0}(k,\Lambda)\ar[r,hook]&\RSC_{\et,\leq 0}(k,\Lambda)\ar[r,hook] &\Sh_{\et}^{\rm tr}(k_{\leq 0},\Lambda).
		\end{tikzcd}
		\]
		By \cite[Lemma 1.2.2]{Ayoub-BV}, the composition above is an equivalence, hence we deduce the equivalences of \eqref{eq;0-gen-equivalent}.
	\end{proof}
\end{cor}

\subsection{\texorpdfstring{$1$}{One}-reciprocity sheaves and Albanese functors} \label{ssec:MotAlbFunctor} In this subsection, we will assume that $k$ has characteristic zero. In particular, $k$ satisfies resolutions of singularities and for any smooth proper modulus pair, we can identify $\mathbf{Alb}_{\kX}^{\CH}$ with $\mathbf{Alb}_{\kX}^{\Omega}$, and we simply write $\mathbf{Alb}_{\kX}$ for the Albanese scheme of $\kX$.

\begin{remark}\label{rmk:h0-curve-is-1-motivic}
	If $\kC = (\ol{C}, C_\infty)$ denotes a $1$-dimensional smooth and proper modulus pair with $\ol{C}$ geometrically integral, by Lemma \ref{lem:h0-nstrgen-is-motivic}, $\hetrec(\kC)$ is a $1$-reciprocity sheaf. Moreover, as observed in \eqref{eq;pic}, we have:
	\[
	\hetrec(\kC) =a_{\et}^V\omega_{\CI} h_0^{\bcube}(\Z_{tr}(\kC))= a_{\et}^V h_0(\kC) = \uPic(\ol{C}, C_\infty),
	\]
	where $\uPic(\ol{C}, C_\infty)$ denotes the relative Picard group scheme, whose connected component  of the identity $\uPic^0(\ol{C}, C_\infty)$ agrees with the Rosenlicht-Serre generalized Jacobian ${\rm Jac}(\ol{C}, C_\infty)$. By Proposition \ref{prop:UniversalPropertyPST} and the proof of Proposition \ref{prop:AlbaneseChow}, we have $\uPic(\ol{C}, C_\infty) = {{\bf Alb}}_{\kC}$. Hence we finally have that $\hetrec (\kC) = {{\bf Alb}}_{\kC}$ is represented by a commutative group scheme.
\end{remark}

\begin{remark}\label{rmk:Alb-1-generated}
	More generally, let $\kX$ be a smooth and proper modulus pair. Then the sheaf $\mathbf{Alb}_{\kX}$ is $1$-generated. In fact, $\mathbf{Alb}_{\kX}$ can be written as extension of a semi-abelian $k$-group scheme (i.e. a $k$-group scheme $G$ such that $G^0$ is a semi-abelian variety) and a unipotent algebraic group $U$. In characteristic $0$, the group $U$ is a product of $\G_a$, and $\G_a$ is a direct summand of $\hetrec (\mathbb{P}^1, 2\infty)$, so that it is $1$-generated. The semi-abelian $k$-group scheme $G$ is a quotient of the generalized Jacobian of a suitable curve contained in $G$ by a theorem of Matsusaka \cite{Matsusaka}, which is $1$-generated by the previous remark (see \cite[1.3]{Ayoub-BV}). Hence $G$ itself is $1$-generated by \cite[Lemma 1.1.15]{Ayoub-BV}. Applying again \cite[Lemma 1.1.15]{Ayoub-BV} to $\mathbf{Alb}_{\kX}$ we get the statement.
\end{remark} 

    We deduce from the previous remarks the following analogue to \cite[Lemma 1.3.4]{Ayoub-BV}.
\begin{lemma}\label{lem:sub-1-gen-is-1-gen} Let $F\in \RSC_{\et,\leq 1}(k,\Lambda)$. Then any subsheaf of $F$ is a $1$-generated \'etale sheaf in the sense of \cite{Ayoub-BV}.
	\begin{proof}We essentially follow the steps in the proof of \cite[Lemma 1.3.4]{Ayoub-BV}, starting from the case of $F= \hetrec(\kC)$, for $\kC$ a smooth and proper modulus pair of dimension $1$. This is a $1$-reciprocity sheaf by Remark \ref{rmk:h0-curve-is-1-motivic}. Let $E\subset F$ be a subsheaf. 
		Since colimits of $1$-generated \'etale sheaves are $1$-generated \'etale sheaves, 
		(see \ref{rmk;commutecolimits}), we can assume that $E$ is the image of a map $a\colon \Lambda_{tr}(X)\to \hetrec(\kC)$, for $X\in \Sm$. Since for $k\subseteq k'$ a finite extension, the map $X_{k'}\to X$ is an \'etale cover: this implies that $ \Lambda_{tr}(X_{k'})\to  \Lambda_{tr}(X)$ is a surjective map of \'etale sheaves, so we can assume $X$ geom. connected. By Remark \ref{rmk:h0-curve-is-1-motivic}, we have then a map
		\[a\colon \Lambda_{tr}(X) \to {\bf Alb}_{\kC} \cong \hetrec (\kC). \]
		Since ${\bf Alb}_{\kC}\in \RSC(k,\Lambda)$, by Remark \ref{rmk;reciprocity-only-(n)} there exists a smooth proper modulus pair $\kX$ with $X=\kX^\circ$ such that $a$ factors through $h_0(\kX)$, and by Proposition \ref{prop:UniversalPropertyPST}, it uniquely factors through $a_{\kX}\colon h_0(\kX)\to \mathbf{Alb}_{\kX}$:
		\[\xymatrix{ h_0(\kX)\ar[d]_{a_{\kX}} \ar[dr]^{a}& \\
			\mathbf{Alb}_{\kX}\ar[r]^{a'} & {\bf Alb}_{\kC}.
		}\]
		By Lemma \ref{lem:Alb-map-surjective}, the motivic Albanese map $a_{\kX}$ is a surjective morphism of \'etale sheaves, hence the image of $a$ agrees with the image of $a'$. Thus $E = {\rm Im}(a')$ is a $1$-generated \'etale sheaf by \cite[Lemma 1.1.15]{Ayoub-BV} and Remark \ref{rmk:Alb-1-generated}.
		
		Now the general case. Let $E\subset F \cong \hetrec (\sigma_{1,!}\sigma_{1}^* G)$ with $G\in \MPST(k,\Lambda)$ (cf. Lemma \ref{lem;concrete-n-reciprocity}(ii)).
The sheaf $\sigma_{1}^* G \in \MPST(k_{\leq 1},\Lambda)$ can be written as colimit of representable sheaves  in $\MPST(k_{\leq 1},\Lambda)$, i.e.
		\[\sigma_{1}^* G  = \colim_{\kC\to\sigma_{1}^* G} \Z_{tr}(\kC)_{\leq 1}, \quad \kC = (\ol{C}, C_\infty),\quad \dim(\ol{C}) = 1.\]
		Since $\sigma_{1,!}$ and $\hetrec$ commute with colimits (being left adjoints), we have then  
		\begin{equation}\label{eq:F-colimit-curves}
			F \cong \hetrec (\sigma_{1,!} (\sigma_{1}^* G)) = \colim \hetrec (\sigma_{1,!} \Z_{tr}(\kC)_{\leq 1} ) = \colim \hetrec(\kC) .
		\end{equation}
		For any $\kC\in (\kC\to\sigma_{1}^* G)_{\leq 1}$, we have in particular a map $\hetrec(\kC) \to F$, and thus a map from the fiber product
		\[\colim_{\kC\to\sigma_{1, *} F } (\hetrec(\kC)\times_F E ) \to E \]
		which is surjective (the proof of surjectivity is formal and identical to the corresponding statement in the proof of \cite[Lemma 1.3.4]{Ayoub-BV}). Now it is enough to notice that each $\hetrec(\kC)\times_F E  \subset \hetrec(\kC)$ is a $1$-generated \'etale sheaf by the previous step and the fact that $1$-generated \'etale sheaves are stable by colimits. To conclude we apply again \cite[Lemma 1.1.15]{Ayoub-BV}.
	\end{proof}
\end{lemma}


\begin{prop}\label{prop:CI1-thick} Let $F\in \Shv_{\et}^{\tr}(k,\Q)$ be an \'etale sheaf of $\Q$-vector spaces which is $1$-generated in the sense of \cite{Ayoub-BV}. If it is a reciprocity sheaf, then it is a $1$-reciprocity sheaf. In particular, any subsheaf of a $1$-reciprocity sheaf of $\Q$-vector spaces is again a $1$-reciprocity sheaf and the category $\RSC_{\et, \leq 1}(k,\Q)$ is closed under taking subobjects, {colimits} and extensions in $\RSC_{\et}(k,\Q)$.
	\begin{proof}By Lemma \ref{lem:sub-1-gen-is-1-gen}, any subsheaf of a $1$-reciprocity sheaf is again $1$-generated, and by \cite[Corollary 2.4.2]{KSY-RecII} any subsheaf of a reciprocity sheaf is a reciprocity sheaf. Then the second part of the Proposition follows from the first, since $\RSC_{\et}(k,\Q)$ is an abelian category {stable by colimits in $\Shv_{\et}^{\tr}(k,\Q)$} (here we are using the fact that we consider $\Q$-coefficients in order to exploit Proposition \ref{prop:et-Nis-Voe}, since \cite[Corollary 2.4.2]{KSY-RecII} is a statement about the Nisnevich sheafification) and $1$-generated \'etale sheaves are stable by colimits and extensions by \cite[Lemma 1.1.15]{Ayoub-BV}.

	We now prove the first assertion. Let $F\in \Shv_{\et}^{\tr}(k,\Q)$ be a $1$-generated \'etale sheaf of $\Q$-vector spaces and suppose that $F\in \RSC(k,\Q)$, i.e. that there exists $G\in \CI(k,\Q)$ such that $F=\omega_!G$. By Remark \ref{rmk:comp-ABV}, we have that\[
	\sigma^V_{1,!}\sigma_1^{V,*}F = \omega_!\sigma_{1,!}\sigma_1^*G
	\]
	and the counit $\sigma^V_{1,!}\sigma_1^{V,*}F\to F$ is the image via $\omega_!$ of the counit $\sigma_{1,!}\sigma_1^*G\to G$. Since $G\in \CI(k,\Q)$, the map $\sigma_{1,!}\sigma_1^*G\to G$ factors through $h_0^{\bcube}\sigma_{1,!}\sigma_1^*G$, which induces a factorization\[
	\begin{tikzcd}
	a^V_{\et}\omega_!\sigma_{1,!}\sigma_1^*G\ar[r,equal]\ar[drr,bend right = 10]&a^V_{\et}\sigma^V_{1,!}\sigma_1^{V,*}F\ar[rr,"(*)"]&&F\\
	&&\hetrec (\sigma_{1,!}\sigma_1^*G).\ar[ur]
	\end{tikzcd}	
	\]
Since $F$ is a $1$-generated \'etale sheaf, the map $(*)$ is surjective, hence the induced map $\hetrec (\sigma_{1,!}\sigma_1^*G)\to F$ is surjective. 
Let $N=\ker(\hetrec (\sigma_{1,!}\sigma_1^*G)\to F)$.
By Lemma \ref{lem;concrete-n-reciprocity}, $\hetrec (\sigma_{1,!} \sigma_{1}^* G)\in \RSC_{\et, \leq 1}(k,\Q)$ so that $N$ is a $1$-generated \'etale sheaf by Lemma \ref{lem:sub-1-gen-is-1-gen}. Hence $N=0$ by Corollary \ref{cor:Key-corollary} so that $\hetrec (\sigma_{1,!}\sigma_1^*G)\cong F$, which concludes the proof by Lemma \ref{lem;concrete-n-reciprocity}.
	\end{proof}
\end{prop}

Recall that $\RSC_{\et}(k,\Q)$ is a Grothendieck abelian category by \cite[Corollary 2.4.2]{KSY-RecII}. The following corollary is immediate from the previous proposition. 
\begin{cor}\label{cor:CI1-is-Grothendieck}The inclusion $\RSC_{\et,\leq 1}(k,\Q) \subset \RSC(k,\Q)$ (and consequently \par
$\RSC_{\et,\leq 1}(k,\Q)\subseteq \Shv_{\et}^{\tr}(k,\Q)$) is exact, and the category $\RSC_{\et,\leq1}(k,\Q)$ is a Grothendieck abelian category (in particular, it has enough injectives).
\end{cor}
The proof of the following Lemma is identical to the proof of \cite[Lemma 1.3.6]{Ayoub-BV}, using Lemma \ref{lem:sub-1-gen-is-1-gen}.
\begin{lemma}\label{lem:subsheaf-of-G-is-1-subgroup}Let $G\in \sG^*$ be a smooth commutative $k$-group scheme and let $F$ be an \'etale subsheaf of $G$ with transfers such that its sheaf of connected components $\pi_0(F)$ is zero. Then $F$ is represented by a closed subgroup of $G$.
	\end{lemma}

\begin{defn}\label{defn:fin-pres1mot}  A $1$-reciprocity sheaf $F\in \RSC_{\et, \leq 1}(k,\Q)$ is called \emph{finitely generated} if there exists a commutative $k$-group scheme $G\in \sG^*$ with $\pi_0(G)$ finitely generated together with a surjection $q\colon G\otimes_\Z \Q\to F$. 
If the kernel of $q$ is itself finitely generated, we say that $F$ is \emph{finitely presented}.
\end{defn}	
We write $\RSC_{\et, \leq 1}^{\star}(k,\Q)\subset \RSC_{\et, \leq1}(k,\Q)$ for the full subcategory of finitely presented $1$-reciprocity sheaves.  	
An almost word-by-word translation of \cite[Proposition 1.3.8]{Ayoub-BV}, using Lemma \ref{lem:sub-1-gen-is-1-gen} and Lemma \ref{lem:subsheaf-of-G-is-1-subgroup} gives the following canonical presentation of every finitely presented $1$-reciprocity sheaf. This result will be repeatedly used in the rest of the paper. 
\begin{prop}\label{prop:can-pres-1-mot-sheaf} Any $1$-reciprocity sheaf is filtered colimit of finitely presented $1$-reciprocity sheaves. If $F$ is a finitely presented $1$-reciprocity sheaf, then there is a unique and functorial exact sequence
	\[ 0\to L\to  G\otimes_{\Z}\Q \to F\to 0\]
	where $G\in \sG^*$ is a smooth commutative $k$-group scheme such that $\pi_0(G)$ is finitely generated and $L\in \RSC_{\et, \leq0}(k,\Q)  \cong \Shv_{\et}^{\tr}(k_{\leq0},\Q)$ is a finitely generated $0$-motivic sheaf.
\end{prop}

\begin{remark}\label{rmk:Fsab}
	For a smooth commutative $k$-group scheme $G\in \sG^*$, write $G^{sab}$ for its semi-abelian quotient: it is a smooth commutative $k$-group scheme such that the connected component of the identity $(G^{sab})^0$ is a semi-abelian variety. We have by Chevalley's theorem an extension 
	\begin{equation}\label{eq:Chev}0\to U\to G\to G^{sab}\to 0\end{equation}
	where $U$ is a unipotent group. Since $k$ is of characteristic zero, $U\cong \G_a^{r}$ where $r$ is the \textit{unipotent rank} of $G$. 
	More generally, for any $F$ finitely presented $1$-motivic sheaf, by Proposition \ref{prop:can-pres-1-mot-sheaf} we have a functorial commutative diagram
	\begin{equation} \label{eq:finpresA1}
		\begin{tikzcd}	
			0\ar[r] &L \ar[r]\ar[rd,"\varphi"] & G\otimes_\Z\Q \ar[r] \ar[d] & F\ar[d] \ar[r]&0\\
			&& G^{sab}\otimes_\Z\Q \ar[r] & F^{\A^1} \ar[r]&0
		\end{tikzcd}
	\end{equation}
	where $F^{\A^1}:=h_{0,\et}^{\A^1}(F)$, and since $\ker(G\to G^{sab})$ is unipotent, {by \cite[Example 2.23]{MVW}} we also conclude that $C_{*}^{\A^1}(F)\simeq F^{\A^1}[0]$, where $C_{*}^{\A^1}(F)$ is the Suslin complex. Moreover, by the right exactness of $h_{0,\et}^{\A^1}$, the kernel of the map $G^{sab}\otimes_{\Z}\Q\to F^{\A^1}$ is a quotient of the lattice $L$, hence it is itself a lattice.
	
	In general, if $F=\colim F_1$ is a $1$-reciprocity sheaf with $F_i$ finitely presented, then $F^{\A^1}=\colim(F_i^{{A^1}})$ and $\ker(F\twoheadrightarrow F^{\A^1}) = \colim\ker(F_i\twoheadrightarrow F_i^{\A^1})$. By \eqref{eq:finpresA1}, for each $i$ there is a lattice $L'_i$ and a finite dimensional $k$-vector space $U_i$ such that \[\ker(F_i\twoheadrightarrow F_i^{\A^1})\cong (U_i\otimes \G_a)/L'_i.\]
	Since $\colim (U_i\otimes_k\G_a)\cong (\colim U_i)\otimes_k\G_a$, which is (non canonically) isomorphic to a (possibly infinite) sum of $\G_a$, and $\mathcal{L}=\colim L_i$ is a $0$-motivic sheaf, we have an exact sequence\[
		0\to \mathcal{L}\to \oplus \G_a\to  \ker(F_i\twoheadrightarrow F_i^{\A^1})\to 0.
	\]
Thus, for every $1$-reciprocity sheaf $F$, we have an exact sequence
\begin{equation}\label{eq:Ga-A1}
0\to (\oplus \G_a)/ \mathcal{L} \to F \to F^{\A^1}\to 0
\end{equation}
describing $F$ as an extension of an $\A^1$-invariant sheaf by a quotient of a (possibly infinite dimensional) vector group by a $0$-motivic sheaf. 
\end{remark}

Recall that by \cite[Prop. 11.1]{Isaksenmodel}, if $\cC$ has all colimits, then $\pro\cC$ has all colimits. We can now  prove the following generalization  to \cite[Proposition 1.3.11]{Ayoub-BV}. 

\begin{thm}\label{thm:Alb} Let $\ch(k)=0$.
The embedding $\RSC_{\et,\leq 1}(k,\Q) \subseteq \Shv_{\et}^{\tr}(k,\Q)$ has a pro-left adjoint:
	\[\Alb\colon \Shv_{\et}^{\tr}(k,\Q)\to \pro\RSC_{\et,\leq 1}(k,\Q)\]
induced by colimit from \[
\Q_{\tr}(X)\mapsto  ``\lim_{n}"(\bAlb_{\kX^{(n)}}\otimes_{\Z}\Q)
\]
for any choice of $\kX\in \Comp(X)$ smooth. 
\end{thm}

\begin{proof}
    For $X\in \Sm(k)$, recall from Definition \ref{defn;comp-modulus} the cofiltered category of Cartier compactifications $\mathbf{Comp}(X)$.
	It is enough to show that for any $X$, for any  choice of $\kX\in \Comp(X)$ with total space smooth (since $\ch(k)=0$, such choice exists), and for any $E\in \RSC_{\et, \leq1}(k,\Q)$, we have 
	\[\colim_{n}\Hom_{\Shv^{\tr}_{\et}(k)}(\bAlb_{\kX^{(n)}}\otimes_{\Z}\Q, E) = \Hom_{\Shv^{\tr}_{\et}(k)}(\Q_{\tr}(X),E).\]
	By Galois descent  it is also enough to prove the claim for $k$ algebraically closed. By Lemma \ref{lem:Alb-map-surjective}, for any $\kX^{(n)}$ {the Albanese map $a_{\kX}\colon \Q_{tr}(X)\to \bAlb_{\kX^{(n)}}\otimes_\Z\Q$ is a surjective morphism of $\Shv^{\tr}_{\et}(k,\Q)$}, hence since filtered colimits are exact, we only need to show that 
	\[\colim_{n}\Hom_{\Shv^{\tr}_{\et}(k)}(\bAlb_{\kX^{(n)}}\otimes_{\Z}\Q, E) \to \Hom_{\Shv^{\tr}_{\et}(k)}(\Q_{\tr}(X),E)\]
	is surjective, i.e. that for any choice of $\kX\in \Comp(X)$, every map $\Q_{\tr}(X)\to E$ factors through $\bAlb_{\kX^{(n)}}\otimes_{\Z}\Q$ for some $n$.  
	Since $\Q_{\tr}(X)$ is compact in $\Shv^{\tr}_{\et}(k,\Q)$, by Proposition \ref{prop:can-pres-1-mot-sheaf} we can suppose that $E$ is finitely presented. Let $E=\coker(L\hookrightarrow G\otimes_{\Z}\Q)$  be as in Proposition \ref{prop:can-pres-1-mot-sheaf}. 
	Then we get a long exact sequence 
	
	\[\ldots\to \Hom(\Q_{tr}(X), G\otimes_\Z\Q)\to \Hom(\Q_{tr}(X), E)\to 
	H^1_{\et}(X, L)\to \ldots.\]
	Since $L\cong \Q^r$, being $k$ separably closed,
	$H^1_{\et}(X, L) = H^1_{\Nis}(X, \Q^r) =0$. Thus the map \[\Hom(\Q_{tr}(X), G)\to \Hom(\Q_{tr}(X), E)\] is surjective, i.e. every map $s\colon \Q_{tr}(X)\to E$ factors via $\sigma\colon \Q_{\tr}(X)\to G\otimes_{\Z}\Q$. By Remark \ref{rmk;reciprocity-only-(n)} (iii) for every $\sigma$ as above and for any choice of $\kX\in \Comp(X)$, there exists $n$ such that $s$ factors through $\hetrec(\kX^{(n)})$. Considering now $G$ with integral coefficients, we deduce that there exists $m$ such that $m\sigma$ defines a section of $G$ over $X$ with modulus $\kX^{(n)}$ {as in \ref{ssec:univ-prop}}, hence by Proposition \ref{prop:UniversalPropertyPST}, $m\sigma$ factors through $\bAlb_{\kX^{(n)}}$, so $\sigma$ factors through $\bAlb_{\kX^{(n)}}\otimes_{\Z}\Q$, concluding the proof.
\end{proof}
\subsection{The log Albanese functor}\label{omegaCIet} We now generalize this to the logarithmic setting. As in Definition \ref{defn;comp-modulus}, we introduce the following
\begin{defn}\label{defn;comp-modulus-log}
Let $X=(\ul{X},\partial X)\in \SmlSm(k)$ and $|\partial X|_{\red}$ be the reduced divisor on $\ul{X}$ associated to $\partial X$. We let $\Comp(X)$ be the cofiltered category formed by the modulus pairs $(\overline{X},\widetilde{|\partial X|}_{\red} + D)$, where $(\overline{X},D)\in \Comp(\ul{X})$ and $\widetilde{|\partial X|}_{\red}\subset \ol{X}$ is an effective Cartier divisor supported on the closure of $|\partial X|_{\red}$ in $\ol{X}$ such that $\widetilde{|\partial X|}_{\red}\times_{\ol{X}} \ul{X}=|\partial X|_{\red}$. 
\end{defn} 
Let \[
\Log_{\det}\colon \colon \RSC_{\et}(k,\Q)\to \Shv_{\textrm{d\'et}}(k,\Q)
\]
be the composition  in \eqref{eq;shuji-functor2}, which is fully faithful and exact.
As observed in  Remark \ref{rmk;omegaCI-commutes-with-colimits}, $\Log_{\det}$ is fully faithful, exact and commutes with all colimits. 
\begin{thm}\label{thm:logAlb}Let $\ch(k)=0$. The fully faithful exact functor 
	\[\RSC_{\et,\leq 1}(k,\Q) \subseteq \RSC_{\et}(k,\Q) \xrightarrow{\Log_{\det}}\mathbf{Shv}_{\mathrm{d}{\et}}^{\mathrm{ltr}}(k,\Q)\]
	has a pro-left adjoint, called the log Albanese functor 
	\[\Alb^{\log}\colon \mathbf{Shv}_{\mathrm{d}\et}^{\mathrm{ltr}}(k,\Q)\to \textrm{pro-}\RSC_{\et,\leq 1}(k,\Q).\]
	induced by colimit from \[
		\Q_{\tr}(X)\mapsto  ``\lim_{n}"(\bAlb_{\kX^{(n)}}\otimes_{\Z}\Q)
	\]
	for any choice of $\kX\in \Comp(X)$ smooth.

	\begin{proof} We proceed as in the non-log case. Since $\Log$ commutes with filtered colimits, we can reduce to prove the adjunction for maps against $E$ finitely presented, quotient of a smooth commutative group $k$-scheme $G\otimes_{\Z}\Q$ by a lattice $L$. As before, it is enough to prove the claim with   $k$ algebraically closed, by Galois descent. Since we are considering sheaves of $\Q$-vector spaces, we can assume that $L\cong\Q^r$. For  $X\in \SmlSm(k)$ we have\[
		H^1_{\textrm{d\'et}}(X,\Log_{\det}(L))\cong H^1_{\et}(X-|\partial X|,\Q^{r})=0.
		\] 
	So following the steps of the previous proof, it is enough to show that {for $G\in \sG^*$ and $X\in \SmlSm(k)$,} any map\[
	\sigma:\Q_{\ltr}(X)\to \Log_{\det}(G\otimes_\Z\Q)
	\]
	factors through $\Log_{\det}(\bAlb_{\kX^{(n)}}\otimes_{\Z}\Q)$ for some $n$, and conclude by full faithfulness of $\Log_{\det}$. 
{On the other hand, 
for $X=(\ul{X},\partial X)\in \SmlSm(k)$ let $\kX=(\overline{X},\widetilde{|\partial X|}_{\red} + D)\in \Comp(X)$ as in Definition \ref{defn;comp-modulus-log}, and for $n\geq 1$ let $\kX^{(n)}=(\ol{X},\widetilde{|\partial X|}_{\red} + nD)\in \Comp(X)$. We have
(see \S\ref{reciprocitysheaves} for the notation)
	\begin{align*}
	\Hom_{\Shv_{\det}^{\ltr}}(\Q_{\ltr}(X),\Log_{\det}(G\otimes_\Z\Q) 
&=\Log(\G\otimes_\Z\Q)(X)\\
&=	\colim_n\Hom_{\RSC_{\Nis}(k,\Z)}(\omega_!h_0^\bcube(\kX^{(n)}),G\otimes_\Z\Q) \\ 
&=	\colim_n\Hom_{\RSC_{\et}(k,\Q)}(\hetrec(\kX^{(n)}),G\otimes_\Z\Q) \\ 
&=\colim_n\Hom_{\logCI_{\det}(k,\Q)}(\Log_{\det}(\hetrec(\kX^{(n)})),\Log_{\det}(G\otimes_\Z\Q)),
	\end{align*}
where the first equality follows from Proposition \ref{prop:rat-coeff-nisequaletale},
the second follows from the definition of $\Log$ in \cite{shujilog} and Remark \ref{rmk;reciprocity-only-(n)}, the third from the isomorphsim \[
(a_{\et}^V\omega_!h_0^{\bcube}(\kX^{(n)}))\otimes_{\Z}\Q\cong a_{\et}^V\omega_!(h_0^{\bcube}(\kX^{(n)})\otimes_{\Z}\Q) = \hetrec(\kX)\quad \textrm{(see Def. \ref{def;hetrec})}
\]
and the fourth from the full faithfulness of $\Log_{\det}$. }
So, there exists $n$ such that $\sigma$ factors through $\Log_{\det}\hetrec(\kX^{(n)})$, hence again there exists $m$ such that $m\sigma$ is a section of $G$ with modulus $\kX^{(n)}$ {as in \ref{ssec:univ-prop}}, so by Proposition \ref{prop:UniversalPropertyPST} and full faithfulness of $\Log_{\det}$, it factors through $\Log_{\det}(\bAlb_{\kX^{(n)}}\otimes_{\Z}\Q)$, proving the claim. 
	\end{proof}
\end{thm}

\section{Extension of 1-reciprocity sheaves}

In this section, we prove some technical results about extensions in $\RSC_{\et,\leq 1}(k,\Q)$. First of all, recall the Breen--Deligne resolution, whose proof is contained in \cite[Appendix to Lecture IV]{condensed}: 
\begin{thm}[Eilenberg--Maclane, Breen, Deligne, Clausen--Scholze]\label{thm:Breen-Deligne}
	Let $G\in \sG^{*}$. Then there exists a complex of presheaves (the Breen--Deligne resolution) $C_\bullet(G)$ such that  $C_i(G) = \Z(G^{\times t_{i}})$ for  some integer $t_i>0$ together with an augmentation map $s\colon C(G)\to G$ in $\Cpx(\PSh(k,\Z))$ that is an equivalence in $\cD(\PSh(k,\Z))$. 
	Moreover, the first terms are computed as follows:\[
		\Z[G^{\times 3}\times G^{\times 2}]\xrightarrow{\partial_2}\Z[G\times G]\xrightarrow{\partial_1} \Z[G]\xrightarrow{\partial_0} G\to 0
	\]
	\begin{itemize}
		\item $\partial_0$ is the sum map,
		\item $\partial_1([x,y]) = [x+y]-[x]-[y]$
		\item $\partial_2([x,y,z,t,u]) = [x+y,z]-[x,y+z] - [y,z] + [x,y] + [t,u] - [u,t]$.
	\end{itemize}
\end{thm}
{\begin{remark}\label{rmk;G-star-compact}
	Every $G\in \sG^*$ is a compact object in $\PSh(k,\Z)$: since $G$ is an extension of $G^{0}$ and $\pi_0(G)$, it is enough to show that $G^{0}$ and $\pi_0(G)$ are compact. By our assumption, $\pi_0(G)$ is finitely generated, so it is compact. Since $G^{0}\in \Sm(k)$ by assumption, it is is compact in $\PSh(\Sm(k),\mathrm{Sets})$, and since the forgetful functor $\PSh(\Sm(k),\Z)\to \PSh(\Sm(k),\mathrm{Sets})$ preserves filtered colimits, $G^0$ is compact in $\PSh(\Sm(k),\Z)$ too.
Moreover, since the inclusion $\Sh_{\et}(k,\Z)\to \PSh(k,\Z)$ preserves filtered colimits, we have that $G$ is a compact object in $\Sh_{\et}(k,\Z)$ too.
\end{remark}}

\begin{cor}\label{cor:ext-et-fppf}
Let $G_1,G_2\in \sG^*$. Then\[
\Map_{\cD(\Sh_{\rm fppf}(k,\Z))}(G_1,G_2) \cong \Map_{\cD(\Sh_{\et}(k,\Z))}(G_1,G_2)
\]
\begin{proof}
	By the Breen--Deligne resolution, for $\tau\in \{\et,{\rm fppf}\}$ we have\[
	\Map_{\cD(\Sh_{\tau}(k,\Z))}(G_1,G_2) \simeq R\Gamma_{\tau}(G_1^{\times t_\bullet},G_2)
	\]
	and by a theorem of Grothendieck (see \cite[III, Theorem 3.9]{MilneEtCoh}):\[
	R\Gamma_{\et}(G_1^{\times t_\bullet},G_2)\simeq R\Gamma_{\rm fppf}(G_1^{\times t_\bullet},G_2),
	\]
	which allows us to conclude.
\end{proof}
\end{cor}
Recall the following result proved in \cite[Appendix B]{AHPL}:
\begin{cor}\label{cor:AHPL}
	Let $G\in \sG^*$ and let $F\in \Sh_{\et}^{\tr}(k,\Q)$. Let $i^{\rm tr}$ be the functor ``forget transfers''. Then\[
	\Map_{\cD(\Sh_{\et}(k,\Q))}(G\otimes_\Z\Q,i^{\rm tr}F) \cong \Map_{\cD(\Sh_{\et}^{\tr}(k,\Q))}(G\otimes_\Z\Q,F)
	\]
	In particular, $G\otimes_\Z\Q$ is a compact object in $\Sh_{\et}^{\tr}(k,\Q)$.
	\begin{proof}
		The second part follows from the first and Remark \ref{rmk;G-star-compact} since $i^{\tr}$ preserves filtered colimits. Recall the left adjoint $L\gamma^*$ of $i^{\rm tr}$, which is computed by colimits by the formula \[L\gamma^*(\Q(X)[i]) := (\gamma^*\Q(X))[i] = \Q_{\ltr}(X)[i].\] If $C_\bullet(G)\to G$ is the Breen--Deligne resolution, we have that \[L\gamma^*(G\otimes_\Z\Q) = L\gamma^*(C_\bullet(G)\otimes_\Z\Q) = \gamma^*(C_{\bullet}(G)\otimes_\Z\Q),\] and the last equality follows from the fact that $C_{i}(G)$ is representable. This gives the following equivalence:
		\begin{align*}
		\Map_{\cD(\Sh_{\et}(k,\Q))}(G\otimes_\Z\Q,i^{\rm tr}F) &\simeq \Map_{\cD(\Sh_{\et}^{\rm tr}(k,\Q))}(L\gamma^*(G\otimes_\Z\Q),F) \\&\simeq
		\Map_{\cD(\Sh_{\et}^{\rm tr}(k,\Q))}(\gamma^*(C_\bullet(G)\otimes_\Z\Q),F)
		\end{align*}
		By \cite[Lemma B.4]{AHPL}, we have an equivalence in $\Sh_{\et}^{\tr}(k,\Q)$:\[
		\gamma^*(C_\bullet(G)\otimes_\Z\Q) \simeq G\otimes_\Z\Q
		\]
		which allows us to conclude.
	\end{proof}
\end{cor}
We need the following easy but fundamental result:
\begin{prop}\label{prop:map-sh-Q-coeff}
	Let $G\in \sG^*$. Then for all $F\in \Sh_{\et}(k,\Z)$ we have an equivalence:\[
	\Map_{\cD(\Sh_{\et}(k,\Z))}(G,F)\otimes_{\Z} \Q \simeq \Map_{\cD(\Sh_{\et}(k,\Q))}(G\otimes_{\Z} \Q,F\otimes_{\Z}\Q)
	\]
	\begin{proof}
		By adjunction we have\[
		\Map_{\cD(\Sh_{\et}(k,\Q))}(G\otimes_{\Z} \Q,F\otimes_{\Z}\Q)\simeq \Map_{\cD(\Sh_{\et}(k,\Z))}(G,F\otimes_{\Z}\Q).
		\]
		Notice that the tensor product is not derived since $\Q$ is flat. Consider the fiber sequence in $\cD(\Sh_{\et}(k,\Z))$: \[
		 F \to F\otimes_{\Z}\Q \to \colim_m F\otimes^L_{\Z}\Z/m\Z
		\]
		Then since the colimit is filtered and $G$ is compact, we have a fiber sequence\[
		\Map_{\cD(\Sh_{\et}(k,\Z))}(G,F)\to \Map_{\cD(\Sh_{\et}(k,\Z))}(G,F\otimes_{\Z}\Q)\to 	\colim \Map_{\cD(\Sh_{\et}(k,\Z))}(G,F\otimes^L_{\Z} \Z/m\Z).
		\]
		On the other hand, since $\Q$ is flat, $\otimes_\Z\Q$ commutes with homotopy groups, so since $\pi_i$ and $\otimes_\Z\Q$ commute with filtered colimits, we have that \begin{align*}
		&\pi_i\bigl((\colim \Map_{\cD(\Sh_{\et}(k,\Z))}(G,F\otimes^L_{\Z} \Z/m\Z))\otimes_{\Z}\Q\bigr) \\
		&\cong  \colim \bigl(\pi_i (\Map_{\cD(\Sh_{\et}(k,\Z))}(G,F\otimes^L_{\Z} \Z/m\Z)\otimes_{\Z}\Q)\bigr).
		\end{align*}
		Since $\pi_i\Map_{\cD(\Sh_{\et}(k,\Z))}(G,F\otimes^L_{\Z} \Z/m\Z)$ are $\Z/m\Z$-modules for all $i$, the above formula is equal to zero, so we conclude.
	\end{proof}
\end{prop}
Finally, we can deduce the following result, which (for the most part) is a well-known application of the Breen--Deligne resolution:
\begin{prop}\label{prop:ext}
	Let $G_1,G_2\in \sG^*$. Then the groups $\Ext^n_{\Sh_{\et}(k,\Z)}(G_1,G_2)$ are torsion for $n\geq 2$
	\begin{proof}
		If $G_2$ is discrete, this exactly \cite[Proposition 3.2.1, (i)]{BVKahn}. If $G_2$ is connected, as pointed out in \cite[Proposition 3.2.1, (ii)]{BVKahn} this can be deduced from Breen's method from \cite{Breen}. In fact, let $X$ be the underlying smooth scheme of $G_1$: using the Breen--Deligne resolution of Theorem \ref{thm:Breen-Deligne}, one reduces to show that $H^q_{\et}(X^{\times s},G_2)$ is torsion for $q\geq 2$. If $G_2$ is represented by a semi-abelian variety, this was pointed out in \cite[\S 7,8, and 9]{Breen}. If $G_2=\G_a$, then $H^q_{\et}(X^{\times s},\G_a) = H^q_{\Zar}(X^{\times s},\G_a)$:
		\begin{itemize}
			\item[(i)] If $G_1$ is discrete, then $\dim(X^{\times s})=0$ so $H^q_{\Zar}(X^{\times s},\G_a) = 0$ for $q\geq 1$.
			\item[(ii)] If $G_1=\G_m$ or $\G_a$, then $X^{\times s}$ is affine so $H^q_{\Zar}(X^{\times s},\G_a) = 0$ for $q\geq 1$.
		\end{itemize}
	By Chevalley's classification, we are left to the case where $G_1$ is an abelian variety. In this case, since $\ch(k)=0$, we have that $\G_a$ is already a sheaf of $\Q$-vector spaces, so by Corollary \ref{cor:AHPL} and Proposition \ref{prop:rat-coeff-nisequaletale} we have:\[
	\Ext^n_{\Sh_{\et}(k,\Q)}(G_1\otimes\Q,\G_a) \cong \Ext^n_{\Sh_{\Nis}^{\tr}(k,\Z)}(G_1,\G_a).
	\]
	So we will conclude by the following:
	\begin{claim}\label{claim:ext-abelian-Ga}
		Let $G_1$ be an abelian variety. Then $\Ext^n_{\Sh_{\Nis}^{\tr}(k,\Z)}(G_1,\G_a)=0$ for $n\geq 2$
		\proof
		For a curve $C\subseteq G_1$ intersection of ample divisors, the map $J_C\otimes_\Z\Q\to G_1\otimes_\Z\Q$ is split surjective, which implies that we have a split inclusion:
		\begin{align*}
			 \Ext^n_{\Sh_{\Nis}^{\tr}(k,\Z)}(G_1,\G_a)\cong &\Ext^n_{\Sh_{\et}(k,\Q)}(G_1\otimes\Q,\G_a)\hookrightarrow \\
			 &\Ext^n_{\Sh_{\et}(k,\Q)}(J_C\otimes\Q,\G_a)\cong \Ext^n_{\Sh_{\Nis}^{\tr}(k,\Z)}(J_C,\G_a),
		\end{align*}
	which reduces to the case where $G_1=J_C$. 
	Since $J_C$ is an abelian variety, it is a birational sheaf in the sense of \cite{KS}, so for $X\in \SmlSm(k)$ we have an isomorphism\[
	J_C(\ul{X})\cong J_C(X^\circ)
	\]
	which implies that we have an isomorphism of $\dNis$ sheaves (cf. Remark \ref{rmk:Ga}) :\[
	\omega^\sharp_{\log}(J_C)\cong \omega^{*}_{\log}(J_C).
	\]
Since $\G_a\cong \omega_{\sharp}^{\log}\Log(\G_a)$, the adjunction 
\eqref{eq:adjunctionomega-derived1-2} implies an equivalence 
\[	\Map_{\cD(\Sh_{\Nis}^{\rm tr}(k,\Z))}(J_C,\G_a) \simeq 	\Map_{\cD(\Sh_{\dNis}^{\rm ltr}(k,\Z))}(\omega^*_{\log}J_C,\Log\G_a).
	\] 
	On the other hand, since $J_C[0]$ is a direct summand of the Suslin--Voevodsky complex $C_*^{\A^1}(C)$ in $\cD(\Sh_{\Nis}^{\rm tr}(k,\Z))$ by \cite[Section 3.4]{VoevTriangCat}, we have that $\omega^*_{\log}J_C[0]$ is a direct summand of $\iota (M(C,\triv))$ by \cite[Theorem 8.2.11]{BPO}, where $\iota$ is the inclusion of $\logDM(k,\Z)$ in $\cD(\Sh_{\dNis}(k,\Z))$. Since now $\Log(\G_a)$ is $\bcube$-local, we have an injective map:
	\begin{align*}
	\Ext^n_{\Sh_{\dNis}^{\rm ltr}(k,\Z)}(\omega^*_{\log}J_C,\Log\G_a)&\hookrightarrow \pi_{-n}\Map_{\cD(\Sh_{\dNis}^{\rm ltr}(k,\Z))}(M(C,\triv),\Log(\G_a))\\
	&\cong \pi_{-n}\Map_{\logDM(k,\Z)}(M(C,\triv),\Log(\G_a))\\
	&\cong H^n(C,\cO_{C})
	\end{align*}
	and the latter is zero for $n\geq 2$. This allows us to conclude. 
	\end{claim}
	\end{proof}
\end{prop}
\begin{remark}\label{rmk;RSC-star-compact}
	Recall that if $F\in \RSC_{\et,\leq 1}(k,\Q)^{\star}$, then there exist $L\to G$ 
as in Proposition \ref{prop:can-pres-1-mot-sheaf} such that\[
	0\to L\to G\otimes \Q \to F\to 0.
	\]
	By Remark \ref{rmk;G-star-compact} and \ref{cor:AHPL} (plus the fact that  $\RSC_{\et,\leq 1}(k,\Q)$ is closed under colimits in $\Sh_{\et}^{\tr}(k,\Q)$), $L$ and $G$ are compact in both $\Shv_{\et}(k,\Q)$ and $\RSC_{\et,\leq 1}(k,\Q)$, which implies that $F$ is compact in both $\Shv_{\et}(k,\Q)$ and $\RSC_{\et,\leq 1}(k,\Q)$.
\end{remark}
We are now ready to prove the following:
\begin{prop}\label{prop;1-rec-closed-under-ext}
	The category $\RSC_{\et,\leq 1}(k,\Q)$ is closed under extensions in $\Shv_{\et}(k,\Q)$. 
	\begin{proof}
		Let $F_1,F_2\in \RSC_{\et,\leq 1}(k,\Q)$.	Since $\RSC_{\et,\leq 1}(k,\Q)$ is a full subcategory of $\Shv_{\et}(k,\Q)$ by Proposition \ref{prop;forget-transfer-fully-faithful}, we have that an extension in $\RSC_{\et,\leq 1}(k,\Q)$ splits in $\Shv_{\et}(k,\Q)$ if and only if it splits in $\RSC_{\et,\leq 1}(k,\Q)$, so \[\Ext^1_{\RSC_{\et, \leq 1}(k,\Q)}(F_1,F_2)\hookrightarrow \Ext^1_{\Shv_{\et}(k,\Q)}(F_1,F_2).\] 
		By Proposition \ref{prop:can-pres-1-mot-sheaf}, we can write $F_1 = \colim F_{i_1}$ with $F_{i_1}\in \RSC_{\et,\leq 1}^{\star}$. In particular, there is a surjective map $\oplus F_{i_1}\twoheadrightarrow F_1.$
		Let $K$ be its kernel: we have the following commutative diagram with exact rows:
		\begin{equation}\label{eq;reduce-fin-pres}
			\begin{small}\begin{tikzcd}
					\Hom_{\RSC_{\et,\leq 1}}(K,F_{2})\ar[r]\ar[d,"\simeq"]&\Ext^1_{\RSC_{\et,\leq 1}}(F_{1},F_{2})\ar[r]\ar[d,"(*)"]&\prod\Ext^1_{\RSC_{\et,\leq 1}}(F_{i_1},F_{2})\ar[r]\ar[d,"(**)"]&\Ext^1_{\RSC_{\et,\leq 1}}(K,F_{2})\ar[d,hook]\\
					\Hom_{\Shv_{\et}}(K,F_2)\ar[r]&\Ext^1_{\Shv_{\et}}(F_{1},F_{2})\ar[r]&\prod\Ext^1_{\Shv_{\et}}(F_{i_1},F_{2})\ar[r]&\Ext^1_{\Shv_{\et}}(K,F_2)
				\end{tikzcd}
			\end{small}
		\end{equation}
so by the five-lemma $(*)$ is surjective if $(**)$ is surjective, hence we can suppose $F_1\in \RSC_{\et,\leq 1}^\star(k,\Q)$. Since $F_{1}$ is compact in both $\Shv_{\et}(k,\Q)$ and $\RSC_{\et,\leq 1}(k,\Q)$ by Remark \ref{rmk;RSC-star-compact} and filtered colimits are exact, again by \ref{prop:can-pres-1-mot-sheaf} it is enough to suppose that $F_2\in \RSC_{\et,\leq 1}^\star(k,\Q)$.
Let $F_{i}=\coker(L_i \hookrightarrow G_i\otimes_{\Z}\Q)$ as in Proposition \ref{prop:can-pres-1-mot-sheaf}, we have a commutative square with exact rows (we omit the $\otimes_\Z\Q$ here):
\[\begin{tikzcd}
		\Ext^1_{\RSC_{\et,\leq 1}}(G_{1},G_{2})\ar[d,hook,"(1)"]\ar[r] &\Ext^1_{\RSC_{\et,\leq 1}}(G_{1},F_{2})\ar[d,hook,"(3)"]\ar[r]&\Ext^2_{\RSC_{\et,\leq 1}}(G_{1},L_{2})\ar[d]\\
		\Ext^1_{\Shv_{\et}}(G_{1},G_{2})\ar[r,"(2)"] &\Ext^1_{\Shv_{\et}}(G_{1},F_{2})\ar[r]&\Ext^2_{\Sh_{\et}}(G_{1},L_{2}).
		\end{tikzcd}
		\] 
		By Proposition \ref{prop:map-sh-Q-coeff} and Corollary \ref{cor:ext-et-fppf}, we have that\[
			\Ext^1_{\Shv_{\et}(k,\Q)}(G_{1}\otimes_\Z\Q,G_{2}\otimes_\Z\Q) \cong \Ext^1_{\Shv_{\rm fppf}(k,\Z)}(G_{1},G_{2})\otimes_\Z\Q.
		\] 
		By \cite[Exercise 5-10]{milnegroups} we have that $\sG^*$ is closed by extensions in $\textrm{fppf}$ sheaves, which implies that if $E\in \Ext^1_{\Shv_{\et}(k,\Q)}(G_{1},G_{2})$, there exists $G'\in \sG^*$ such that $E=G'\otimes_{\Z}\Q$, in particular $E$ is in the image of $(1)$. This implies that $(1)$ is an isomorphism. 
		Moreover, $\Ext^2_{\Sh_{\et}(k,\Z)}(G_{1},L_{2})\otimes_\Z\Q =0$ by Proposition \ref{prop:ext}, which implies that $(2)$ is surjective, so $(3)$ is surjective, hence an isomorphism. 
		We have now the following commutative diagram with exact rows:\[
		\begin{small}\begin{tikzcd}
				\Hom_{\RSC_{\et,\leq 1}}(L_{1},F_{2})\ar[r]\ar[d,"\simeq"]&\Ext^1_{\RSC_{\et,\leq 1}}(F_{1},F_{2})\ar[r]\ar[d,hook]&\Ext^1_{\RSC_{\et,\leq 1}}(G_{1},F_{2})\ar[r]\ar[d,"\simeq"]&\Ext^1_{\RSC_{\et,\leq 1}}(L_{1},F_{2})\ar[d,hook]\\
				\Hom_{\Shv_{\et}}(L_{1},F_2)\ar[r]&\Ext^1_{\Shv_{\et}}(F_{1},F_{2})\ar[r]&\Ext^1_{\Shv_{\et}}(G_{1},F_{2})\ar[r]&\Ext^1_{\Shv_{\et}}(L_1,F_2).
			\end{tikzcd}
		\end{small}
		\]
        and we conclude using the 5-lemma. 
	\end{proof}
\end{prop}

\begin{lemma}\label{lem:hom-dim-in-Shv} 
	For all $F_1,F_2\in \RSC_{\et,\leq 1}(k, \Q)$ we have \[
	\Ext^i_{\Shv_{\et}(k,\Q)}(F_1,F_2)=0\quad \textrm{for }i\geq 2.
	\]
	\begin{proof}
First suppose that $F_1\in \RSC_{\et, \leq 1}^{\star}$, hence it is a compact object in $\Shv_{\et}(k,\Q)$. Since filtered colimits are exact, as in the proof of Proposition \ref{prop;1-rec-closed-under-ext} we can assume that $F_2\in \RSC_{\et,\leq 1}^{\star}(k,\Q)$. For $i=1,2$, let $F_i=\coker(L_i\hookrightarrow G_i)$ be as in Proposition \ref{prop:can-pres-1-mot-sheaf} (again, we omit $\otimes_\Z\Q$). Since $L_1$ is a lattice, we have that $\Ext^i_{\Sh_{\et}(k,\Q)}(L_1,F_2)=0$ for $i\geq 1$, so we can suppose $F_1=G_1$. Since by Proposition \ref{prop:ext}, we have\[
	\Ext^i_{\Sh_{\et}(k,\Q)}(G_1,L_2) = \Ext^i_{\Sh_{\et}(k,\Q)}(G_1,G_2) = 0 \quad \textrm{for }i\geq 2,
		\]
we get $\Ext^i_{\Shv_{\et}(k,\Q)}(F_1,F_2)=0$ for $i\geq 2$, which concludes
the case $F_1\in \RSC_{\et, \leq 1}^{\star}$.
In general, let $0\to \bigoplus \G_a/ \mathcal{L}_1\to F_1\to F_1^{\A^1}\to 0$ be as in \eqref{eq:Ga-A1}.  
Then we have a long exact sequence sequence
\begin{equation}\label{eq:reduction-F1-A1local}
 \Ext^i_{\Sh_{\et}(k,\Q)}(F_1^{\A^1},F_2)\to \Ext^i_{\Sh_{\et}(k,\Q)}(F_1,F_2)\to \Ext^i_{\Sh_{\et}(k,\Q)}(\oplus\G_a/ \mathcal{L}_1,F_2).
\end{equation}
As in \cite[Proposition 2.4.10]{Ayoub-BV}, $\mathcal{L}_1$ splits as a direct sum $\mathcal{L}_1\cong \oplus_a\mathcal{L}_{a_1}$ of finitely presented $0$-motivic sheaves, which implies that \[
\Ext^i_{\Sh_{\et}(k,\Q)}(\mathcal{L}_1,F_2) \cong \prod \Ext^i_{\Sh_{\et}(k,\Q)}(\mathcal{L}_{a_1},F_2)=0\quad \textrm{for }i\geq 1.
\] 
Moreover, $\G_a\in \RSC_{\et, \leq 1}^{\star}$, so by the previous case:\[
\Ext^i_{\Sh_{\et}(k,\Q)}(\oplus \G_a,F_2) \cong \prod \Ext^i_{\Sh_{\et}(k,\Q)}(\G_a,F_2)=0\quad \textrm{for }i\geq 2.
\] 
This imples that the last term of \eqref{eq:reduction-F1-A1local} vanishes for $i\geq 2$. In order to conclude, it is enough to show that $\Ext^i_{\Sh_{\et}(k,\Q)}(F_1^{\A^1},F_2)=0$. In other words, we can suppose that $F_1\in \HI_{\et,\leq 1}$. Arguing as before, consider now the exact sequence 
\[0\to \oplus \G_a/ \mathcal{L}_2\to F_2\to F_2^{\A^1}\to 0.\] It induces a long exact sequence\[
\Ext^i_{\Sh_{\et}(k,\Q)}(F_1,\bigoplus \G_a/ \mathcal{L}_2)\to \Ext^i_{\Sh_{\et}(k,\Q)}(F_1,F_2)\to  \Ext^i_{\Sh_{\et}(k,\Q)}(F_1,F_2^{\A^1})
\]
By \cite[Proposition 2.4.10]{Ayoub-BV} we have that for $i\geq 2$\[
\Ext^i_{\Sh_{\et}(k,\Q)}(F_1,F_2^{\A^1})=\Ext^i_{\Sh_{\et}(k,\Q)}(F_1,\mathcal{L}_2) = 0,
\]
so we need to show that \[\Ext^i_{\Sh_{\et}(k,\Q)}(F_1,\oplus \G_a)=0\;\text{ for $i\geq 2$. }\]
As in the proof of \cite[Proposition 2.4.10]{Ayoub-BV}, we can reduce to separately analyze the following cases:
\begin{enumerate}
	\item[(i)] $F_1=\mathcal{L}$ is a $0$-motivic sheaf.
	\item[(ii)] $F_1=\mathcal{L}\otimes\G_m$, where $\mathcal{L}$ is a $0$-motivic sheaf.
	\item[(iii)] $F_1\simeq \bigoplus_{\beta} A_{\beta}\otimes\Q$, where $A_\beta$ are abelian varieties.
\end{enumerate}
If $\mathcal{L}$ is $0$-motivic, again it splits as a direct sum $\bigoplus \mathcal{L}_a$ of finitely presented $0$-motivic sheaves.
In particular 
\[ \Ext^i_{\Sh_{\et}(k,\Q)}(\mathcal{L}\oplus \G_a)\cong \prod \Ext^i_{\Sh_{\et}(k,\Q)}(\mathcal{L}_a,\oplus \G_a),
\]
and\[
\Ext^i_{\Sh_{\et}(k,\Q)}(\mathcal{L}\otimes \G_m,\oplus \G_a)\cong \prod \Ext^i_{\Sh_{\et}(k,\Q)}(\mathcal{L}_a\otimes \G_m,\oplus \G_a)
\]
Since now both $\mathcal{L}_a$ and $\mathcal{L}_a\otimes \G_m$ are in $\RSC_{\et,\leq 1}^{\star}$, so they are compact objects in $\Sh_{\et}(k,\Q)$, we can put the direct sum outside the Ext and conclude the vanishing for $i\geq 2$ in case  (i) and (ii) thanks to Proposition \ref{prop:ext}.
 In case (iii), $A_\beta\otimes\Q\in\RSC_{\et,\leq 1}^{\star}$, so 
 \begin{align*}
 \Ext^i_{\Sh_{\et}(k,\Q)}(\oplus_{\beta} A_{\beta}\otimes\Q,\oplus \G_a)&=
\prod_{\beta}  \Ext^i_{\Sh_{\et}(k,\Q)}(A_\beta\otimes\Q,\oplus \G_a) \\
&\cong \prod_{\beta}  \bigoplus\Ext^i_{\Sh_{\et}(k,\Q)}(A_\beta\otimes\Q,\G_a)
\end{align*}
and the last term vanishes by Claim \ref{claim:ext-abelian-Ga}.
    \end{proof}
\end{lemma}
\begin{prop}\label{prop:hom-dim}
	For $F_1,F_2\in \RSC_{\et,\leq 1}(k,\Q)$, we have
\[\Ext^i_{\RSC_{\et,\leq 1}(k,\Q)}(F_1,F_2)\cong \Ext^i_{\Shv_{\et}(k,\Q)}(F_1,F_2)\;\text{  for all $i\geq 0$.}\] 
In particular the category $\RSC_{\et, \leq 1}(k, \Q)$ has cohomological dimension $1$, i.e. for $i\geq 2$ we have \[\Ext^i_{\RSC_{\et,\leq 1}(k,\Q)}(F_1,F_2)=0.\] 
	\begin{proof}
		{
			The case $i=0$ is Proposition \ref{prop;forget-transfer-fully-faithful} and $i=1$ is Proposition \ref{prop;1-rec-closed-under-ext}. In light of Lemma \ref{lem:hom-dim-in-Shv}, the first part for $i\geq 2$ follows from the second. Let $0\to F_2\to I\to B\to 0$ with $I$ injective in $\RSC_{\et,\leq 1}(k,\Q)$,
			so that for $i\geq 2$ we have that \[
			\Ext^i_{\RSC_{\et,\leq 1}(k,\Q)}(F_1,F_2)\cong\Ext^{i-1}_{\RSC_{\et,\leq 1}(k,\Q)}(F_1,B).\]  Notice that by Proposition \ref{prop;1-rec-closed-under-ext}, we have a commutative diagram where the vertical maps are isomorphisms:}\[
			\begin{tikzcd}
				\Ext^1_{\Sh_{\et}(k,\Q)}(F_1,I)\ar[r,"(*1)"]\ar[d,"\cong"]&\Ext^1_{\Sh_{\et}(k,\Q)}(F_1,B)\ar[d,"\cong"]\\
				\Ext^1_{\RSC_{\et,\leq 1}(k,\Q)}(F_1,I)\ar[r,"(*2)"]&\Ext^1_{\RSC_{\et,\leq 1}(k,\Q)}(F_1,B).
			\end{tikzcd}
			\]
		{
		Since $\Ext^2_{\Sh_{\et}(k,\Q)}(F_1,F_2)=0$ by Lemma \ref{lem:hom-dim-in-Shv}, the map $(*1)$ is surjective, so the map $(*2)$ is surjective. Since $I$ is injective, $\Ext^1_{\RSC_{\et,\leq 1}(k,\Q)}(F_1,I)=0$, so $\Ext^1_{\RSC_{\et,\leq 1}(k,\Q)}(F_1,B)=0$. We conclude that: \[
		\Ext^2_{\RSC_{\et}(k,\Q)}(F_1,F_2)=0.
		\]
	which concludes the case $i=2$. In general, suppose that $\Ext^i_{\RSC_{\et, \leq 1}(k,\Q)}(F_1,F_2)=0$ for all $F_1,F_2\in \RSC_{\et, \leq 1}(k,\Q)$, in particular $\Ext^i_{\RSC_{\et, \leq 1}(k,\Q)}(F_1,B)=0$, so we conclude that $\Ext^{i+1}_{\RSC_{\et, \leq 1}(k,\Q)}(F_1,F_2)=0$, which by induction finishes the proof.}
	\end{proof}	
\end{prop}

\begin{thm}\label{lem:split}
	Every complex in $C\in \cD(\RSC_{\et, \leq 1}(k,\Q))$ is formal, i.e. we have an equivalence:\[
	C\simeq \bigoplus \pi_i(C)[i]
	\]
	\begin{proof}
		By Proposition \ref{prop:hom-dim}, we have that for every $F\in \RSC_{\et,\leq 1}$, and $i\geq 2$, \[\Ext^i_{\RSC_{\et, \leq 1}}(\pi_jC,F)=0,\]
		so the spectral sequence
		\begin{equation}\label{eq;spectral-sequence-ext-coherent}
			E_2^{ij}=	\Ext^i_{\RSC_{\et, \leq 1}}(\pi_j C,F) \Rightarrow \pi_{i+j}\Map_{\cD(\RSC_{\et, \leq 1})}(C,F[0])
		\end{equation}
		degenerates at page 2, which allows us to conclude by \cite[Prop (1.2) and Rem (1.4)]{DeligneLefschetz}
	\end{proof}
\end{thm}

\section{The derived Albanese functor}\label{sec:derived}
Assume that $k$ has characteristic zero. 
The goal of this section is to show the existence of a left derived Albanese functor in the sense of Definition \ref{defn;left-derived}. To ease the notation, we will write $\Shv^{\tr}$ (resp. $\Shv^{\ltr}$, resp. $\logCI$) for $\Shv^{\tr}_{\Nis}(k, \Q)$ (resp. $\Shv^{\ltr}_{\dNis}(k, \Q)$,  resp. $\logCI_{\dNis}$) and we will identify it with $\Shv^{\tr}_{\et}(k, \Q)$ (resp. $\Shv^{\ltr}_{\textrm{d\'et}}(k, \Q)$, resp. $\logCI_{\det}$) by Proposition \ref{prop:et-Nis-Voe} (resp. Proposition \ref{prop:rat-coeff-nisequaletale}). We also write $\RSC_{\et,\leq 1}\subset \RSC_{\et}$ for 
$\RSC_{\et,\leq 1}(k,\Q)\subset \RSC_{\et}(k,\Q)$.
In this section we let $\Log_{\det}$ denote the functor from Theorem \ref{thm:logAlb}:
\begin{equation}\label{omegaCI-1mot}
\begin{tikzcd}\RSC_{\et,\leq 1}\arrow[r, hook] & \Shv^{\ltr}.\end{tikzcd}
\end{equation} 
By Remark \ref{rmk;omegaCI-commutes-with-colimits} and Proposition \ref{prop:CI1-thick} this functor is fully faithful, exact and commutes with all colimits. 

We announce the main theorem (cfr.~\cite[Thm. 2.4.1]{Ayoub-BV}), whose proof will occupy the rest of the Section.
\begin{thm}\label{thm:derivedAlbanese}
	The functor $\Alb^{\log}$ of Theorem \ref{thm:logAlb} has a pro-left derived functor $L\Alb^{\log}$ which factors through the stable $\infty$-category of effective log motives, giving rise to the \textit{log motivic Albanese functor} (still denoted $L\Alb^{\log}$):
	\[L\Alb^{\log}\colon \logDM(k, \Q) \to \Pro \cD(\RSC_{\et, \leq 1}), \]
	which is a pro-left adjoint of the functor \[\omega^{\lDM}_{\leq 1}\colon \cD(\RSC_{\et, \leq 1})\xrightarrow{\cD(\Log_{\det})} \cD(\Shv^{\ltr})\xrightarrow{L_{\bcube}} \logDM(k,\Q).\]

\end{thm}	

\subsection{Some preliminary results} 
We collect now some technical lemmas that will be used in the proof of the main theorem.

Recall the pair of adjoint functors from \eqref{omegaadjunction}
\[\begin{tikzcd}
	{\omega^*_{\log} \colon \Shv^{\tr}_{\et}(k,\Q)}\ar[r,shift right = 1] &{\Shv_{\textrm{d\'et}}^{\ltr}(k,\Q)\colon \omega_{\sharp}^{\log}}\ar[l,shift right=1.5]
\end{tikzcd}
\]
where $\omega^*_{\log}F(X) = F(\underline{X}-|\partial X|)$. Recall that $\omega_{\sharp}^{\log} \omega^*_{\log}=\mathrm{id}$ and that 
by \cite[Proposition 8.2.8]{BPO}, for $F\in \HI_{\et}$ we have that $\omega^*_{\log}(F)\cong \Log_{\det}F$.

\begin{lemma}\label{lem:useful-lemma} For any finitely presented $1$-reciprocity sheaf $F\in \RSC_{\et, \leq 1}$, we have an isomorphism:
	\[\omega_{\sharp}^{\log}\underline{\Hom}_{\Shv^{\ltr}}(\omega^*_{\log}\G_m,\Log_{\det}F)\cong   \underline{\Hom}_{\Shv^{\tr}}(\G_m, F^{\A^1}),\]
	where $F^{\A^1}$ denotes the maximal $\A^1$-invariant quotient of $F$ (see \eqref{eq:finpresA1}).

\begin{proof}
Consider the presentation $0\to L\to G\to F\to 0$ of Proposition \ref{prop:can-pres-1-mot-sheaf}. 
Since $L\in \HI$, we have that $\Log_{\det}L=\omega^*_{\log}L$ and by Proposition \ref{prop;internal-hom-omega}, by \cite[Lemma 3.1.4]{BVKahn} and by Proposition \ref{prop;forget-transfer-fully-faithful}, we have
\begin{align*}
	&\omega_{\sharp}^{\log}\uHom_{\Shv^{\ltr}}(\omega^*_{\log}\G_m,\omega^*_{\log}L) = \uHom_{\Shv^{\tr}}(\G_m,L)=0 \quad \textrm{and} \\ 
	&\omega_{\sharp}^{\log}\uExt^1_{\Shv^{\ltr}}(\omega^*_{\log}\G_m,\omega^*_{\log}L) =\uExt^1_{\Shv^{\tr}}(\G_m,L)=0.
\end{align*}
Since $\Log_{\det}$ is exact, this gives an isomorphism\[
\omega_{\sharp}^{\log}\underline{\Hom}_{\Shv^{\ltr}}(\omega^*\G_m, \Log_{\det}G)\cong \omega_{\sharp}^{\log}\underline{\Hom}_{\Shv^{\ltr}}(\omega^*_{\log}\G_m, \Log_{\det}F)
\]
Moreover, 
by \cite[8.2.4]{BPO} together with \cite[Theorem 7.16]{MVW}, $\omega^*_{\log}\G_m[0]\oplus \Q= M(\P^1,0+\infty)$, so for all $i\geq 0$\[
\omega_{\sharp}^{\log}\underline{\rm Ext}^i_{\Shv^{\ltr}}(\omega^*_{\log}\G_m, \Log_{\det}\G_a^r) = \pi_i\omega_{\sharp}^{\log}\underline{\rm Map}_{\Shv^{\ltr}}(M(\P^1,0+\infty), \Log_{\det}\G_a^r)/\Log_{\det}\G_a^r.
\]
By the fiber sequence $M(\P^1,\triv)\to M(\P^1,0)\oplus M(\P^1,\infty)\to M(\P^1,0+\infty)$ and the fact that $\Log_{\det}(\G_a)$ is both $(\P^1,\triv)$ and $\bcube$-local, we conclude that \[
\omega_{\sharp}^{\log}\underline{\rm Ext}^i_{\Shv^{\ltr}}(\omega^*_{\log}\G_m, \Log_{\det}\G_a^r) = 0\quad\textrm{for }i\geq 0.\]
We can now conclude from the sequence \eqref{eq:Chev} and Proposition \ref{prop;internal-hom-omega}, since we have that:
\[\omega_{\sharp}\underline{\Hom}_{\Shv^{\ltr}}(\omega^*_{\log}\G_m, \Log_{\det}G ) = \omega_{\sharp}\underline{\Hom}_{\Shv^{\tr}}(\omega^*_{\log}\G_m, \omega^*_{\log}G^{sab})=\underline{\Hom}_{\Shv^{\tr}}(\G_m, G^{sab}).\]

\end{proof}
\end{lemma}

Let $\kX =(\ol{X}, D)$ be a  smooth proper and geometrically integral modulus pair, and let $\bAlb_{\kX}^{sab}$ be the maximal semi-abelian quotient of $\bAlb_{\kX}$. Suppose that $X = \ol{X} \setminus D$ has a $k$-rational point. Then, by Proposition \ref{prop:max-semiab-quotient}, $\bAlb_{\kX}^{sab}$ agrees with Serre's semi-abelian Albanese variety of $X$, $\Alb_X$. 
{We have an exact sequence as in Remark \ref{rmk:Fsab}}: 
\begin{equation}\label{eq:can-map-Alb-Albsemiab} 0\to  U(\kX)\otimes_k\G_a\to \bAlb_{\kX} \to \bAlb_\kX^{sab} = \Alb_{X} \to 0
\end{equation}
{where $U(\kX)$ is a finite dimensions $k$-vector space.}
Since $\Alb_{X}\in \HI_{\leq 1, \et}$, which is a Serre subcategory of $\RSC_{\leq 1,\et}$, 
the natural surjection of \eqref{eq:can-map-Alb-Albsemiab} (with rational coefficients) gives then rise to a map
\begin{equation}\label{eq:ext-hi-ext-ci}  \Ext^1_{\HI_{\et, \leq 1}} ( \Alb_{X}, \G_m) \to \Ext^1_{\RSC_{\et,\leq 1}} (\bAlb_{\kX}, \G_m). 
\end{equation}
\begin{lemma}\label{lem:ext-hi-ext-ci}The map \eqref{eq:ext-hi-ext-ci} is an isomorphism.
\begin{proof}
We need to show that 
\[ \Hom_{\RSC_{\et, \leq 1}}  (U,\G_m) =0, \quad \text{and}\quad	  \Ext^1_{\RSC_{\et, \leq 1}} (U,\G_m)=0.  \]
which by Proposition \ref{prop:hom-dim} follow form analogous vanishing in $\Shv_{\et}$, which are well-known (see e.g. \cite[VII, Proposition 7]{SerreGACC}).
\end{proof}
\end{lemma}

Notice that for every  $G\in \mathbf{logCI}$, we have that $\underline{\Hom}_{\Shv^{\ltr}}(\omega^*_{\log} \G_m, G)\in \mathbf{logCI}$ 
since \[
\underline{\Hom}_{\Shv^{\ltr}}(\omega^*_{\log} \G_m, G) = \pi_0(R\underline{\Hom}_{\Shv^{\ltr}}(\omega^*_{\log} \G_m, G))
\] 
and $R\underline{\Hom}_{\Shv^{\ltr}}(\omega^*_{\log} \G_m, G)$ is clearly $\bcube$-local (see \cite[Lemma 2.10]{BindaMerici}).
\begin{lemma}\label{lem;applying-uneful-tensor-hi-closed}
For any finitely presented $1$-reciprocity sheaf $F\in \RSC_{\et, \leq 1}$, we have an isomorphism:\[
			\omega^*_{\log}\G_m\otimes_{\mathbf{logCI}}\underline{\Hom}_{\Shv^{\ltr}}(\omega^*_{\log}\G_m, \Log_{\det}(F)) \xrightarrow{\simeq}  \omega^*_{\log}\big(\G_m \tensor_{\HI} \underline{\Hom}_{\Shv^{\tr}}(\G_m, F^{\A^1})\big).
		\]
where $\otimes_{\mathbf{logCI}}$ is the tensor product of the category $\mathbf{logCI}$.
	\begin{proof}
By Lemma \ref{lem:useful-lemma}, we have that	
	\[
		\omega_{\sharp}^{\log}\underline{\Hom}_{\Shv^{\ltr}_{\et}(k,\Q)}(\omega^*_{\log} \G_m, \Log_{\det} F) \cong \underline{\Hom}_{\Shv^{\tr}_{\et}(k,\Q)}(\G_m, F^{\A^1}).
		\] 
		Moreover, the right-hand side is in $\HI_{\et}(k,\Q)$, hence by \eqref{eq;only-one-homotopy-invariant} we have an isomorphism
			\[
			\underline{\Hom}_{\Shv^{\ltr}}(\omega^*_{\log} \G_m, \Log_{\det} F) \cong \omega^*_{\log}\omega_{\sharp}^{\log}\underline{\Hom}_{\Shv^{\ltr}}(\omega^*_{\log} \G_m, \Log_{\det} F) \cong \omega^*_{\log}\underline{\Hom}_{\Shv^{\tr}}(\G_m, F^{\A^1}),
			\]
			so we conclude since by Lemma \ref{lem:uneful-tensor-hi-closed} we have that
			\[
			\omega^*_{\log}\G_m\otimes_{\mathbf{logCI}}\omega^*_{\log}\underline{\Hom}_{\Shv^{\tr}}(\G_m, F^{\A^1}) \cong  \omega^*_{\log}\big(\G_m \tensor_{\HI} \underline{\Hom}_{\Shv^{\tr}}(\G_m, F^{\A^1})\big).
			\]
	\end{proof}		
\end{lemma}

\begin{remark}\label{rmk;colimits-hom-tensor}
Thanks to Proposition \ref{prop:can-pres-1-mot-sheaf}, we can write every $F\in \RSC_{\et,\leq 1}(k,\Q)$ as filtered colimit of finitely presented $1$-reciprocity sheaves, $F=\colim_i F_i$. Since the functor $\Log_{\det}$ from \eqref{omegaCI-1mot} commutes with all colimits, we have
$\Log_{\det}(F)=\colim_i \Log_{\det}F_i$. By the isomorphism  \[\omega^*_{\log}\G_m\cong \coker( \Q\to h_{0}^{\log}({\P^1}, [0]+[\infty])),\]
given for example by \cite[8.2.4]{BPO} together with \cite[Theorem 7.16]{MVW}, we see that $\omega^*_{\log}\G_m$ is a compact object in $\mathbf{logCI}$, hence in $\Shv^{\ltr}$. 
By Lemma \ref{lem;applying-uneful-tensor-hi-closed} we have an isomorphism
\begin{equation}\label{eq;tensor-colim-fg}
\omega^*_{\log}\G_m \tensor_{\mathbf{logCI}}  \underline{\Hom}_{\Shv^{\ltr}}(\omega^*_{\log}\G_m, \Log_{\det} F)\cong \colim_i\omega^*_{\log}\big(\G_m \tensor_{\HI} \underline{\Hom}_{\Shv^{\tr}}(\G_m, F_i^{\A^1})\big).
\end{equation}
Moreover, since $\omega_{\sharp}^{\log}$ commutes with (filtered) colimits, we have that 
\begin{equation}\label{eq;tensor-in-HI}
\omega_{\sharp}^{\log}\bigl(\omega^*_{\log}\G_m \tensor_{\mathbf{logCI}}  \underline{\Hom}_{\Shv^{\ltr}}(\omega^*_{\log}\G_m, \Log_{\det} F)\big) = \colim \big(\G_m \tensor_{\HI} \underline{\Hom}_{\Shv^{\tr}}(\G_m, F_i^{\A^1})\big).
\end{equation}
In particular, since $\HI_{\et}$ is closed under colimits we have that for exery $F\in \RSC_{\et,\leq 1}$, \[\omega_{\sharp}^{\log}\bigl(\omega^*_{\log}\G_m \tensor_{\mathbf{logCI}}  \underline{\Hom}_{\Shv^{\ltr}}(\omega^*_{\log}\G_m, \Log_{\det} F)\big)\in \HI_{\et}.
\]
\end{remark}

\begin{lemma}\label{lem;vanish-tensor}
Let $F\in \RSC_{\et, \leq 1}(k, \Q)$, then:
	\begin{equation}\label{eq:vanishing-coh-GmtensorHom-jlarg1} 
		H^j_{\dNis}(X, \omega^*_{\log}\G_m \tensor_{\mathbf{logCI}}\underline{\Hom}_{\Shv^{\ltr}}( \omega^*_{\log}\G_m, \Log_{\det} F)) =0 \quad\textrm{for }j>1.
	\end{equation}
\begin{proof}
	Recall that for a torus $T$, we have that
	\begin{equation}\label{eq;torus-vanish}
		H^j_{\rm Nis} (X^{\circ},T) = H^j_{\rm Zar} (X^{\circ},T) = 0 \quad\textrm{for }j>1.
	\end{equation}
	By Proposition \ref{prop:can-pres-1-mot-sheaf}, let $F=\colim_i F_i$ with $F_i\in \RSC_{\et, \leq 1}^\star$. Since $\dNis$-cohomology commutes with filtered colimits, the left hand side of \eqref{eq:vanishing-coh-GmtensorHom-jlarg1} is isomorphic to\[
	\colim H^j_{\dNis}(X,\omega^*_{\log}\big(\G_m \tensor_{\HI} \underline{\Hom}_{\Shv^{\tr}}(\G_m, F_i^{\A^1})\big)).
	\]
	The sheaf $\underline{\Hom}_{\Shv^{\tr}}(\G_m, F_i^{\A^1})$ is a $0$-motivic homotopy invariant sheaf by \cite[Corollary 1.3.9]{Ayoub-BV}, hence $T_i:=\G_m\tensor_{\HI} \underline{\Hom}_{\Shv^{\tr}}(\G_m, F_i^{\A^1})$ is a torus. Then since $\omega^*_{\log}$ is exact and preserves injective sheaves, we have by \eqref{eq;torus-vanish} that
\[ 
H^j_{\dNis} (X,\omega^*_{\log}T_i) = H^j_{\rm Nis} (X^{\circ},T_i) = 0 \quad\textrm{for }j>1.
\]
\end{proof}
\end{lemma}
Let $X\in \SmlSm(k)$ and $\kX\in \Comp(X)$. For $H\in \HI_{\et}$, consider the composition:
\begin{equation}
	\Ext^1_{\RSC_{\et}} (\bAlb_{\kX},  H) \xrightarrow{(1)} \Ext^1_{\logCI_{\dNis}} (\Log_{\det}(\bAlb_{\kX}),  \omega^*_{\log}(H)) \xrightarrow{(2)} H^1_{\dNis}(X,\omega^*_{\log}(H))
\end{equation}
where $(1)$ is the application of the exact functor $\Log$ and the isomorphism \eqref{eq;only-one-homotopy-invariant} and $(2)$ is given by the map $\Q_{\ltr}(X)\to \Log_{\det}\bAlb_{\kX}$. 
For a smooth scheme $X$, let ${\rm NS}^{r}(X)$ be the group of codimension $r$-cycles modulo algebraic equivalence.
\begin{lemma}\label{lem;iso-alb-H1}
Let $F\in \RSC_{\et,\leq 1}$ and let $X\in \SmlSm(k)$ connected such that $X^\circ$ is affine and ${\rm NS}^1(X^\circ_{\ol{k}}) =0$. Then for $\kX\in \Comp(X)$ the map $\Q_{\tr}(X)\to \Log(\Alb_{\kX})$ induces an isomorphism:
\begin{multline*}
\Ext^1_{\RSC_{\et}} (\bAlb_{\kX},  \omega_{\sharp}^{\log}(\omega^*_{\log}\G_m \otimes_{\mathbf{logCI}} \underline{\Hom}(\omega^*_{\log}\G_m, \Log_{\det}(F))))\simeq\\ 
H^1_{\dNis} (X, \omega^*_{\log}\G_m \otimes_{\mathbf{logCI}} \underline{\Hom}( \omega^*_{\log} \G_m, \Log_{\det}(F))).
\end{multline*}
\begin{proof}
	Let $F=\colim F_i$ with $F_i\in \RSC_{\et, \leq 1}^\star$ and $T=\omega^*_{\log}\G_m \otimes_{\mathbf{logCI}} \underline{\Hom}(\omega^*_{\log}\G_m, \Log_{\det}F)$ and $T_i=\omega^*_{\log}\G_m \otimes_{\mathbf{logCI}} \underline{\Hom}(\omega^*_{\log}\G_m, \Log_{\det}F_i)$.
	By Corollary \ref{cor:AHPL}, $\bAlb_{\kX}$ is a compact object in $\RSC_{\et}$. Since filtered colimits are exact and $\RSC_{\et,\leq 1}\subseteq \RSC_{\et}$ is closed under extensions, by \eqref{eq;tensor-in-HI} we have that\[
		\Ext^1_{\RSC_{\et}} (\bAlb_{\kX},  \omega_{\sharp}^{\log}T)\cong \\
		\colim_i \Ext^1_{\RSC_{\et, \leq 1}} (\bAlb_{\kX}, \omega_{\sharp}^{\log}T_i).
	\]
	On the other hand, by Lemma \ref{lem;applying-uneful-tensor-hi-closed}
	we have
	\[H^1_{\dNis} (X, T) \cong \\
	\colim_i H^1_{\dNis} (X,\omega^*_{\log}(\G_m \otimes_{\HI} \underline{\Hom}_{\Shv^{\tr}}(\G_m,F_i^{\A^1}))).\]
	Moreover, since $\omega^*_{\log}$ is exact and preserves injective sheaves we have that
	\[
	H^1_{\dNis} (X,\omega^*_{\log}(\G_m \otimes_{\HI} \underline{\Hom}_{\Shv^{\tr}}(\G_m,F_i^{\A^1})))\cong H^1_{\rm Nis} (X^{\circ}, \G_m \otimes_{\HI} \underline{\Hom}_{\Shv^{\tr}}(\G_m, F_i^{\A^1} )).
	\]
	Finally, $\underline{\Hom}_{\Shv^{\tr}}(\G_m, F_i^{\A^1} )$ is a lattice by \cite[Corollary 1.3.9]{Ayoub-BV}, hence it is enough to show that for every lattice $L$ we have an isomorphism\[
	\Ext^1_{\RSC_{\et,\leq 1}}(\bAlb_{\kX},\G_m\otimes_{\HI}L)\cong H^1_{\Nis}(X^{\circ},\G_m\otimes_{\HI} L).
	\]
	By \cite[Lemma 2.4.5]{Ayoub-BV}, since $X^{\circ}$ is affine and ${\rm NS}^1(X^\circ_{\ol{k}}) =0$, we have that
	\[H^1_{\rm Nis} (X^{\circ}, \G_m \otimes_{\HI} L)\cong 
	\Ext^1_{\HI_{\et,\leq 1}}(\Alb_{X^\circ},\G_m \otimes_{\HI} L),
	\]
	so it is enough to show that the canonical map $\bAlb_{\kX}\to \bAlb_{X^\circ}$ induces an isomorphism\[
	\Ext^1_{\HI_{\et,\leq 1}}(\bAlb_{X^\circ},\G_m \otimes_{\HI}L )\cong\Ext^1_{\RSC_{\et,\leq 1}}(\bAlb_{\kX},\G_m \otimes_{\HI}L).
	\]
	If $k$ is algebraically closed, we have $L\cong \Q^r$, hence the above isomorphism  comes from Lemma \ref{lem:ext-hi-ext-ci}.
	A Galois descent argument (see \cite[Lemma 2.4.5, Step 1]{Ayoub-BV}) allows us to deduce the isomorphism above for any $k$. 
\end{proof}
\end{lemma}

\subsection{Deriving the Albanese functor}We are ready to prove Theorem \ref{thm:derivedAlbanese}. The categories $\RSC_{\et,\leq 1}$ and $\Shv^{\ltr}$ are Grothendieck abelian categories and the functor $\omegaCI_{\log}$ from \eqref{omegaCI-1mot} is exact and commutes with filtered colimits. 

The derived $\infty$-category $\cD(\Shv^{\ltr})$ is equivalent by classical reason to the $\infty$-category underlying the model category $\Cpx(\PSh^{\ltr}(k,\Lambda))$ with the $\dNis$-local model structure considered in \cite{BindaMerici}. In particular, by \cite[Lemma 2.15]{BindaMerici}, the functor $i_{\ltr}\colon \cD(\Shv^{\ltr})\to \Ch_{\dg}(\Shv^{\ltr})$ preserves filtered colimtis. In particular, the commutative square of $\infty$-categories:\[
\begin{tikzcd}
	\Ch_{\dg}(\RSC_{\et,\leq 1})\ar[r,"{\Ch_{\dg}(\Log_{\det})}"]\ar[d,"{L_{\leq 1}}"]&\Ch_{\dg}(\Shv^{\ltr})\ar[d,swap,"{L_{\ltr}}"]\\
	\cD(\RSC_{\et,\leq 1})\ar[r,"{\cD(\Log_{\det})}"]&\cD(\Shv^{\ltr})
\end{tikzcd}
\]
satisfies the hypotheses of \ref{ex-derived-cats}, so that $\cD(\Log_{\det})$ has a pro-left adjoint $L\Alb^{\log}$.  We consider the $BC$-admissibility with respect to this diagram.
Recall from Lemma \ref{lem:equivalent-condition-admissible} that a compact object $P\in \Sh^{\ltr}$ is $BC$-admissible as an object of $\Ch_{\dg}(\Sh^{\ltr})$
if and only if for every injective object $I\in \RSC_{\et,\leq 1}$ we have\[
\Ext^n_{\Sh^{\ltr}}(P,\Log_{\det}(I))=0\textrm{ for }i\neq 0.
\]
We make the following definition (see \cite[Def.~2.4.2]{Ayoub-BV}):
\begin{defn} $X\in \SmlSm(k)$ is \textit{$\Alb^{\log}$-trivial} if $X^\circ$ is affine, ${\rm NS}^1(X^\circ_{\overline{k}})=0$ and\[
	H^j_{\rm Zar}(\underline{X}, \mathcal{O}_{\underline{X}})  =0 \quad \text{ for } j>0.
	\]
\end{defn}

\begin{remark}
	If $\underline{X}=\Spec(R)$ is affine and $\partial X$ is supported on a principal divisor with global equation $f$, then $X^\circ = \Spec(R[\frac{1}{f}])$ is affine, in particular if ${\rm NS}^1(X^{\circ}_{\overline{k}}) =0$ we have that $X$ is $\Alb^{\log}$-trivial and $X^\circ$ satisfies the hypotheses of \cite[Proposition 2.4.4]{Ayoub-BV}.
\end{remark}

The main technical input of the proof of Theorem \ref{thm:derivedAlbanese} is the following result:

\begin{prop}\label{prop:key-prop-Alb-local-admissible} Let $X\in \mathbf{SmlSm}(k)$ be $\Alb^{\log}$-trivial, then the complex $\Q_{\ltr}(X)[0]$ is $BC$-admissible.
\begin{proof}
	We follow (with some modifications) the path of the proof of \cite[Proposition 2.4.4]{Ayoub-BV}. 
Since $\mathbb{Q}_{\ltr}(X)[0]$ is a compact object, by Lemma \ref{lem:equivalent-condition-admissible} it is enough to prove that 
	\[ \Ext^i_{\Shv^{ltr}}(\Q_{\ltr}(X), \Log_{\det}(I) ) =0, \quad \text{for $i>0$ and for every }I\in \RSC_{\et, \leq 1}\text{ injective.} \]
	By \cite[Proposition 4.3.2]{BPO} the $\Ext$ groups in $\Shv^{\ltr}$ can be computed as cohomology groups:
	\[ \Ext^i_{\Shv^{\ltr}}(\Q_{\ltr}(X), \Log_{\det}(I) ) =  H^i_{\dNis}(X, \Log_{\det}(I)),\]
	so we need to check that $H^i_{\rm dNis}(X, \Log_{\det}(I)) =0 $ for $i>0$. In order to control this cohomology, we look then at the adjunction map
	\begin{equation}\label{eq:key-prop.adjunction} 
		\omega^*_{\log} \G_m \tensor_{\mathbf{logCI}} \underline{\Hom}_{\Shv^{\ltr}}(\omega^*_{\log} \G_m, \Log_{\det}(I)) \to \Log_{\det} (I)
	\end{equation}
By Lemma \ref{lem;vanish-tensor} and Lemma \ref{lemma;vanish-ker-coker} below, we get that $H^j_{\dNis}(X, \Log_{\det}(I)) =0$ for $j>1$ and that we have a surjection
\begin{equation}\label{eq:final-surjection} 
		H^1_{\dNis} (X, \omega^*_{\log}\G_m \otimes_{\mathbf{logCI}} \underline{\Hom}_{\Shv^{\ltr}}( \omega^*_{\log}\G_m, \Log_{\det}(I)))  \to H^1_{\dNis}(X, \Log_{\det}(I)) \to 0. 
	\end{equation}
We are then left to show that the displayed morphism in \eqref{eq:final-surjection} is the zero map.

For every modulus pair $\kX\in \textbf{Comp}(X)$, the canonical map $\Q_{tr}(X) \to \Log_{\det}\bAlb_{\kX}$ gives for any $F\in \RSC_{\et, \leq1}$ a natural map (again we are using the fact that $\Log_{\det}$ is exact):
\[\Ext^1_{\RSC_{\et,\leq 1}} (\bAlb_{\kX}, F)\to \Ext^1_{\Shv^{ltr}} (\Log_{\det}\bAlb_{\kX}, \Log_{\det}F) \to H^1_{\dNis}(X, \Log_{\det}F), \]
hence, from \eqref{eq:final-surjection} we get a commutative diagram (cfr.~with the proof of \cite[2.4.4]{Ayoub-BV})
\begin{equation}\label{eq:key-diagram-ext-h1}\begin{small} 
\begin{tikzcd}
	\Ext^1_{\RSC_{\et, \leq 1}} (\bAlb_{\kX},  \omega_{\sharp}^{\log}(\omega^*_{\log}\G_m \otimes_{\mathbf{logCI}} \underline{\Hom}(\omega^*_{\log}\G_m, \Log_{\det}(I)))) \ar[r] \ar[d] & \Ext^1_{\RSC_{\et, \leq 1}} (\bAlb_{\kX}, I) \ar[d]\\
			H^1_{\dNis} (X, \omega^*_{\log}\G_m \otimes_{\mathbf{logCI}} \underline{\Hom}( \omega^*_{\log} \G_m, \Log_{\det}(I))) \ar[r] & H^1_{\dNis}(X,\Log_{\det}(I)).
\end{tikzcd}
\end{small}
\end{equation}
Since $I$ is injective in $\RSC_{\et, \leq 1}$, the term $\Ext^1_{\RSC_{\et, \leq 1}}(\bAlb_{\kX}, I)$ is zero. On the other hand, the left vertical map is an isomorphism by Lemma \ref{lem;iso-alb-H1}, which implies that \eqref{eq:final-surjection} is indeed the zero map. This finishes the reduction of the proof of \ref{prop:key-prop-Alb-local-admissible} to Lemma \ref{lemma;vanish-ker-coker}.
\end{proof}
\end{prop}
\begin{lemma}\label{lemma;vanish-ker-coker}
	Let $F\in \RSC_{\et,\leq 1}$ and let $N$ and $Q$ be respectively the kernel and the cokernel (computed in $\Shv^{\ltr}$) of the morphism 
	\begin{equation}\label{eq:key-prop.adjunction-general}
		\omega^*_{\log}\G_m \tensor_{\mathbf{logCI}}  \underline{\Hom}_{\Shv^{\ltr}}(\omega^*_{\log}\G_m, \Log_{\det} F) \to \Log_{\det} F.
	\end{equation} 
	Then for $j>0$ we have the following vanishing:
	\begin{align}
		\label{eq:vanishing-coh-N-Q} H^j_{\dNis}(X, N)=H^j_{\dNis}(X, Q)=0.
	\end{align}
\end{lemma}
\begin{proof}
	Since cohomology commutes with filtered colimits, let $F=\colim F_i$ with $F_i$ finitely generated $1$-reciprocity sheaves. For all $i$, let $K_i$ and $N_i$ be the kernel and the cokernel of the adjunction map
	\begin{equation}\label{eq;reduction-fin-gen}
	\omega^*_{\log}\G_m \tensor_{\mathbf{logCI}}  \underline{\Hom}_{\Shv^{\ltr}}(\omega^*_{\log}\G_m, \Log_{\det} F_i) \to \Log_{\det} F_i .
	\end{equation}
	As observed in Remark \ref{rmk;colimits-hom-tensor}, since filtered colimits are exact we have that $K=\colim K_i$ and $Q=\colim Q_i$, hence it is enough to show that for all $i$:
	\begin{align} &\label{eq;vanishN_i}H^j_{\dNis}(X, N_i) = 0, \quad j>0 \\
		&\label{eq;vanishQ_i}H^j_{\dNis}(X, Q_i) = 0, \quad j>0.
	\end{align}
	By Lemma \ref{lem;applying-uneful-tensor-hi-closed} we have that the left hand side of \eqref{eq;reduction-fin-gen} is isomorphic to $\omega^*_{\log}\bigl(\G_m \tensor_{\HI}  \underline{\Hom}_{\Shv^{\tr}}(\G_m, F_i^{\A^1})\bigr)$. 
Let $K_i$ and $R_i$ be the kernel and the cokernel of the adjunction map
	\[\psi_i: \G_m \tensor_{\HI}  \underline{\Hom}_{\Shv^{\tr}}(\G_m, F_i^{\A^1})
	\to F_i^{\A^1}.\] We have the following diagram:
	\begin{equation}\label{eq:comparison-A1}
		\small \begin{tikzcd}
			0\ar[r]& N_i \ar[r]\ar[d,hook] & \omega^*_{\log}\G_m \tensor_{\mathbf{logCI}} \underline{\Hom}_{\Shv^{\ltr}}(\omega^*_{\log}\G_m, \Log_{\det} F_i)  \ar[r,"\varphi_i"]  \ar[d,"\simeq"] & \Log_{\det} F_i\ar[r,two heads]\ar[d] & Q_i \ar[r] \ar[d]&0\\
			0\ar[r] & \omega^*_{\log} K_i \ar[r] & \omega^*_{\log}\big(\G_m \tensor_{\HI} \underline{\Hom}_{\Shv^{\tr}}(\G_m, F_i^{\A^1})\big) \ar[r,"\omega^*\psi_i"] &  \omega^*_{\log}F_i^{\A^1} \ar[r] & \omega^*_{\log} R_i \ar[r] &0.
		\end{tikzcd}
	\end{equation}  
	By the cancellation theorem \cite{Voecancel},  
	$\underline{\Hom}_{\Shv^{\tr}}(\G_m,\G_m\otimes_{\HI} M)\cong M$ for $M\in \HI$,
	and $\underline{\Hom}_{\Shv^{\tr}}(\G_m,\_)$ is exact as an endo-functor on $\HI$.
	Hence we get 
	$\underline{\Hom}_{\Shv^{\tr}}(\G_m,K_i)=\underline{\Hom}_{\Shv^{\tr}}(\G_m,R_i)=0$, 
in particular the sheaves $K_i$ and $R_i$ are birational sheaves in the sense of \cite{KS} (see \cite[Proposition 2.5.2]{KS}).
	In particular, since $\omega^*_{\log}$ is exact and preserves injectives, by \cite[Proposition 14.23]{MVW} and \cite[Proposition 2.3.3]{KS} that
	\begin{align} &\label{eq;vanishK_i}H^j_{\dNis}(X, \omega^*_{\log} K_i) = H^j_{\rm Nis}(X^{\circ}, K_i) = 0, \quad j>0 \\
		&\label{eq;vanishR_i}H^j_{\dNis}(X, \omega^*_{\log}R_i) = H^j_{\Nis}(X^{\circ}, R_i)= 0, \quad j>0.
	\end{align}
	Since $N_i$ is a subsheaf of $\omega^*_{\log}K_i$, 
	$\omega_{\sharp}^{\log}N_i$ is a subsheaf of $K_i$, so it is a birational sheaf, in particular it is an object of $\HI$ by \cite[Proposition 2.3.3 (a)]{KS} so $N_i\cong \omega^*_{\log}\omega_{\sharp}^{\log}N_i$ by \eqref{eq;only-one-homotopy-invariant}.
	Therefore the same argument gives the vanishing \eqref{eq;vanishN_i}.
	
	Let $H_i$ be the kernel of the map $F_i\to F_i^{\A^1}$. By a snake lemma argument on \eqref{eq:finpresA1}, there exists a lattice $L_i'$ and $r_i\geq 0$ such that $\G_a^{r_i}/L_i'\cong H_i$. By the exactness of $\Log_{\det}$, we have that $\Log_{\det}\G_a^{r_i}/\omega^*L'\cong \Log_{\det}H_i$. Since $L'$ is a lattice, $H^j_{\dNis}(X,\omega^*L') = H^j_{\Nis}(X^{\circ},L')=0$ for $j>0$ and by \cite[Corollary 6.8]{RulSaito} with $q=0$ 
	we have that (see \eqref{eq;colimdiv}):
	\begin{equation}\label{eq;vanishGa}
		H^j_{\dNis}(X,\Log_{\det}\G_a)=\colim_{Y\in X_{\div}^{\Sm}} H^j_{\Nis}(\underline{Y},\cO_{\underline{Y}}).
	\end{equation}
	By the comparison of Zariski cohomology with Nisnevich cohomology for coherent sheaves we have that for all $Y\in X_{\div}^{\Sm}$:\[
	H^j_{\Nis}(\underline{Y},\cO_{\underline{Y}})\cong H^j_{\Zar}(\underline{Y},\cO_{\underline{Y}}).
	\]
	By definition the map $\underline{Y}\to \underline{X}$ is the composition of blowups in smooth centers, hence the well know blow-up formula (see e.g. \cite[Corollary IV.1.1.11]{Gros1985}) implies:
	\[H^j_{\Zar}(\underline{Y},\cO_{\underline{Y}})\cong H^j_{\Zar}(\underline{X},\cO_{\underline{X}})=0,\;\; j>0,
	\]
	where the last vanishing comes from the fact that $X$ was taken $\Alb^{\log}$-trivial. In particular, we conclude that
	\begin{equation}\label{eq;vanishH}
		H^i_{\dNis}(X,\Log_{\det}H_i)=0\;\text{ for }i\not=0.
	\end{equation}
	
	From the diagram \eqref{eq:comparison-A1} and a snake lemma argument, we get the following short exact sequence:
	\begin{equation}\label{eq:ker-Q1} 
		0\to \omega^*_{\log}K_i/N_i \to \Log_{\det}H_i \to \ker(Q_i\to \omega^*_{\log}R_i) \to 0.
	\end{equation}
	Now by \eqref{eq;vanishK_i} and \eqref{eq;vanishN_i} we have that $H^j_{\dNis}(X,\omega^*_{\log}K_i/N_i)=0$ for $j>0$, so by \eqref{eq;vanishH}, \eqref{eq;vanishR_i} and \eqref{eq:ker-Q1} we deduce \eqref{eq;vanishQ_i}.
\end{proof}

Given Proposition \ref{prop:key-prop-Alb-local-admissible}, we can show the following 
\begin{lemma}\label{lem:enough-admissible} The category $\Shv_{\dNis}^{\ltr}(k, \Q)$ is generated by the set of $\Alb^{\log}$-trivial objects of $\Shv_{\dNis}^{\ltr}(k, \Q)$.
	\begin{proof} The category $\Shv_{\dNis}^{\ltr}(k, \Q)$ is compactly generated, and a set of compact generators is given by $\Q_{\rm ltr}(X)[i]$, for $X\in \mathbf{SmlSm}$ and $i\in \Z$. By Proposition \ref{prop:key-prop-Alb-local-admissible}, it is then enough to show that any $X\in \mathbf{SmlSm}$ can be covered (even Zariski-locally) by $X_i\in \mathbf{SmlSm}$ which are $\Alb^{\log}$-trivial. Let $\underline{U_i}$ be a Zariski cover of $\underline{X}$ such that $|\partial X|_{|\underline{U_i}}$ is principal. By \cite[Corollary 2.4.6]{Ayoub-BV}), we can cover each $\underline{U_i}$ by affine $\underline{U_{ij}}$ such that $\NS^1((\underline{U_{ij}})_{\ol{k}})=0$, and since $|\partial X|_{\underline{U_i}}$ is principal, $|\partial X|_{\underline{U_{ij}}}$ is again principal. Considering the log schemes $U_{ij}:=(\underline{U_{ij}},\partial X_{|U_{ij}})$, we have that $\NS^1(\underline{(U_{ij}})_{\ol{k}})\to \NS^1((U_{ij}^{\circ})_{\ol{k}})$ is surjective by \cite[Example 10.3.4]{Fulton}, hence $\NS^1((U_{ij}^{\circ})_{\ol{k}})=0$. We conclude that $\{U_{ij}\}$ is a Zariski cover of $X$ by $\Alb^{\log}$-trivial log schemes.
\end{proof}
\end{lemma}

\begin{proof}[Proof of Theorem \ref{thm:derivedAlbanese}] From Lemma \ref{lem:enough-admissible} and Proposition \ref{prop:key-prop-Alb-local-admissible}, we have that the $\infty$-category $\cD(\Shv_{\dNis}^{\ltr}(k, \Q)))$ is generated by a set of compact $BC$-admissible objects concentrated in degree zero. The existence of the derived log Albanese functor $L\Alb^{\log}$ as pro-left derived functor of $\Alb^{\log}$ follows then from Theorem \ref{thm:gen-thm-existenceLF}, and by constrution it is equivalent to the pro-left adjoint of the functor $\cD(\Log_{\det})$. 

We are left to show that the functor $L\Alb^{\log}$ factors through the localization \[ L_{\bcube} \colon \cD(\Shv_{\dNis}^{\ltr}(k, \Q)) \to \logDM(k,\Q).\]
Recall that $\logDM(k,\Q)$ is obtained as localization of $\cD(\Shv_{\dNis}^{\ltr}(k, \Q))$ with respect to the class of maps:
	\begin{equation}\tag{CI} \Q_{\rm ltr}(X\times \bcube)[n]\to \Q_{\rm ltr}(X)[n]
	\end{equation}
for $X$ in $\mathbf{SmlSm}(k)$. From the proof of Lemma \ref{lem:enough-admissible}, we can suppose that $X$ is $\Alb^{\log}$-trivial (this is the exact analogue of \cite[2.4.1]{Ayoub-BV}). We are then left to show that $L\Alb^{\log} (\Q_{\rm ltr}(X\times \bcube)\to \Q_{\rm ltr}(X))$ is contractible for $X$ an $\Alb^{\log}$-trivial object. Since $\Q_{\ltr}(X)$ is $BC$-admissible by Proposition \ref{prop:key-prop-Alb-local-admissible}, we have by Remark \ref{rmk;BC-adm-is-LF-trivial} that
	\[ L\Alb^{\log}( \Q_{\rm tr}(X)[n]) = ``\lim_{i}"\bAlb_{\kX^{(i)}}[n],\quad \textrm{for any choice of }\kX\in \mathbf{Comp}(X).\] 
	Note also that if $X$ is $\Alb^{\log}$-trivial, so is $X\times \bcube$. Indeed, $(X\times \bcube)^\circ = X^\circ \times \A^1$ is affine if $X^\circ$ is affine, $\NS^1((X^\circ \times \A^1)_{\overline{k}})\cong \NS^1(X^{\circ}_{\overline{k}})$ and $H^i_{\Zar}(\underline{X}\times \P^1,\mathcal{O}_{\underline{X}\times \P^1})\cong H^i_{\Zar}(\underline{X},\mathcal{O}_{\underline{X}})$ for all $i$. Therefore\[ L\Alb^{\log}( \Q_{\rm ltr}(X\times \bcube)[n]) = ``\lim_{i}"\bAlb_{\kX^{(i)}\times\bcube},\quad \textrm{for any choice of }\kX\in \mathbf{Comp}(X).\]  On the other hand, by construction we have
	\[\bAlb_{\kX^{(i)}\times \bcube} \cong \bAlb_{\kX^{(i)}},\]
proving the factorization. The pro-adjunction now is formal since for $X\in \cD(\RSC_{\et,\leq 1})$, we have that $\cD(\Log_{\det})(X)\in \cD(\logCI)$, hence it is $(\dNis,\bcube)$-local by \cite[Corollary 5.5]{BindaMerici}, so:\[\cD(\Log_{\det})(X)\simeq i_{\bcube}L_{\bcube}\cD(\Log_{\det})(X)\simeq i_{\bcube}\omega_{\leq 1}^{\lDM}(X),\] hence for any $Y\in \logDM(k,\Q)$
\begin{align*}
\Map_{\logDM(k,\Q)}(Y,\omega_{\leq 1}^{\lDM}(X)) &\simeq
\Map_{\cD(\Shv_{\det}^{\ltr}(k,\Q))}(i_{\bcube}Y,\cD(\Log_{\det})(X)) \\
&\simeq \Map_{\Pro\cD(\RSC_{\et,\leq 1}(k,\Q))}(L\Alb^{\log}(Y),X)
\end{align*}
as required.
\end{proof}

\subsection{Some computations}\label{sec:computations}

Recall from \cite[Theorem 9.2.3]{BVKahn} (see also \cite[Theorem 1.1]{Doosung} for a statement in a language more similar to ours) that 
\begin{equation}\label{eq;LAlb-usual}
	L\Alb(\omega(X))\simeq \Alb(\omega(X))\oplus \NS^*(\omega(X))[1]\footnote{The splitting is not stated but it can easily be deduced by \cite[Prop (1.2) and Rem (1.4)]{DeligneLefschetz}, similarly to Theorem \ref{lem:split}},
\end{equation} 
where $\NS^*(\omega(X))$ is the torus dual to the N\'eron-Severi. The goal of this section is to give an explicit description of $L\Alb^{\log}(X)$: for $X=(\underline{X},\partial X)$, recall that $\omega(X)$ is defined as $\underline{X}-|\partial X|$. We will prove the following result:

\begin{thm}\label{thm;compute-LAlbX}
	Let $X\in \SmlSm$ geometrically connected and $(\overline{X},D)$ a Cartier compactification of $X$. Then we have that\[
	L_i\Alb^{\log}(X)\cong \begin{cases}
		``\lim"(H^i(\overline{X},\cO_{\overline{X}}(nD))^\vee\otimes_k\G_a) &\textrm{for }2\leq i\leq \dim(X)\\
		``\lim"\Bigl((H^1(\overline{X},\cO_{\overline{X}}(nD))/H^1(\overline{X},\cO_{\overline{X}}))^\vee\otimes_k\G_a\Bigr) \oplus \NS^*(\omega(X))&\textrm{for }i=1\\
		\Alb^{\log}(X)&\textrm{for }i=0\\
		0&\textrm{otherwise}.
	\end{cases}
	\]
	Moreover, the canonical map $L\Alb^{\log}(X)\to \oplus_{i=0}^{\dim(X)}L_i\Alb^{\log}(X)[i]$ is an equivalence.
\end{thm}

We start with the following observation:

\begin{prop}\label{prop;suslin-complex-1-motivic}
	The inclusion $\cD(i_{\A^1})\colon \cD(\HI_{\et,\leq 1})\to \cD(\RSC_{\et,\leq 1})$ has a left adjoint $L_{\A^1}^{\leq 1}$ such that $\pi_i(L_{\A^1}^{\leq 1})(F)=h_i^{\A^1}(F)$ (the Suslin hyperhomology).
	\begin{proof}
		By Proposition \ref{prop:hom-dim}, the inclusion $\cD(\RSC_{\et,\leq 1})\to \cD(\Shv_{\et}^{\tr}(k,\Q))$ is fully faithful, then for $F\in \cD(\RSC_{\et,\leq 1})$ and $H\in \cD(\HI_{\et,\leq 1})$:
		\begin{align*}
			\Map_{\cD(\RSC_{\et,\leq 1})}(F,\cD(i_{\A^1})(H)) \simeq \Map_{\cD(\Shv_{\et}^{\tr}(k,\Q))}(F,\cD(i_{\A^1})(H))\\
			\simeq \Map_{\mathcal{DM}^{\eff}(k,\Q))}(C_{*}^{\A^1}(F),H) \simeq \Map_{\cD(\HI_{\et,\leq 1})}(L\Alb C_{*}^{\A^1}(F),H)
		\end{align*}
		which implies that $L_{\A^1}^{\leq 1}$ exists and it coincides with $L\Alb C_{*}^{\A^1}$. 
We can write $F=\hocolim_{i,n} F_i[n]$ with $F_i\in \RSC_{\leq 1}^\star$.
Since $C_{*}^{\A^1}\colon \cD(\Shv_{\et}^{\tr}(k,\Q))\to \mathcal{DM}^{\eff}(k,\Q))$
commutes with all (homotopy) colimits as a left adjoint,
we have\[
		C_{*}^{\A^1}(F) \simeq \hocolim C_{*}^{\A^1}(F_i[n]) \overset{(*)}{\simeq} \hocolim (h_0^{\A^1}(F_i)[n]),	\]
where $(*)$ follows from Remark \ref{rmk:Fsab}. 
Hence the homotopy groups of $C_{*}^{\A^1}(F)$ are $1$-motivic so that $\pi_iL_{\A^1}^{\leq 1}(F) = \pi_i L\Alb C_{*}^{\A^1}(F)=\pi_i C_{*}^{\A^1}(F) = h_i^{\A^1}(F)$.	
	\end{proof}	
\end{prop}

\begin{lemma}
	Let $\Pro L_{\A^1}^{\leq 1}\colon \Pro\cD(\RSC_{\et,\leq 1})\to \Pro \cD(\HI_{\et,\leq 1})$. For $X\in \SmlSm(k)$ geometrically connected, $\Pro L_{\A^1}^{\leq 1}L\Alb^{\log}(X)$ is a constant pro-object and \[
	\Pro L_{\A^1}^{\leq 1}L\Alb^{\log}(X)\simeq c(L\Alb(\omega(X))),
	\]
	where $c$ denotes the constant pro-object.
	\begin{proof}
		Let $\Q_{\ltr}(Y_{\bullet})\to \Q_{\ltr}(X)$ be a resolution in $\Shv_{\det}^{\ltr}(k,\Q)$ by $\Alb$-trivial objects. Then $L_i\Alb^{\log}(X)=\pi_i(\Alb^{\log}(Y_{\bullet}))$. On the other hand, by construction, $\Q_{\tr}(\omega(Y_{\bullet}))\to \Q_{\tr}(\omega(X))$
is a resolution in $\Shv_{\et}^{\tr}(k,\Q)$ by affine $\NS^1$-local objects in the sense of \cite{Ayoub-BV}, so $L_i\Alb(\omega(X))=\pi_i\Alb(\omega(Y_\bullet))$. By Proposition \ref{prop:max-semiab-quotient}  and \eqref{eq:can-map-Alb-Albsemiab},
for $\kY_\bullet$ Cartier compactifications of $Y_\bullet$, we have a fiber-cofiber sequence:
\begin{equation}\label{eq;Fibbullet}
		``\lim_n" U(\kY^{(n)}_{\bullet})\otimes_k\G_a\to \Alb^{\log}(Y_\bullet)\to c(\Alb(\omega(Y_\bullet))),
		\end{equation}
where $U(\kY^{(n)}_{\bullet})$ are finite dimensional $k$-vector spaces. 
		The complex $\Alb(\omega(Y_\bullet))$ is $\A^1$-local, so we have that \[\Pro L_{\A^1}^{\leq 1}c(\Alb(\omega(Y_{\bullet})))\simeq c(\Alb(\omega(Y_{\bullet}))) = c(L\Alb(\omega(X)))\] and for all $i,n$, there exist $r_{i,n}$ such that $U(\kY^{(n)}_{i})\otimes_k\G_a=\G_a^{r_{i,n}}$,
 so we conclude since:\[\Pro L_{\A^1}^{\leq 1}(``\lim_n"(\kY^{(n)}_{\bullet}))= ``\lim_n" L_{\A^1}^{\leq 1}(U(\kY^{(n)}_{\bullet}))=0.\]
	\end{proof}
\end{lemma} 
The lemma above gives a natural map $L\Alb^{\log}(X)\to cL\Alb(\omega(X))$. Let
\begin{equation}\label{eq;Fib}
\textrm{Fib}(X):=\hofib(L\Alb^{\log}(X)\to cL\Alb(\omega(X))).
\end{equation}
In view of \eqref{eq;Fibbullet}, we have that
\begin{equation}\label{eq;fib-is-vg}
\textrm{Fib}(X)\simeq ``\lim_n"U(\kY^{(n)}_\bullet)\otimes_k\G_a.
\end{equation}
since for $\Q_{\ltr}(Y_{\bullet})\to \Q_{\ltr}(X)$ an $\Alb$-trivial resolution and $\kY_\bullet$ a Cartier compactification of $Y_\bullet$, the map $\bAlb(\kY^{(n)}_i)\to \Alb(\omega(Y_i))$ is surjective with kernel $U(\kY^{(n)}_i)$ for all $i$.

\begin{defn}
	Let $\langle\G_a\rangle\subseteq \cD(\RSC_{\et,\leq 1})$ be the stable $\infty$-subcategory generated by direct sums of $\G_a$ and let $\langle\G_a\rangle^\omega$ be the full subcategory of compact objects. In particular, $F\in \langle\G_a\rangle^\omega$ if and only if there exist $s\leq t\in \Z$ and $r_s,\ldots, r_t\geq 0$ such that \[F\simeq \bigoplus_{i=s}^t \G_a^{r_i}[i].\]
\end{defn}
\begin{remark}
	Let $\Pro(\langle\G_a\rangle^\omega)\subseteq\Pro\cD(\RSC_{\et,\leq 1})$, then for all $X$ as above $\textrm{Fib}(X)\in \Pro(\langle\G_a\rangle^\omega)$
\end{remark}
\begin{remark}
	Let $\mathbf{Vect}_{\rm fd}$ be the category of finite dimensional $k$-vector spaces. By Proposition \ref{prop:hom-dim} and \cite[VII, n.7, Proposition 8]{SerreGACC}:
\[\Map_{\cD(\RSC_{\et,\leq 1})}(\G_a^r,\G_a^s[i])\simeq \Hom_{\mathbf{Vect}_{\rm fd}}(k^r,k^s)[i]\simeq \Map_{\cD(\mathbf{Vect}_{\rm fd})}(k^r,k^s[i]).\]
	In particular, the functor $V\mapsto \Spec(k[V^\vee])$ induces an equivalence of $\infty$-categories \[\_\otimes_k\G_a\colon \cD(\mathbf{Vect}_{\rm fd})\to \langle\G_a\rangle^\omega\]
	with quasi inverse given by $R\Gamma(\Spec(k),\_)$.
\end{remark}
Let $(\_)^\vee\colon \mathbf{Vect}_{\rm fd}\to \mathbf{Vect}_{\rm fd}^{\rm op}$ denote the equivalence given by the dual vector space. It induces an equivalence: \[
(\_)^\vee\colon \Pro\cD(\mathbf{Vect}_{\rm fd})\to \Ind\textrm{-}\cD(\mathbf{Vect}_{\rm fd}^{\rm op}).
\]
\begin{remark}\label{rmk;vector-groups}
	There is a commutative diagram of stable $\infty$-category:\[
	\begin{tikzcd}
		\Pro\langle\G_a\rangle^\omega\ar[rrr,swap,"{\Pro R\Gamma(\Spec(k),\_)}"]\ar[rrrr,bend left = 10,"{\Map(\_,\G_a)}"] &&&\Pro\cD(\mathbf{Vect_{\rm fd}})\ar[r,swap,"(\_)^\vee"]&\Ind\textrm{-}\cD(\mathbf{Vect}_{\rm fd}^{\rm op})
	\end{tikzcd}
	\]
where $\Map$ denotes the mapping space in $\Pro\cD(\RSC_{\et,\leq 1})$.
	This easily follows from the fact that any map of sheaves $f\colon \G_a^r\to \G_a^s$ is indeed a map of vector groups, hence since $\G_a=k\otimes_k\G_a$, we have a commutative diagram in $\cD(\mathbf{Vect}_{\rm fd})$:
\[	\begin{tikzcd}
		\Map(\G_a^s,\G_a)\ar[r,"\simeq"]\ar[d,"f^*"] &R\Gamma(\Spec(k),\G_a^s)^\vee\ar[d,"f(k)^t"]\\
		\Map(\G_a^r,\G_a)\ar[r,"\simeq"] &R\Gamma(\Spec(k),\G_a^r)^\vee.
	\end{tikzcd}	\] 
\end{remark}
\begin{prop}\label{prop;LAlb-i-geq-2}
Let $X\in \SmlSm(k)$ geometrically connected. Then for any $(\overline{X},D)$ Cartier compactification of $\underline{X}$ and $i\geq 2$, we have that\[
L_i\Alb^{\log}(X) \simeq ``\lim"(H^i(\overline{X},\cO_{\overline{X}}(nD))^\vee)\otimes_k\G_a.
\]
In particular, $L_i\Alb^{\log}(X)=0$ for $i\geq \max(\dim(X),2)$.
\begin{proof}
	By \eqref{eq;LAlb-usual}, for $X$ as above we have that for $i\geq 2$\[
	L_i\Alb^{\log}(X)\simeq \pi_i\textrm{Fib}(X).
	\] 
	By Remark \ref{rmk;vector-groups} we have that \begin{equation}\label{eq;fib-duality}
	\textrm{Fib}(X)\simeq \Map (\textrm{Fib}(X),\G_a)^{\vee}\otimes_k\G_a.
\end{equation}
where for $V=\colim V_i\in \Ind\textrm{-}\cD(\mathbf{Vect}_{\rm fd}^{\rm op})$, $V^{\vee}= \plim V_i^\vee$. In particular, it is enough to compute $\pi_{-i}\Map_{\Pro\cD(\RSC_{\et, \leq 1})}(\textrm{Fib}(X),\G_a)$ for $i\geq 2$, which again by \eqref{eq;LAlb-usual} agree with $\pi_{-i}\Map_{\Pro\cD(\RSC_{\et, \leq 1})}(L\Alb^{\log}(X),\G_a)$.
By Theorem \ref{thm:derivedAlbanese} and \cite[Theorem 9.7.1]{BPO} and 
\eqref{omegaCIGa}.  
	\begin{equation}\label{eq;map-LAlb-coherent-coh}
		\Map_{\Pro\cD(\RSC_{\et, \leq 1})}(L\Alb^{\log}(X),\G_a)\simeq \Map_{\logDM}(M(X),\Log_{\det}\G_a)\simeq R\Gamma(\underline{X},\cO_{\underline{X}}).
	\end{equation} 
Let now $(\overline{X},D)$ be a Cartier compactification of $\underline{X}$. This  gives  an isomorphism:\[
	R\Gamma(\underline{X},\cO_{\underline{X}})\simeq \colim_n R\Gamma(\overline{X},\cO_{\overline{X}}(nD)).\]
and the right-hand side is in $\Ind\textrm{-}\cD(\mathbf{Vect}_{\rm fd}^{\rm op})$, which completes the proof.
\end{proof}
\end{prop}

\begin{proof}[Proof of Theorem \ref{thm;compute-LAlbX}]
The only case left is $i=1$: we consider the long exact sequence of homotopy groups arising from \eqref{eq;Fib}. The map $L_1\Alb(\omega(X))\to \pi_0(\textrm{Fib}(X))$ is zero
since by \eqref{eq;LAlb-usual}, $L_1\Alb(\omega(X))$ is a torus, so 
we get a short exact sequence in $\pro\RSC_{\et,\leq 1}$:
\begin{equation}\label{eq:ses-L1}
		0\to \pi_1\textrm{Fib}(X)\to L_1\Alb^{\log}(X)\to L_1\Alb(\omega(X))\overset{\eqref{eq;LAlb-usual}}{\cong}c(\NS^*(\omega(X)))\to 0.
		\end{equation}
		Since {$\pi_1\textrm{Fib}(X) = \plim V_i$ with $V_i$ vector groups, we have 
		\begin{align*}
		\Ext^1_{\pro\RSC_{\et, \leq 1}}(\NS^*(\omega(X)),\pi_1\textrm{Fib}(X)) = \pi_1\bigl(\lim\Map_{\cD(\RSC_{\et, \leq 1})}(\NS^*(\omega(X),V))\bigr)
		\end{align*}}
		Since $\NS^*(\omega(X))$ is a torus, by \cite[VII, n. 6, Proposition 7]{SerreGACC} {and Proposition \ref{prop:ext}}, we have that for all vector groups $V$\[
		\Map_{\cD(\RSC_{\et, \leq 1})}(\NS^*(\omega(X),V))\simeq 0,
		\]
		so \eqref{eq:ses-L1} above splits and
		\begin{equation}\label{eq;prelim-L1Alb}
			L_1\Alb^{\log}(X)\cong \pi_1\textrm{Fib}(X)\oplus \NS^*(\omega(X)).
		\end{equation}
		Moreover, we have 
\[\Map(\pi_1\textrm{Fib}(X),\G_a)\simeq \Map(L_1\Alb^{\log}(X),\G_a)\simeq\Hom_{\pro\RSC_{\et,\leq 1}}(L_1\Alb^{\log}(X),\G_a)[0],\] 
where $\Map$ denotes the mapping space in $\Pro\cD(\RSC_{\et,\leq 1})$.
Hence \eqref{eq;prelim-L1Alb} gives:
		\begin{equation}\label{eq;prelim-L1Alb-II}
			\pi_1L\Alb^{\log}(X)\cong (\Hom(L_1\Alb^{\log}(X),\G_a)^\vee)\otimes_k\G_a\oplus \NS^*(\omega(X)).
		\end{equation}
		Let us compute $\Hom(L_1\Alb^{\log}(X),\G_a)$: by \eqref{eq;map-LAlb-coherent-coh} and the degeneration of \eqref{eq;spectral-sequence-ext-coherent} we have an exact sequence:
		\begin{equation}\label{eq;exact-sequence-LiAlb}
			0\to \Ext^1(\Alb^{\log}(X),\G_a)\to H^1(X,\cO_X)\to \Hom(L_1\Alb^{\log}(X),\G_a)\to 0.
		\end{equation}
In particular, we have a similar exact sequence for the log scheme $(\overline{X},\triv)$, which we now investigate. For $\overline{X}$ proper we have by construction (see Theorem \ref{thm:logAlb}) that $\Alb^{\log}(\overline{X})$ is the constant pro-object $\Alb(\overline{X})$, so there is a surjective map\[
\Alb^{\log}(X)\twoheadrightarrow \Alb^{\log}(\overline{X})
\]
whose kernel is an extension of the torus $T:=\ker(\Alb(\omega(X))\to \Alb(\overline{X}))$ by the pro-vector group $``\lim"U(\Xb,nD)\otimes_k\G_a$, where $U(\Xb,nD)$ comes from Definition \ref{defn:COmega}. 
For $i\geq 1$, we have that $\colim\Ext^i(U(\Xb,nD),\G_a)=0$ and $\Ext^i(T,\G_a)=0$, so we have a surjective map:
\begin{equation}\label{eq;surj-ext1-L0Alb}
	\Ext^1(\Alb^{\log}(\overline{X}),\G_a)\twoheadrightarrow \Ext^1(\Alb^{\log}(X),\G_a).
\end{equation}  

		Combining \eqref{eq;exact-sequence-LiAlb}, and \eqref{eq;surj-ext1-L0Alb} we have a commutative diagram with exact rows:\[
\begin{tikzcd}
	0\ar[r] &\Ext^1(\Alb^{\log}(\overline{X}),\G_a)\ar[r]\ar[d,two heads] &H^1(\overline{X},\cO_X)\ar[r]\ar[d] &\Hom(L_1\Alb^{\log}(\overline{X}),\G_a)\ar[r]\ar[d] &0\\
	0\ar[r] &\Ext^1(\Alb^{\log}(X),\G_a)\ar[r] &H^1(X,\cO_X)\ar[r] &\Hom(L_1\Alb^{\log}(X),\G_a)\ar[r] &0
\end{tikzcd}
\]
so to conclude it is enough to show that $L_1\Alb^{\log}(\overline{X})\cong \NS^*(\overline{X})$, which implies that $\Hom(L_1\Alb^{\log}(\overline{X}),\G_a) = 0$, so the diagram above implies \[
\Hom(L_1\Alb^{\log}(X),\G_a)\simeq \coker(H^1(\overline{X},\cO_X)\to H^1(X,\cO_X))
\]
and we will conclude by duality and \eqref{eq;prelim-L1Alb-II}.

By \eqref{eq;prelim-L1Alb}, it is enough to show that $\pi_1\textrm{Fib}(\overline{X})=0$. By \eqref{eq;fib-duality}, we have \[\pi_1\textrm{Fib}(\overline{X}) \cong \pi_1(\Map(\textrm{Fib}(\overline{X}),\G_a)^\vee\otimes_k\G_a)\cong (\pi_{-1}\Map(\textrm{Fib}(\overline{X}),\G_a)^\vee\otimes_k\G_a),\]
	so it is enough to show that $\pi_{-1}\Map(\textrm{Fib}(\overline{X}),\G_a)=0$. By \eqref{eq;Fib} for $\overline{X}$ we have a fiber-cofiber sequence:
	\begin{equation}\label{eq;compute-pi-1-map-fib}
	\Map(L\Alb(\overline{X}),\G_a)\to \Map(L\Alb^{\log}(\overline{X}),\G_a)\to \Map(\textrm{Fib}(\overline{X}),\G_a).
	\end{equation}
By \eqref{eq;LAlb-usual}, we have \[
	\pi_{-1}\Map(L\Alb(\overline{X}),\G_a)\cong \pi_0\Map(\NS^*(\overline{X}),\G_a)\oplus \pi_{-1}\Map(\Alb(\overline{X}),\G_a).
	\]
	Since $\NS^*(\overline{X})$ is a torus and $\Alb(\overline{X})$ is an abelian variety, we have:\[
	\Map(\NS^*(\overline{X}),\G_a)\simeq 0\quad\textrm{and}\quad \Ext^1(\Alb(\overline{X}),\G_a)\cong H^1(\overline{X},\cO_{\overline{X}}),
	\]
where the last isomorphism is classical (see \cite[VII, n.17, Th\'eor\`eme 7]{SerreGACC}), and\[
\pi_{-2}\Map(L\Alb(\overline{X}),\G_a) = 0.
\]
Finally, $\pi_{-1}\Map(L\Alb^{\log}(\overline{X}),\G_a)\cong H^1(\overline{X},\G_a)$ 
by \eqref{eq;map-LAlb-coherent-coh}, so the map
\[\pi_{-1}\Map(L\Alb(\overline{X}),\G_a)\to \pi_{-1}\Map(L\Alb^{\log}(\overline{X}),\G_a)\]
 in the long exact sequence of homotopy groups of \eqref{eq;compute-pi-1-map-fib} is an isomorphism, which implies the desired vanishing.
	\end{proof}
\begin{remark}\label{rmk:main_theorem_optimal}
	We observe from Theorem \ref{thm;compute-LAlbX} two extreme cases: if $X$ is affine, we have that 
$H^i(X,\cO_{X})\simeq \underset{n}{\colim} H^i(\overline{X},\cO_{X}(nD))=0$ for $i\geq 1$, so \[
	L_i\Alb^{\log}(X)=\begin{cases}
		\Alb^{\log}(X)&\textrm{if }i=0\\
		\NS^*(X)&\textrm{if }i=1\\
		0&\textrm{otherwise}.
	\end{cases}
	\]
	In particular,  $L_i\Alb^{\log}(X)$ is constant for $i\geq 1$.
	For $X$ proper, $H^1(X,\cO_X)\cong H^1(\overline{X},\cO_{\overline{X}})$, so \[
	L_i\Alb^{\log}(X)=\begin{cases}
		\Alb(X)&\textrm{if }i=0\\
		\NS^*(X)&\textrm{if }i=1\\
		(H^i(\overline{X},\cO_{X})^\vee)\otimes_k\G_a&\textrm{if }i\geq 2.
	\end{cases}
	\]
	In this case, $L\Alb^{\log}(X)$ is a constant pro-object. This shows that Proposition \ref{prop;comparison-usual-LAlb} cannot be extended to the whole $\mathcal{DM}^{\eff}(k,\Q)$: in general, if $M\in \mathcal{DM}^{\eff}(k,\Q)$, $L\Alb^{\log}(\omega^*M)$ is not equal to $L\Alb(M)$: the difference is controlled by coherent cohomology of degree $\geq 2$.
\end{remark}
\subsection{Open questions}\label{BlochSrinivas} We end this Section with the following observation. It seems to be an interesting question to determine under which conditions $L_i\Alb^{\log}$ is a constant pro-object.

This is related to the following problem: let $X'\to X$ be a desingularisation of a $d$-dimensional, integral variety over a field $k$, and let $D$ be an effective Cartier divisor on $X'$ covering the exceptional fibre, and assume that $\codim_X(\pi(D))\geq 2$. Let $rD$ denote the $r$-th infinitesimal thickening of $D$ and $F^dK_0(X',rD)$ the subgroup of the relative $K$-group $K_0(X', rD)$ generated by the cycle classes of closed points of $X' - |D|$, for each $r \geq 1$.

Bloch and Srinivas conjectured (see \cite[p. 6]{SrinivasZeroCyclesII}) that the pro-object $``\lim_n"F^dK_0(X',nD)$ is essentially constant and equal to $F^dK_0(X)$. The Bloch--Srinivas conjecture was proved for normal surfaces by Krishna--Srinivas \cite[Theorem 1.1]{KrishnaSrinivas}, and for $\ch(k)=0$ it was later extended to higher dimensional projective and affine varieties over an algebraically closed field by Krishna \cite[Theorem 1.1]{KrishnaThreefold} \cite[Thorem 1.2]{KrishnaArtinReese} and Morrow \cite[Theorem 0.1, \textit{(iii), (iv)}]{MorrowBS}\footnote{The conjecture is indeed true in a more general class of examples: the interested reader can check \cite[Theorem 0.1 \textit{(i)-(vii)}]{MorrowBS}}.

The proof of \cite{MorrowBS} indeed relies on a natural reformulation of the Bloch--Srinivas conjecture for the Chow groups with modulus:

\begin{thm}[cfr. {\cite[Theorem 0.3]{MorrowBS}}]\label{thm;BS-Chow} Let $k$ be an algebraically closed field of characteristic zero and $\pi \colon X' \to X$ and $D$ be as above and assume that $X$ is projective. Then the pro-object $``\lim"_n\CH_0(X,nD)$ is constant with stable value equal to the Levine--Weibel Chow group of zero cycles $\CH^{LW}_0(Y)$ of \cite{LW} (see also \cite{BK}).
\end{thm}

By the universal property of the Albanese map, we deduce that if in the situation of Theorem \ref{thm;BS-Chow} we assume that $D$ is a simple normal crossing divisor, the pro-object $``\lim"_n\bAlb_{(X',nD)}$ is indeed essentially constant, so the pro-object
$L_0\Alb^{\log}(X- \pi(|D|),\triv)$ is essentially constant.
Then the following question arises naturally:
\begin{qn}\label{qn:BS-LAlb}
	Let $X$ be a proper variety and $U\subset X$ be a smooth open subvariety. When is the pro-reciprocity sheaf $L_i\Alb^{log}(M(U,\triv))$ essentially constant?
\end{qn}
Notice that in general $L_i\Alb^{log}(M(U,\triv))$ is not essentially constant: let $X$ be a proper non-singular surface, and $U = X - Y$ for some closed subscheme $Y$ of $X$. Let $(X',D)$ be the blow-up of $X$ in $Y$. As observed in \cite[p. 407]{HartshorneCohDim}, if some irreducible component of $Y$ is a point, then $\dim(H^1(U, \cO_U))=\infty$. On the other hand, $\dim(H^1(X', \cO_{X'}(nD)))$ is finite for every $n$, so $L_1\Alb^{log}(M(U,\triv))$ is not constant by Theorem \ref{thm;compute-LAlbX-intro}. At the moment, we do not know if there is a nice family of pairs $U\subseteq X$ that answers Question \ref{qn:BS-LAlb} positively.

\section{Logarithmic 1-motivic complexes}

For any perfect field $k$ and any commutative ring $\Lambda$, recall the stable $\infty$-category of $1$-motivic complexes $\mathcal{DM}^{\rm eff}_{\leq 1}(k,\Lambda)$, i.e. the full stable $\infty$-subcategory of $\mathcal{DM}^{\rm eff}(k,\Lambda)$ generated by $M(X)$, with $\dim(X)\leq 1$. If $\Lambda=\Q$, by \cite[Theorem 2.4.1]{Ayoub-BV}, the composition\[
L\Alb_{\leq 1}\colon \mathcal{DM}^{\rm eff}_{\leq 1}(k,\Q)\hookrightarrow\mathcal{DM}^{\rm eff}(k,\Q)\xrightarrow{L\Alb} \cD(\HI_{\leq 1}(k,\Q))
\]
is an equivalence. To ease the notation, we will denote the functor $R\omega^*_{\log}$ from \eqref{eq;adjunctionomega-motivic} simply by:
\[ \omega^*: \mathcal{DM}^{\rm eff}(k,\Q) \to \logDM(k,\Q).\] 
We give the following definition:
\begin{defn}
	For any perfect field $k$ and any commutative ring $\Lambda$, we let $\logDMone(k,\Lambda)$ be the full stable $\infty$-subcategory of $\logDM(k,\Lambda)$ generated by $\omega^*\mathcal{DM}^{\rm eff}_{\leq 1}$ and $\Log_{\det}(U)[n]$ for a unipotent group scheme $U$. We will call it the stable $\infty$-category of \emph{log $1$-motives}.
\end{defn}

\begin{remark} By \cite[Theorem 7.6.7]{BPO}, if $k$ satisfies $(RS)$ the $\infty$-category $\omega^*\mathcal{DM}^{\rm eff}_{\leq 1}$ is equivalent to the $\infty$-subcategory of $\logDM(k,\Lambda)$ generated by $M(X)$ with $\ul{X}$ a proper smooth curve. Moreover, if $\ch(k)=0$, every unipotent group scheme splits as a direct sum of $\G_a$, hence the categroy $\logDMone(k,\Lambda)$ is generated by $\omega^*\mathcal{DM}^{\rm eff}_{\leq 1}$ and $\Log_{\det}\G_a[n]$. Moreover, if $\Lambda$ is a $\Q$-algebra, the functor $\omega^{\lDM}_{\leq 1}$ from Theorem \ref{thm:derivedAlbanese} factors through $\logDMone(k,\Q)$: indeed the category $\cD(\RSC_{\leq 1}(k,\Q))$ is generated by $h_0(\kC)[n]$ with $\kC=(\ol{C},C_{\infty})$ a proper modulus pair of dimension $1$ such that $\ol{C}-|C_{\infty}|$ is affine. By \cite{RulYama} and the proof of Proposition \ref{prop:AlbaneseChow}, we have that $h_0(\kC)\simeq \bAlb_{\kC}$, so the exact sequence \eqref{eq:can-map-Alb-Albsemiab} gives a fiber sequence in $\cD(\RSC_{\leq 1}(k,\Q))$:\[
	\G_a^{\oplus r}[n]\to h_0(\kC)[n]\to h_0(\kC_{red})[n] .
	\]
	Let $C\in \SmlSm(k)$ be the log scheme $(\ol{C},\partial C)$ with $|\partial C|=|C_{\infty}|$, then\[
	\omega^{\lDM}_{\leq 1}h_0(\kC_{red})[m] = \omega^*h_0^{\A^1}(\ol{C}-C_{\infty})[m] = \omega^*M(\ol{C}-C_{\infty})[m]
	\] 
	where the last equality is \cite[Thm. 3.4.2]{VoevTriangCat} since $\ol{C}-C_{\infty}$ is affine. By construction $\Log_{\det}\G_a[n]= \omega^{\lDM}_{\leq 1}\G_a[n]$. By repeating this argument backwards we conclude that the functor $\omega^{\lDM}_{\leq 1}$ is also essentially surjecrive on $\logDMone(k,\Lambda)$.
\end{remark}

From now on, we consider again $\ch(k)=0$ and $\Lambda=\Q$. We have the following generalization of \cite[Theorem 2.4.1]{Ayoub-BV}:

\begin{thm}\label{thm;fully-faithfyl} The composition $L\Alb^{\log} \circ \omega^{\lDM}_{\leq 1}$ is equivalent to the constant pro-object functor. In particular, the functor $\omega^{\lDM}_{\leq 1}$ is fully faithful and induces an equivalence\[
	\omega^{\lDM}_{\leq 1}\colon \cD(\RSC_{\leq 1}(k,\Q)) \xrightarrow{\sim} \logDMone(k,\Q) \colon \lim L\Alb^{\log}.
	\]
\end{thm}
\begin{proof}
	By construction, the proof follows from Proposition \ref{prop;comparison-usual-LAlb} and Lemma \ref{lem;compute-Li-Alb-Ga} below.
\end{proof}

\begin{prop}\label{prop;comparison-usual-LAlb}
	There is a commutative diagram of stable $\infty$-cateogries:\[
	\begin{tikzcd}
		&\cD(\HI_{\leq 1}(k,\Q))\ar[r,"j"]&\Pro\cD(\RSC_{\leq 1}(k,\Q))\\
		\mathcal{DM}^{\rm eff}_{\leq 1}(k,\Q)\ar[r,hook]\ar[ur,bend left=10,"L\Alb_{\leq 1}"]&\mathcal{DM}^{\rm eff}(k,\Q) \ar[r,"\omega^*"]
		&\logDM(k,\Q).\ar[u,swap,"L\Alb^{\log}"]
	\end{tikzcd}
	\]
	\begin{proof}
		The category $\mathcal{DM}^{\rm eff}_{\leq 1}(k,\Q)$ is generated by $M^{\A^1}(C)$ with $C$ affine curve, and for such $C$ we have ${\rm NS}^1(C_{\overline{k}})=0$. It is then enough to show that\[
		L\Alb^{\log}(\omega^*M^{\A^1}(C))\simeq jL\Alb_{\leq 1}(M^{\A^1}(C)).
		\]
		Recall that $\omega^*M^{\A^1}(C) = M(\ol{C},\partial C)$ with $\ol{C}$ any smooth compactification and $\partial C$ is the log structure associated to
the closed subscheme $C_\infty:=\ol{C}\setminus C$ with reduced structure. Since $C$ is affine, the right-hand side is equivalent to the constant pro-object $\Alb(C)[0]$, so we conclude by Theorem \ref{thm;compute-LAlbX} since there exists $n\gg 0$ such that $H^1(\ol{C},\cO_{\ol{C}}(nC_{\infty})) = 0$ and $\Alb^{\log}(\ol{C},\partial C)=\Alb(C)$.
\end{proof}
\end{prop}

\begin{lemma}\label{lem;compute-Li-Alb-Ga}
	For all $i>0$, $L_i\Alb^{\log}(\omega^{\lDM}_{\leq 1}\G_a)=0$. 
	
	\begin{proof}
		For all $F\in \RSC_{\et, \leq 1}(k,\Q)$ and all $i$, we have by adjunction that
		\begin{equation}\label{eq:adjunction-LAlb-Ga}
			\Map_{\Pro\cD(\RSC_{\et,\leq 1})}(L\Alb^{\log}(\omega^{\lDM}_{\leq 1}\G_a),F[i]) \simeq \Map_{\logDM(k,\Q)}(\omega^{\lDM}_{\leq 1}\G_a,\Log_{\det}(F)[i]).
		\end{equation}
Noting $\omega^{\lDM}_{\leq 1}\G_a\simeq L_{\bcube}\Log_{\det}(\G_a)[0]$
and $\Log_{\det}(\G_a) \cong \omega^\sharp_{\log}\G_a$ by \eqref{omegaCIGa}, 
we get
	\begin{equation}\label{eq:adjunction-LAlb-Ga-2}
\begin{aligned}
		\Map_{\logDM(k,\Q)}(\omega^{\lDM}_{\leq 1}\G_a,\Log_{\det}(F)[i])&\simeq\Map_{\cD(\Sh_{\et}^{\ltr}(k,\Q))}(\omega^\sharp_{\log}\G_a,\Log_{\det}(F)[i])\\
&\simeq\Map_{\cD(\Sh_{\et}^{\tr}(k,\Q))}(\G_a,\omega_\sharp^{\log}\Log_{\det}(F)[i])\\
		&\simeq  \Map_{\cD(\Sh_{\et}^{\tr}(k,\Q))}(\G_a,F[i]).
		\end{aligned}\end{equation}
where the first equivalence holds since {$\Log_{\det}(F)[i]$ is $\bcube$-local, the second follows from \eqref{eq:adjunctionomega-derived1-2}} and the last holds since
$\omega_\sharp^{\log}\Log_{\det}\simeq id$.
 		By Theorem \ref{lem:split}, we have that \[
		L\Alb^{\log}(\omega^{\lDM}_{\leq 1}\G_a)\simeq \bigoplus_j L_j\Alb^{\log}(\omega^{\lDM}_{\leq 1}\G_a)[j],
		\]
		hence noting $\pi_0$ commutes with products, we have\[
		\begin{tikzcd}
		\pi_0\Map(L\Alb^{\log}(\omega^{\lDM}_{\leq 1}\G_a),F[i]) \ar[r,"{(*)}","\cong"']\ar[d,"\cong"] &\pi_0\Map(\G_a,F[i])\ar[dd,"\cong"]\\
		\prod_{j\geq 0}\pi_0\Map(L_j\Alb^{\log}(\omega^{\lDM}_{\leq 1}\G_a),F[i-j])\ar[d,"\cong"]\\
		\bigoplus_{j=0}^i \Ext^{i-j}_{\pro\RSC_{\et, \leq 1}(k,\Q)}(L_j\Alb^{\log}(\omega^{\lDM}_{\leq 1}\G_a),F) &\Ext^i_{\Sh_{\et}^{\tr}(k,\Q)}(\G_a,F),
		\end{tikzcd}
	\]
where $(*)$ follows from \eqref{eq:adjunction-LAlb-Ga} and \eqref{eq:adjunction-LAlb-Ga-2}. So, $\Hom_{\pro\RSC_{\et,\leq 1}}(L_i\Alb(\omega^{\lDM}_{\leq 1}\G_a),F)$ is a direct summand of $\Ext^i_{\Sh_{\et}(k,\Q)}(\G_a,F)$, which for $i\geq 2$ is zero by Proposition \ref{prop:hom-dim}. For $i=1$, the above diagram gives 
		\begin{align*}
		{\Ext^1_{\Sh^{\tr}_{\et}(k,\Q)}(\G_a,F) } \cong &\Ext^1_{\pro\RSC_{\et,\leq 1}}(L_0\Alb^{\log}(\omega^{\lDM}_{\leq 1}\G_a),F)\oplus \\
		&\Hom_{\pro\RSC_{\et,\leq 1}}(L_1\Alb^{\log}(\omega^{\lDM}_{\leq 1}\G_a),F).
		\end{align*}
		On the other hand, $L_0\Alb^{\log}(\omega^{\lDM}_{\leq 1}\G_a)\cong \Alb^{\log}(\Log_{\det}(\G_a))\cong \G_a$ since $L\Alb^{\log}$ is the derived functor of $\Alb^{\log}$ and $\Log_{\det}$ is fully faithful, so Proposition \ref{prop;1-rec-closed-under-ext} and Corollary \ref {cor:AHPL} implies that \[\Hom_{\pro\RSC_{\et,\leq 1}}(L_1\Alb^{\log}(\omega^{\lDM}_{\leq 1}\G_a),F)=0,\] 
which concludes the proof.
	\end{proof}
\end{lemma}

\section{Laumon \texorpdfstring{$1$}{one}-motives and compact objects}\label{sec:Laumon}
In this Section, we combine the results of \cite{BertapelleJAlgebra} with some arguments of \cite{BVKahn}. As before, let $k$ be a field of characteristic zero.
\subsection{Review of Laumon $1$-motives} 
The following definition is adapted from \cite[1]{BertapelleJAlgebra}.

\begin{defn}\label{def:eff-1-mot} An \textit{effective Laumon $1$-motive} is a two-terms complex $M =[\Gamma \xrightarrow{u} G]$, where $\Gamma$ is a formal $k$-group and $G$ is a connected algebraic $k$-group, both seen as objects of $\Sh_{\et}(k)$. We say that $M$ is \emph{\'etale} if $\Gamma$ is a lattice\footnote{ Note that this definition is different from the one given in \cite[1.4]{BVBert}, where the authors require in addition that $U(G) =0$}.  An \textit{effective morphism} 
	\[\xymatrix{ M = [\Gamma \xrightarrow{u} G]\ar[r]^{(f,g)} &M'= [\Gamma'\xrightarrow{u'} G']}\] 
	is a map of complexes. 
	We denote the category of effective (resp. \'etale) Laumon $1$-motives by $\cM_{1}^{a, {\rm eff}}$ (resp. $\cM_{1,\et}^{a, {\rm eff}}$). An effective Laumon $1$-motive is an \emph{effective Deligne $1$-motive} if $G$ is semi-abelian. We write $\cM_1^{D,{\rm eff}}$ for the full subcategory of effective Deligne $1$-motives.
\end{defn}

\begin{defn}An effective morphism $(f,g)\colon M\to M'$ is \textit{strict} if $g$ has (smooth) connected kernel, and a \textit{quasi-isomorphism} if $g$ is an isogeny, $f$ is surjective and $\ker(f) =
\ker(g)$ is a finite $k$-group scheme.
\end{defn}
Note that if $(f,g)$ is strict and $g$ is an isogeny, then $g$ is an isomorphism of commutative algebraic groups. 
Write $\Sigma$ for the class of quasi-isomorphisms: it admits a calculus of right fractions (see \cite[C.2.4]{BVKahn} or \cite[Lemma 1.6]{BertapelleJKT}). 
We can now give the following Definition (see \cite[Definition 2]{BertapelleJAlgebra} and \cite[Definition 1.4.4]{BVBert}).
\begin{defn}\label{def:1-mot} The category of \textit{\'etale Laumon} (resp. \textit{Laumon}, resp \textit{Deligne}) \textit{$1$-motives} $\cM_{1, \et}^a$ (resp. $\cM_{1}^a$, resp. $\cM_{1}^{D}$) is the localization by $\Sigma$ of 
	the effective category.
\end{defn}
Recall \cite[Appendix B]{BVKahn} for the notion of $\cC\otimes \Q$ for an additive category $\cC$. The proof of the following proposition is identical to \cite[Corollary C.7.3]{BVKahn}:
\begin{prop}\label{prop;1-mot-abelian}
The categories $\cM_{1}^D\otimes \Q$, $\cM_{1}^{a}\otimes \Q$ and $\cM_{1,\et}^{a}\otimes \Q$ are abelian.
\end{prop}

\begin{defn}[{see \cite{BertapelleJAlgebra}}]\label{def:M1star} Let $\cM_{1, \et}^{a, \star}$ 
	be the full subcategory of $\cM_{1, \et}^{a}$ 
	whose objects are $1$-motives $M=[\Gamma \xrightarrow{u} G]$ with $\ker(u)=0$.
\end{defn}
\begin{lemma}\label{lem:gen-sub} The category $\cM_{1, \et}^{a, \star}\otimes \Q$ is a generating subcategory of $\cM_{1, \et}^{a}\otimes \Q$, and it is closed under kernels and extensions. Moreover, for every object $M\in \cM_{1, \et}^{a}\otimes \Q$ there exists a monomorphism $f:M'\to M''$ in $\cM_{1, \et}^{a, \star}\otimes \Q$ such that $M=\coker(f)$.
	\begin{proof} This is essentially \cite[Lemma 4]{BertapelleJAlgebra}. 
\end{proof}
\end{lemma}

\begin{remark} The reader might wonder if there are interesting examples of \'etale Laumon $1$-motives which are not Deligne $1$-motives. The prototype of such example is given by the $1$-motive $M^\natural= [\Gamma\to \mathbf{G}^\natural]$ which is the universal $\G_a$-extension of the Deligne $1$-motive $[\Gamma\to G]$. Starting from the motive $M^\natural$ it is possible to construct the universal \emph{sharp} extension $M^\sharp$ of $M$, as discussed in \cite{BVBert}. Note however that the cateogry $\cM_{1, \et}^a$ is not closed under $\sharp$-extensions: as remarked in \cite[3.1.5]{BVBert}, $[0\to \G_a]^{\sharp}=[\widehat{\G}_a\to \G_a^2]$, which is clearly not \'etale.
\end{remark}
\begin{remark}
	The category of Deligne $1$-motives has an interesting self-duality, induced by the classical Cartier duality for algebraic groups. This extends to Laumon $1$-motives, see \cite{BVBert}. Note that while the Cartier dual of a Deligne $1$-motive is again a Deligne $1$-motive, the dual of an \'etale Laumon $1$-motive is in general not \'etale.  For example, if $A$ is an Abelian variety (see as $1$-motive $[0\to A]$), its universal $\G_a$ extension is the \'etale Laumon $1$-motive $A^\natural = [0\to A^\natural]$, which is not a Deligne $1$-motive. Its Cartier dual $(A^\natural)^*$ is the $1$-motive $[\widehat{A}'\to A']$, where $A'$ is the dual Abelian variety of $A$ and $\widehat{A}'$ is the formal completion of $A'$ along the identity. Clearly, $(A^\natural)^*$ is a Laumon $1$-motive that is not \'etale.  
\end{remark}

\begin{remark}\label{rmk:ES-and-CI-finpres}
Consider the functor\[
\rho^{\rm eff}\colon \cM_{1, \et}^{a,{\rm eff}} \to \Sh_{\et}(k,\Q)\quad [L\xrightarrow{u} G] \mapsto \coker(u)\otimes_\Z \Q.
\]
If $[L_1\xrightarrow{u_1} G_1]\to [L_2\xrightarrow{u_2} G_2] \in \Sigma$, then by definition $\coker(u_1)\otimes \Q\cong \coker(u_2)\otimes \Q$. This together with \cite[Lemma B.1.2]{BVKahn} implies that $\rho^{\rm eff}$ induces:\[
\rho\colon \cM_{1, \et}^{a}\otimes \Q \to \Sh_{\et}(k,\Q).
\]
\end{remark}
\begin{lemma}\label{lem;rho-star}
The restriction of $\rho$ to $\cM_{1, \et}^{a, \star}\otimes \Q$ induces an equivalence (cf. Def. \ref{defn:fin-pres1mot}):\[
\rho^\star \colon  \cM_{1, \et}^{a, \star}\otimes \Q \xrightarrow{\cong} \RSC_{\et, \leq 1}^\star.
\]
\begin{proof}  By definition, for $[L\to G]\in \cM_{1, \et}^{a, \star}\otimes \Q$, $\rho([L\to G])\in \RSC_{\et,\leq 1}^\star$, and by Proposition \ref{prop;forget-transfer-fully-faithful} for every morphism $f$ we have that $\rho(f)$ is a map in $\RSC_{\et, \leq 1}^\star$, hence $\rho^\star$ is well defined. The presentation of Proposition \ref{prop:can-pres-1-mot-sheaf} gives then a quasi-inverse of $\rho^\star$.
\end{proof}
\end{lemma}

\begin{remark}\label{rmk;generators-laumon-1-mot}
For a category $\mathcal{C}$, we write $\Ind(\cC)$ for the Ind-category of $\cC$ as in e.g.~\cite{Kashiwara-Schapira}. By \cite[Prop. 6.3.4]{Kashiwara-Schapira} and Remark \ref{rmk;RSC-star-compact} the functor $\Ind(\RSC_{\et, \leq 1}^{\star})\to \RSC_{\et, \leq 1}$ induced by filtered colimits is fully faithful. It is also essentially surjective by Proposition \ref{prop:can-pres-1-mot-sheaf}, hence it is an equivalence. Combining this with Lemma \ref{lem:gen-sub} and \ref{lem;rho-star}, we have a functor
\begin{equation}\label{eq;Ma-generators-RSC-1}
T\colon \cM_{1,\et}^{a}\otimes \Q\to \Ind(\cM_{1,\et}^{a}\otimes\Q)\overset{(*)}{\simeq}\Ind(\cM_{1,\et}^{a,\star}\otimes\Q)\simeq \RSC_{\et,\leq 1}(k,\Q)
\end{equation}
where $(*)$ follows from the fact that $\cM_{1,\et}^{a,\star}\otimes\Q$ is a generating subcategory of $\cM_{1,\et}^{a}\otimes\Q$. Since $\cM_{1,\et}^{a}\otimes \Q$ is abelian by Lemma \ref{prop;1-mot-abelian}, it is idempotent-complete, hence following the steps of \cite[Exercise 6.1]{Kashiwara-Schapira} the functor \eqref{eq;Ma-generators-RSC-1} is fully faithful and it identifies $\cM_{1,\et}^{a}\otimes \Q$ with a set of compact generators of $\RSC_{\et,\leq 1}$. Moreover, by \cite[Proposition 8.6.11]{Kashiwara-Schapira}, the category $T(\cM_{1,\et}^{a}\otimes \Q)$ is closed under extensions in $\RSC_{\et,\leq 1}$.
\end{remark}
\subsection{The derived category of \'etale Laumon \texorpdfstring{$1$}{1}-motives}

By \cite[Proposition 8.6.11]{Kashiwara-Schapira} and Remark \ref{rmk;generators-laumon-1-mot}, the image of the functor $\rm T$ of \eqref{eq;Ma-generators-RSC-1} is a Serre subcategory of $\RSC_{\leq 1,\et}(k,\Q)$, hence we can consider the triangulated category $D^b_{\cM_{1,\et}^{a}}(\RSC_{\leq 1,\et}(k,\Q))$ of bounded complexes of $\RSC_{\leq 1}(k,\Q)$ such that $H_n(C)=T([L_n\to G_n])$ for $[L_n\to G_n]\in \cM_{1,\et}^{a}\otimes \Q$.
\begin{remark}\label{lem;bounded-der-cat}
The functor $T$ of \eqref{eq;Ma-generators-RSC-1} induces an equivalence of triangulated categories:\[
D^b(\cM_{1,\et}^{a}\otimes \Q)\simeq D^b_{\cM_{1,\et}^{a}}(\RSC_{\et,\leq 1}(k,\Q))
\]
where the latter is the triangulated derived category , since by Proposition \ref{prop:hom-dim}, every object of $\RSC_{\et,\leq 1}^\star$ is of projective dimension at most 1 in the sense of \cite{Kellercychom}, in particular the image of $T$ satisfies \cite[1.21 Lemma (c2)]{Kellercychom}, hence the equivalence comes from \cite[1.21 Lemma (c)]{Kellercychom}.
\end{remark}
\begin{defn}
Let $\cD^b(\cM_{1,\et}^{a}\otimes\Q)$ be the full $\infty$-subcategory of $\cD(\RSC_{\et,\leq 1})$ spanned by bounded complexes $C\in \cD^b(\RSC_{\et,\leq 1})$ such that $\pi_n C=T([L_n\to G_n])$ for $[L_n\to G_n]\in \cM_{1,\et}^{a}\otimes \Q$. 
\end{defn}
\begin{remark}
Notice that since the category $\cM_{1,\et}^{a}\otimes \Q$ does not have enough injective nor projective objects, we cannot use \cite[1.3]{HA} to construct $\cD^b(\cM_{1,\et}^{a}\otimes \Q)$ directly.
\end{remark}

\begin{lemma}\label{lem;compact-laumon}
There is an equivalence of $\infty$-categories:\[
\cD(\RSC_{\et,\leq 1}(k,\Q))^\omega \simeq \cD^b(\cM_{1,\et}^{a}\otimes \Q)
\]
where the left hand side denotes the subcategory of compact objects as in \cite[Notation 5.3.4.6]{HTT}.
\begin{proof}
Since the set of objects of $\RSC_{\et,\leq 1}(k,\Q)$ lying in the image of \eqref{eq;Ma-generators-RSC-1} is a set of compact generators, by \cite[Lemma 094B]{stacks-project} we have an equivalence\[
\cD(\RSC_{\et,\leq 1}(k,\Q))^\omega \simeq {\rm Idem}(\cD^b(\cM_{1,\et}^{a}\otimes \Q))
\]
where the right-hand side is the idempotent completion of $\cD^b(\cM_{1,\et}^{a}\otimes \Q)$, see \cite[Definition 5.1.4.1]{HTT}. On the other hand, the category $\cD^b(\cM_{1,\et}^{a}\otimes \Q)$ is idempotent-complete since the image of \eqref{eq;Ma-generators-RSC-1} is idempotent complete (it is an abelian subcategory), hence every object of $\cD(\RSC_{\et,\leq 1}(k,\Q))^\omega$ lies in $\cD^b(\cM_{1,\et}^{a}\otimes \Q)$. The other inclusion is clear.
\end{proof}
\end{lemma}

\begin{thm}\label{thm;Laumon-are-motivic}
Let $\mathbf{log}{\mathcal{DM}}^{\rm eff}_{\leq 1,{\rm gm}}(k,\Q):= \mathbf{log}{\mathcal{DM}}^{\rm eff}_{\leq 1}(k,\Q)^\omega$. The functor $\omega^{\lDM}_{\leq 1}$ preserves compact objects and it induces an equivalence\[
\cD^b(\cM_{1,\et}^{a}\otimes \Q)\xrightarrow{\sim} \mathbf{log}{\mathcal{DM}}^{\rm eff}_{\leq 1,{\rm gm}}(k,\Q).
\] 
\begin{proof} 
By Lemma \ref{lem;compact-laumon}, if $C\in \cD(\RSC_{\leq 1})$ is compact, then it is a bounded complex such that $\pi_i(C)=T(M_i)$ for $M_i\in \cM_{1,\et}^{a}\otimes \Q$. In particular, there exists $n$ such that $C=\tau_{\geq n}C$, so  we get a fiber-cofiber sequence in $\cD(\RSC_{\et,\leq 1})$:\[
T(M_n)[n]\to C \to \tau_{\geq n-1}C.
\]
If $C$ is compact, then $\tau_{\geq n-1}C$ is compact, 
so by induction on the length of the bounded complex it is enough to show that for $M\in \cM_{1,\et}^{a}\otimes \Q$, the object $\omega^{\lDM}_{\leq 1}(T(M)[n])$ is compact. As observed in Lemma \ref{lem:gen-sub}, there is an exact sequence in $\cM_{1,\et}^{a}\otimes \Q$:\[
0\to M'\to M''\to M\to 0
\]
with $M', M''\in \cM_{1,\et}^{a,\star}$, and since $T$ is exact we have
\[
\omega^{\lDM}_{\leq 1}(T(M)[n]) = {\rm cofib}(\omega^{\lDM}_{\leq 1}(T(M')[n])\to \omega^{\lDM}_{\leq 1}(T(M'')[n]).
\]
Thus, it is enough to show that $\omega^{\lDM}_{\leq 1}(T(M)[n])$ is compact for $M=[L\overset{u}{\hookrightarrow} G]$. In this case, we have that $T(M)=\coker(u)$, so we have a cofiber sequence\[
\omega^{\lDM}_{\leq 1}(L[n])\to \omega^{\lDM}_{\leq 1}(G[n])\to \omega^{\lDM}_{\leq 1}(T(M)[n]).
\] 
We conclude since $\omega^{\lDM}_{\leq 1}(L)[n]=\omega^*L[n]$ and $\omega^{\lDM}_{\leq 1}(G[n])=\omega^{\CI}_{\log}(G)[n]$ are compact. 
The equivalence then follows from Theorem \ref{thm;fully-faithfyl} and Lemma \ref{lem;compact-laumon}.
\end{proof}

\end{thm}

\appendix

\section{Pro-left derived functors}\label{sec:appendix}

In this appendix we generalize to pro-left adjoints the results discussed in \cite[2.1]{Ayoub-BV} and \cite[14.3]{Kashiwara-Schapira} for left adjoints. We use in an essential way the formalism of (stable) $\infty$-categories of \cite{HA} and \cite{HTT}.

\subsection{}\label{setting-BC} We consider the following commutative square:\begin{equation}
\label{eq:cd-BC}
\begin{tikzcd}
\cC\ar[r,"G"]\ar[d,"L_{\cC}"]&\cD\ar[d,swap,"L_{\cD}"]\\
\cC'\ar[r,"G'"]&\cD'
\end{tikzcd}
\end{equation}
where $\cC,\cC'$, $\cD$ and $\cD'$ are 
 $\infty$-categories
 , $L_{\cC}$ (resp. $L_{\cD}$) has a fully faithful right adjoint $i_{\cC}$ (resp. $i_{\cD}$). Let $\alpha:L_{\cD}G\to G'L_{\cC}$ be the equivalence that makes the diagram commute. As observed in \cite[Definition 4.7.4.13]{HA}, $\alpha$ induces a natural transformation (the \emph{Beck--Chevalley} map)
\[BC(\alpha):Gi_{\cC}\to i_{\cD}G'.\]

\begin{defn}
In the situation of \ref{setting-BC}, an object $X$ of $\cD$ is $BC(\alpha)$-admissible if and only if for any $Y\in \cC'$ the Beck-Chevalley map induces an equivalence:
\begin{equation}\label{eq;admiss}
\Map_{\cD}(X,i_{\cD}G'(Y))\simeq \Map_{\cD}(X,Gi_{\cC}(Y)).
\end{equation}

\end{defn}

Assume that $G$ and $G'$ have left adjoints $F$ and $F'$, then by adjunction $\alpha$ induces a map\[
\beta\colon F'L_{\cD}\to L_{\cC}F.
\]

\subsection{} Let $\cC$ be an (arbitrary) $\infty$-category. We consider as in \cite[Section 2]{Hoyois} the $\infty$-category of pro-objects $\Pro\cC$ 
together with the ``constant pro-object" embedding $c\colon \cC\to \Pro(\cC)$ such that $\Map_{\Pro\cC}(\_,c(Y))$ commmutes with cofiltered limits. 

Every element in $\Pro\cC$ is corepresented by a diagram $I\to \cC$ for $I$ (the nerve of) a small cofiltered poset. 
We will often denote an object of $\Pro\cC$ as $``\lim_{i\in I}"X_i$ for a diagram $I\to \cC$. By construction we have that\[
\Map_{\Pro\cC}(``\lim_{i\in I_1}"X_i,``\lim_{j\in I_2}"Y_j) \simeq \lim_{j}\colim_{i}\Map_{\cC}(X_i,Y_j)
\]
where the limits and colimits are computed in the $\infty$-category of spaces $\cS$. 

\begin{remark}\label{rmk;pro-BC}
The functors $L_{\cC}$ and $i_{\cC}$ extend levelwise to an adjunction $(\Pro L_{\cC},\Pro i_{\cC})$ on $\Pro\cC$ and $\Pro\cC'$ with the same properties. The verification is immediate. 
In particular, if $\cD$ and $\cD'$ have all limits, $G$ and $G'$ give the following commutative diagram:\[
\begin{tikzcd}
\Pro\cC\ar[r,"\Pro G"]\ar[d,"\Pro L_{\cC}"]&\cD\ar[d,swap,"L_{\cD}"]\\
\Pro\cC'\ar[r,"\Pro G'"]&\cD'
\end{tikzcd}
\]
which satisfies the hypotheses of Situation \ref{setting-BC} with equivalence \[\alpha^{pro}:L_{\cD}\Pro G\to \Pro G' \Pro L_{\cC}.\] In particular, since $i_{\cD}$ commutes with all limits being a right adjoint, it is immediate that $X\in \cD$ is $BC(\alpha)$-admissible if and only if it is $BC(\alpha^{pro})$-admissible. 

\end{remark}
\begin{remark} If $\cC$ is an accessible stable $\infty$-category equipped with a $t$-structure $(\cC_{\leq 0},\cC_{\geq 0})$ with heart $\cC^{\heartsuit}$, the $\infty$-category $\Pro \cC$ is also stable (see e.g. \cite[Lemma 2.5]{KST}) and it comes equipped with a $t$-structure such that $(\Pro \cC)_{\leq 0}$ (resp. $(\Pro \cC)_{\geq 0}$) is the full subcategory of objects which are formal limits of objects in $\cC_{\leq 0}$ (resp $\cC_{\geq 0}$). 
\end{remark}

\subsection{}
Let $\A$ be a Grothendieck abelian category. Write $\Ch(\A)$ for the model category of chain complexes with the injective model structure. Let $W$ be the class of quasi isomorphisms. We consider the $\infty$-categories (see \cite[1.3.5]{HA}) $\Ch_{\dg}(\A)=N_{\dg}(\Ch(\A))$ and $\cD(\A)=N_{\dg}(\Ch(\A))[W^{-1}]$. An exact functor $G\colon\A\to \B$ between Grothendieck abelian categories induces a $\dg$-functor $\Ch(G):\Ch(\A)\to \Ch(\B)$ which preserves $W$, so by taking the $\dg$-nerve it induces a functor $\Ch_{\dg}(G)\colon \Ch_{\dg}(\A)\to \Ch_{\dg}(\B)$ such that $\Ch_{\dg}(G)(C)=\Ch(G)(C)$, and by e.g. \cite[Proposition 4.3]{Hinich}, it induces $\cD(G)\colon \cD(\A)\to \cD(\B)$ on the localizations. Note that both functors are clearly stable (i.e. they commute with shifts). By construction, we have a commutative square of $\infty$-categories:
\begin{equation}\label{eq;diag-derived}
\begin{tikzcd}
	\Ch_{\dg}(\A)\ar[r,"\Ch_{\dg}(G)"]\ar[d,"L_{\A}"]&\Ch_{\dg}(\B)\ar[d,swap,"L_{\B}"]\\
	\cD(\A)\ar[r,"\cD(G)"]&\cD(\B)
\end{tikzcd}
\end{equation}
where $L_{\A}$ and $L_{\B}$ have fully faithful right adjoints $i_{\A}$ and $i_{\B}$. We will fix $\alpha$ that makes \eqref{eq;diag-derived} commute and just say that an object is $BC$-admissible.
\begin{remark}\label{rem:BC-abeliancase}
    Note  that, since $G$ is exact, we can identify the Beck-Chevalley transformation $\Ch_{\dg}(G) i_{\A} \to i_{\B} \cD(G)$ as follows. For any object $I\in \cD(\A)$ (i.e. a fibrant complex in $\Ch(\A)$ for the injective model structure), the object $\cD(G)(I)$ is 
    a fibrant replacement of $\Ch_{\dg}(G)(i_{\cA}(I))$. The map $\Ch_{\dg}(G)(i_{\A}I) \to i_{\B}\cD(G)(I)$ in $\Ch(\B)$ is thus given by the functorial fibrant replacement. In particular, if $\Ch(G)$ is a right Quillen functor, 
    the map $\Ch_{\dg}(G)(i_{\A}I) \to i_{\B}\cD(G)(I)$ is 
    an equivalence in $\Ch_{\dg}(\B)$. On the other hand, the functors considered here are not necessarily right Quillen.
\end{remark}
\begin{remark}
	If the functor $G$ has a left adjoint $F$, then $X$ is $BC$-admissible if and only if it is $F$-admissible in the sense of \cite[Definition 2.1.5]{Ayoub-BV}.
\end{remark}
\subsection{Pro-left derived functors}\label{ex-derived-cats}
Fix $\A$, $\B$ and $G$ as above and we assume that $G$ and $i_{\B}$ commute with filtered colimits. Since $G$ is exact, it preserves finite limits, so it has a pro-left adjoint $F\colon \B\to \pro\A$. The functor $\cD(G)$ is then an accessible functor between presentable $\infty$-categories that preserves finite limits, hence it has a pro-left adjoint $LF\colon \cD(\B)\to \Pro\cD(\A)$ (see e.g. \cite[Remark 2.2]{Hoyois}). 
	
For any chain complex $C$, let $\sigma_{\leq n}C$ and $\sigma_{\geq n}C$ denote the stupid truncations (see \cite[Tag 0118]{stacks-project}) \footnote{Notice that we chose the convention for \emph{chain} complexes, which is different from \cite[Lemma 2.1.10]{Ayoub-BV}: there the convention is for \emph{cochain} complexes}. We have an equivalence in $\Ch_{\dg}$:
\begin{equation}\label{eq:colim-lim-trunc}
	C\cong \colim_n(\lim_m\sigma_{\geq -m}\sigma_{\leq n} C).
\end{equation}

\begin{defn}\label{def;strictly-bounded}
	We say that a chain complex $C$ is \emph{strictly bounded} if there exists $m,n$ such that $C=\sigma_{\leq n}C=\sigma_{\geq -m}C$.
\end{defn} 

\begin{remark}\label{rmk:strictly-bounded-F-heart}
	Notice that if $C\in \Ch(\B)$, the object $\Ch(F)(C)$ a priori lives in $\Ch_{\dg}(\pro\A)$, which contains strictly $\Pro\Ch_{\dg}(\A)$. If $C$ is a strictly bounded complex, let $m,n$ such that $C=\sigma_{\leq n}C = \sigma_{\geq -m}C$, then for $r\in [-m,n]$ let $F^{\heartsuit}(C_r)=``\lim"_{i\in I_r} (X_r)_{i}$, for small cofiltered posets $I_r$ . Then one can find a cofinal set $I\subseteq I_r$ for all $r$ such that\[
	\Ch(F)(C) = ``\lim_{i\in I}"(\ldots\to (X_r)_{i}\to (X_{r-1})_i\to \ldots). 
	\]
	In particular, $\Ch(F)(C)\in \Pro(\Ch_{\dg}(\A))$.
\end{remark}	
	
\begin{prop}\label{prop;chain-homotopy-pro-left-adj}
For all $C \in \Ch_{\dg}(\B)$ strictly bounded, there is an equivalence in $\cS$: \[
\Map_{\Pro \Ch_{\dg}(\A)}(\Ch(F)(C)[0],Y)\simeq \Map_{\Ch_{\dg}(\B)}(C[0],\Ch_{\dg}(G)(Y)).
\]
\begin{proof}
Let $C_{n}=0$ for $n\notin[-r,s]$. The cofiber of the map $\sigma_{\leq s-1}C \to C$ is equivalent to $C_{s}[s]$, hence by induction on $r+s$ we are reduced to the case where $C=C_s[s]$ with $C_s\in \mathcal{B}$, and clearly $\Ch(F)(C_s[s])=F(C_s)[s]$.
Since $\Pro \Ch_{\dg}(\A)$ is pointed, by \cite[Remark 1.1.2.8]{HA} it is enough to show that for all $m$, \[\pi_0\Map_{\Pro \Ch_{\dg}(\A)}(F(X)[m],Y) \simeq \pi_0\Map_{\Ch_{\dg}(\B)}(X[m],\Ch_{\dg}(G)(Y)).\]
Let $F(X)=``\lim"T_i$, then for all $m$ we have an isomorphism of abelian groups:\[
\Hom_{\Ch(\B)}(X[m],\Ch(G)(Y)) = \colim \Hom_{\Ch(\A)}(T_i[m],Y),
\]
and since $\pi_0$ commutes with filtered colimits in $\cS$ we have by \cite[Remark 1.3.1.5, Remark 1.3.1.11 and Definition 1.3.2.1]{HA}:
\begin{align*}
\pi_0\Map_{\Pro\Ch_{\dg}(\A)}&(F(X)[n],Y)  \cong \colim \pi_0\Map_{\Ch_{\dg}(\A)}(T_i[n],Y)\\
&\cong \coker(\colim \Hom_{\Ch(\A)}(T_i[n+1],Y)\to
\colim\Hom_{\Ch(\A)}(T_i[n],Y)).\\
&\cong \coker(\Hom_{\Ch(\B)}(X[n+1],\Ch(G)(Y))\to\Hom_{\Ch(\A)}(X[n],\Ch(G)(Y)))\\
&\cong \pi_0\Map_{\Ch_{\dg}(\B)}(X[n],\Ch_{\dg}(G)(Y)).
\end{align*}
\end{proof}
\end{prop}

\begin{remark}\label{rmk;BC-adm-is-LF-trivial}
A priori, there is no relation between $F$ and $LF$, but if $X$ is strictly bounded and $BC$-admissible, then Proposition \ref{prop;chain-homotopy-pro-left-adj} implies that\[
LF(L_{\B}(X))\simeq L_{\Pro\A}\Ch(F)(X)[0].
\]
In particular for $X\in \B$ such that $X[0]$ is $BC$-admissible, \[
\pi_n(LF(L_{\B}(X[0])))=\begin{cases}
F(X)&\textrm{if }n=0\\
0&\textrm{otherwise}.
\end{cases}\]
\end{remark}
Remark \ref{rmk;BC-adm-is-LF-trivial} motivates the following definition:
\begin{defn}\label{defn;left-derived}
	In the situation of  \ref{ex-derived-cats}, $LF$ is said to be a \emph{pro-left derived functor} of $F$ if for every $X\in \B$\[
	\pi_0LF(X[0])\cong F(X).
	\] 
\end{defn}

\subsection{\texorpdfstring{$BC$}{BC}-admissible resolution} We will fix the setting of \ref{ex-derived-cats}. By abuse of notation, we will say that $P\in \B$ is $BC$-admissible if $P[0]\in \Ch_{\dg}(\B)$ is.

\begin{prop}\label{prop;bounded-admissible}
Let $P_\bullet\in \Ch_{\dg}(\B)$ be a strictly bounded complex (see Definition \ref{def;strictly-bounded}) such that $P_n$ is $BC$-admissible for all $n$. Then $P_{\bullet}$ is $BC$-admissible.
\begin{proof}
Up to shift, we can suppose that $P_{\bullet} = \sigma_{\geq 0}P_{\bullet}= \sigma_{\leq n}P_{\bullet}$ for some $n\geq 0$: we proceed by induction on $n$. 
If $n=0$, then $P_{\bullet}=P_{0}[0]$, and it is $BC$-admissible by assumption. 
Let $n>0$ and consider the fiber-cofiber sequence in $\Ch_{\dg}(\B)$:\[
\sigma_{\leq n-1}P_{\bullet}\to P_{\bullet} \to P_{n}[n].
\]
For all $I\in \cD(\A)$, we conclude by the following diagram in $\mathcal{S}$ where the left and right vertical maps are equivalences by induction:\[
\begin{small}
\begin{tikzcd}
\Map(P_{n}[0],\Ch_{\dg}(G)(i_{\A}I[-n]))\ar[r]\ar[d,"\simeq"]&\Map(P_{\bullet},\Ch_{\dg}(G)(i_{\A}I))\ar[r]\ar[d]&\Map(\sigma_{\leq n-1}P_{\bullet},\Ch_{\dg}(G)(i_{\A}I))\ar[d,"\simeq"]\\
\Map(P_{n}[0],i_{\B}\cD(G)(I[-n]))\ar[r]&\Map(P_{\bullet},i_{\B}\cD(G)(I))\ar[r]&\Map(\sigma_{\leq n-1}P_{\bullet},i_{\B}\cD(G)(I)).
\end{tikzcd}
\end{small}
\]
\end{proof}
\end{prop}

Recall that $\B$ is said to be generated by a set of objects $E$ if and only if $E$ is closed under direct sums and for every $X\in \B$ there exists a surjective map
\begin{equation}\label{eq;generating}
P_0\to X\to 0
\end{equation}
with $P_0\in E$.
Suppose that $\B$ is generated by a set of objects $E$ which are $BC$-admissible. Then let $K$ be the kernel of \eqref{eq;generating}, so there exists $P_1\in E$ together with a surjective map $P_1\twoheadrightarrow K$, hence we have an exact sequence:
\begin{equation}\label{eq;resolution}
P_1\to P_0\to X\to 0.
\end{equation}
By iterating \eqref{eq;resolution} one can construct a resolution $P_{\bullet}\to X[0]$ where $P_n\in E$ and $P_n=0$ for $n<0$. We will call this a \emph{connective $BC$-admissible resolution}.

\begin{lemma}\label{lem;compute-derived-with-resolution}
Suppose that $\B$ is generated by a set of objects which are $BC$-admissible. For any $X\in \B$ and any connective $BC$-admissible resolution $P_{\bullet}\to X[0]$, we have that\[
LF(L_{\B}(X[0])\simeq \colim_{n}L_{\Pro\A}(\Ch(F)(\sigma_{\leq n}P_{\bullet})).
\]
\begin{proof}
Since $P_{\bullet}\to X[0]$ is a resolution, we have $L_{\B}(P_\bullet)\simeq L_{\B}(X[0])$. Moreover, since $P_\bullet$ is connective, we have that $P_{\bullet}=\colim \sigma_{\leq n}P_{\bullet}$, and $\sigma_{\leq n}P_{\bullet}$ are $BC$-admissible by Proposition \ref{prop;bounded-admissible}. Since $LF$ and $L_{\B}$ commute with all colimits, 
by Remark \ref{rmk;BC-adm-is-LF-trivial} we conclude that\[
LF(L_{\B}X[0]) \simeq \colim LF(L_{\B}\sigma_{\leq n}P_{\bullet}) \simeq \colim_{n}L_{\Pro\A}(\Ch(F)(\sigma_{\leq n}P_{\bullet})),
\]
\end{proof}
\end{lemma}

We can now prove the main theorem of this appendix:

\begin{thm}\label{thm:gen-thm-existenceLF}
In the situation of \ref{ex-derived-cats}, suppose that $\B$ is generated by a set objects which are $BC$-admissible. Then the functor $LF$ is a pro-left derived functor of $F$.
\end{thm}
\begin{proof}
Let $X\in \B$ and $P_\bullet \to X[0]$ a $BC$-admissible resolution, in particular $P_\bullet\in \cD(\B)_{\geq 0}$. Since $L_{\B}X[0]\in \cD(\B)_{\geq 0}$ and $LF$ is right $t$-exact, $LF(L_{\B}X[0])\in \Pro\cD(\A)_{\geq 0}$, hence\[
\pi_0LF(L_{\B}X[0])[0]\simeq \tau_{\leq 0}LF(L_{\B}X[0]).
\]
Since $\tau_{\leq 0}$ is a left adjoint, it commutes with colimits, so by Lemma \ref{lem;compute-derived-with-resolution} we have\[
\tau_{\leq 0}LF(X[0]) \simeq \tau_{\leq 0} \colim_{n} L_{\Pro\A}\Ch(F)(\sigma_{\leq n}P_{\bullet})\simeq \colim_{n}\tau_{\leq 0} L_{\Pro\A}\Ch(F)(\sigma_{\leq n}P_{\bullet})
\]
where the last colimit is computed in $\Pro\cD(\A)_{\leq 0}$. On the other hand, by definition of the $t$-structure on $\Pro\cD(\A)$ we have\[
\tau_{\leq 0} L_{\Pro\A}\Ch(F)(\sigma_{\leq n}P_{\bullet}) \simeq L_{\Pro\A}\tau_{\leq 0}\Ch(F)(\sigma_{\leq n}P_{\bullet}).
\]
For $n\geq 1$, we have that $\tau_{\leq 0}\Ch(F)(\sigma_{\leq n}P_{\bullet}) = \coker(F(P_1)\to F(P_0))[0]$, and since $F$ is a left adjoint, it preserves cokernels, so\[
\coker(F(P_1)\to F(P_0)) = F(\coker(P_1\to P_0)) = F(X).
\]
We conclude that in $\Pro\cD(\A)_{\leq 0}$ we have\[
\colim_{n}\tau_{\leq 0} L_{\Pro\A}\Ch(F)(\sigma_{\leq n}P_{\bullet})\simeq \colim_{n}L_{\Pro\A}F(X)[0]\simeq L_{\Pro\A}F(X)[0]
\]
since $L_{\Pro\A}F(X)[0]\in \Pro\cD(\A)^{\heartsuit}$ we conclude that\[
\pi_0LF(X[0])\simeq\pi_0(\colim_{n}\tau_{\leq 0} L_{\Pro\A}F(\sigma_{\leq n}P_{\bullet}))\simeq F(X).
\]
\end{proof}

\subsection{} We end this appendix with a criterion of $BC$-admissibility. 

\begin{lemma}[{See \cite[Lemma 2.1.10]{Ayoub-BV}}]\label{lem:equivalent-condition-admissible} In the situation of \ref{ex-derived-cats}, let $P\in \B$ such that $P[0]$ is compact in $\Ch_{\dg}(\B)$. Then $P$ is $BC$-admissible if and only if for any injective object $I_0\in \A$, $\Ext^i_{\B}(P, G(I_0))=0$ for $i\neq 0$.
\begin{proof} Suppose first that $P$ is $BC$-admissible. If $I_0\in \A$ is injective, then $I_0[0]$ is fibrant, so $I_0[0]=i_{\A}L_{\A}(I_0[0])$. Let $I=L_{\A}(I_0[0])$, then: 
\begin{align*}
\Ext^n_{\B}(P, G(I_0)) 
&=\pi_0\Map_{\Ch_{\dg}(\B)}(P[0],i_{\cB}\cD(G)(I[-n])).
	\end{align*}
Since $P[0]$ is $BC$-admissible, we have: 
\begin{align*}
\pi_0\Map_{\Ch_{\dg}(\B)}(P[0],i_{\B}\cD(G)(I[-n])) & = \pi_0\Map_{\Ch_{\dg}(\B)}(P[0],\Ch_{\dg}(G)(i_{\A}I[-n])))  \\
& =\pi_0\Map_{\Ch_{\dg}(\B)}(P[0],G(I_0)[-n])).
\end{align*}
The last term is zero for $n\neq 0$, hence $\Ext^n_{\cB}(P, G(I_0))=0$ if $n\neq 0$.

Let us now show the converse implication. 
We need to show that for every $Y\in \cD(\A)$ the Beck-Chevalley map induces an equivalence:\[
\Map_{\Ch_{\dg}(\mathcal{B})}(P[0],i_{\B}\cD(G)(Y)) \simeq \Map_{\Ch_{\dg}(\mathcal{B})}(P[0],\Ch_{\dg}(G)(i_{\A}Y)).
\]
Let $I:=i_{\A}(Y)$, so $\sigma_{\geq -n}\sigma_{\leq m}I$ is a strictly bounded complex of injectives; since $P[0]$ is compact and the colimit in \eqref{eq:colim-lim-trunc} is filtered, Lemma \ref{lem;ext-admiss-reduct-bounded} below implies that:
\begin{equation}\label{eq;claim-almost}
	\begin{aligned}
		\Map_{\Ch_{\dg}(\B)}(P[0],\Ch_{\dg}(G)(I)) &\simeq \colim_n(\lim_m\Map_{\Ch_{\dg}(\B)}(P[0],\sigma_{\geq -n}\sigma_{\leq m}\Ch_{\dg}(G)(I)))\\ &\simeq
		\colim_n(\lim_m\Map_{\Ch_{\dg}(\B)}(P[0],\Ch_{\dg}(G)(\sigma_{\geq -n}\sigma_{\leq m}I)))\\
		&\simeq \colim_n(\lim_m \Map_{\Ch_{\dg}(\B)}(P[0],i_{\B}\cD(G)(L_{\A}\sigma_{\geq -n}\sigma_{\leq m}I))\\
		&\simeq \Map_{\Ch_{\dg}(\B)}(P[0],\colim_n\lim_m i_{\B}\cD(G)(L_{\A}\sigma_{\geq -n}\sigma_{\leq m}I)).
	\end{aligned}
\end{equation}

Next, observe that, for any $n$ and $m$, the map $\Ch_{\dg}(G)(\sigma_{\geq -n}\sigma_{\leq m}I)\to i_{\B}\cD(G)(L_{\A}\sigma_{\geq -n}\sigma_{\leq m}I)$ is a fibrant replacement (see Remark \ref{rem:BC-abeliancase}) of bounded complexes, which is given by the total complex of injective resolutions of each $G(I_r)\to J_{r}^{\bullet}$. So, we have that
\begin{equation}\label{eq;get-rid-of-limits}
\lim_mi_{\B}\cD(G)(L_{\A}\sigma_{\geq -n}\sigma_{\leq m}I) \simeq \lim_{m} \Tot_{r\in [-n,m]}(J_{r}^{\bullet}) = \Tot_{r\geq -n} (J_{r}^{\bullet}) \simeq i_{\B}\cD(G)(L_{\A}\sigma_{\leq -n}I).
\end{equation}
Since $i_{\B}$ commutes with filtered colimits by assumption, \eqref{eq;claim-almost} and \eqref{eq;get-rid-of-limits} imply:\begin{equation}\label{eq;claim-only-colimit}
	\Map_{\Ch_{\dg}(\B)}(P[0],\Ch_{\dg}(G)(i_{\A}Y))\simeq \Map_{\Ch_{\dg}(\B)}(P[0],i_{\B}\colim_n \cD(G)(L_{\A}\sigma_{\leq -n}I)).
\end{equation}
For every $q\in \Z$, we have that for $n\gg q$: \[
\pi_q \cD(G)(Y)\cong G(\pi_qY)\cong G(\pi_q\sigma_{\leq -n} I) \cong\pi_q \cD(G)(L_{\A}\sigma_{\leq -n} I),\]
so since homotopy groups commute with filtered colimits:
\begin{equation}\label{eq;get-rid-of-colimit}
\pi_q \colim_n \cD(G)(L_{\A}\sigma_{\leq -n} I) \cong\colim_n \pi_q \cD(G)(L_{\A}\sigma_{\leq -n} I) \cong \pi_q \cD(G)(Y)
\end{equation}
The proof follows then from \eqref{eq;claim-only-colimit} and \eqref{eq;get-rid-of-colimit}.
\end{proof}
\end{lemma}

\begin{lemma}\label{lem;ext-admiss-reduct-bounded}
	Let $P\in \B$ such that for any injective object $I_0\in \mathcal{A}$, $\Ext^i_{\mathcal{B}}(P, G(I_0))=0$ for $i\neq 0$. Then for any strictly bounded complex $I^b\in \Ch(\A)$ 
	of injective objects of $\A$:\[
	\Map_{\Ch_{\dg}(\B)}(P[0],\Ch_{\dg}(G)(I^b))\simeq \Map_{\Ch_{\dg}(\B)}(P[0],i_{\B}\cD(G)(L_{\A}I^b)).
	\] 
	\begin{proof} Let $I^b_{n}=0$ for $n\notin[-r,s]$. The cofiber of $\sigma_{\leq s-1}I^b \to I^b$ is equivalent to $I^b_{s}[s]$: by induction on $r+s$ we reduce to $I^b=I_s[s]$ with $I_s$ an injective object of $\mathcal{B}$. We conclude:
		\[
		\pi_n\Map_{\Ch_{\dg}(B)}(P[0],i_{\B}\cD(G)(L_{\A}I_s[s])) 
		\cong\Ext^{s-n}_{\mathcal{B}}(P, G(I_s))=\begin{cases}0&\textrm{if }n\neq s\\
			\Hom_{\mathcal{B}}(P, G(I_s))&\textrm{if }n= s.\end{cases}
		\]\end{proof}
\end{lemma}

\bibliographystyle{amsalpha}
\bibliography{bib1mot}

\end{document}

%% file: include.tex
\usepackage{amscd,verbatim}
\usepackage{amssymb}

\usepackage{bm}

\usepackage{amssymb, amsmath}
\usepackage[all]{xy}
\usepackage[inline]{enumitem}
\usepackage{hyphenat}
\usepackage[colorlinks,linkcolor=blue,citecolor=blue,urlcolor=red]{hyperref}
\usepackage[dvipsnames]{xcolor}
\usepackage{mathtools}
\usepackage{tikz-cd}
\usepackage[labelformat=empty]{caption}
 \usepackage[ left=3cm, right=3cm, top = 2.8cm, bottom = 2.8cm ]{geometry}

\makeatletter
\def\@tocline#1#2#3#4#5#6#7{\relax
  \ifnum #1>\c@tocdepth 
  \else
    \par \addpenalty\@secpenalty\addvspace{#2}%
    \begingroup \hyphenpenalty\@M
    \@ifempty{#4}{%
      \@tempdima\csname r@tocindent\number#1\endcsname\relax
    }{%
      \@tempdima#4\relax
    }%
    \parindent\z@ \leftskip#3\relax \advance\leftskip\@tempdima\relax
    \rightskip\@pnumwidth plus4em \parfillskip-\@pnumwidth
    #5\leavevmode\hskip-\@tempdima
      \ifcase #1
      \or\or \hskip 2em \or \hskip 2em \else \hskip 3em \fi%
      #6\nobreak\relax
    \dotfill\hbox to\@pnumwidth{\@tocpagenum{#7}}\par
    \nobreak
    \endgroup
  \fi}
\makeatother

\newcommand{\A}{\mathbf{A}}
\newcommand{\B}{\mathbf{B}}

\newcommand{\G}{\mathbf{G}}

\renewcommand{\P}{\mathbf{P}}
\newcommand{\Q}{\mathbb{Q}}

\newcommand{\Z}{\mathbb{Z}}

\newcommand{\sC}{\mathcal{C}}

\newcommand{\sG}{\mathcal{G}}

\newcommand{\sM}{{\mathscr M}}
\newcommand{\sO}{\mathcal{O}}
\newcommand{\sP}{\mathcal{P}}

\newcommand{\sY}{\mathcal{Y}}

\newcommand{\Xb}{{\overline{X}}}

\newcommand{\Cor}{\operatorname{\mathbf{Cor}}}

\newcommand{\HI}{\operatorname{\mathbf{HI}}}

\newcommand{\Ext}{\operatorname{Ext}}
\newcommand{\ul}[1]{{\underline{#1}}}

\newcommand{\Cpx}{{\operatorname{\mathbf{Cpx}}}}

\newcommand{\PST}{{\operatorname{\mathbf{PST}}}}

\newcommand{\NST}{\operatorname{\mathbf{NST}}}

\newcommand{\DM}{\operatorname{\mathbf{DM}}}

\newcommand{\Map}{\operatorname{Map}}
\newcommand{\Hom}{\operatorname{Hom}}
\newcommand{\uHom}{\operatorname{\underline{Hom}}}
\newcommand{\uExt}{\operatorname{\underline{Ext}}}

\newcommand{\Ker}{\operatorname{Ker}}

\newcommand{\Coker}{\operatorname{Coker}}

\newcommand{\Pic}{\operatorname{Pic}}

\newcommand{\Spec}{\operatorname{Spec}}

\newcommand{\Comp}{\operatorname{Comp}}
\newcommand{\Sm}{\operatorname{\mathbf{Sm}}}

\newcommand{\Sch}{\operatorname{\mathbf{Sch}}}
\newcommand{\Shv}{\operatorname{\mathbf{Shv}}}

\newcommand{\pro}[1]{\text{\rm pro}_{#1}\text{\rm--}}
\newcommand{\tr}{{\operatorname{tr}}}

\newcommand{\eff}{{\operatorname{eff}}}

\newcommand{\fin}{{\operatorname{fin}}}

\newcommand{\red}{{\operatorname{red}}}

\newcommand{\Zar}{{\operatorname{Zar}}}
\newcommand{\Nis}{{\operatorname{Nis}}}
\newcommand{\et}{{\operatorname{\acute{e}t}}}

\newcommand{\Jac}{{\operatorname{Jac}}}
\newcommand{\id}{{\operatorname{Id}}}

\newcommand{\codim}{{\operatorname{codim}}}
\newcommand{\ch}{{\operatorname{ch}}}

\newcommand{\Tot}{{\operatorname{Tot}}}
\newcommand{\CH}{{\operatorname{CH}}}

\renewcommand{\lim}{\operatornamewithlimits{\varprojlim}}
\newcommand{\colim}{\operatornamewithlimits{\varinjlim}}

\newcommand{\hocolim}{\operatorname{hocolim}}
\newcommand{\ol}{\overline}

\renewcommand{\phi}{\varphi}
\renewcommand{\epsilon}{\varepsilon}
\renewcommand{\div}{\operatorname{div}}

\newcommand{\MNS}{\operatorname{\mathbf{MNS}}}
\newcommand{\MNST}{\operatorname{\mathbf{MNST}}}

\newcommand{\MCor}{\operatorname{\mathbf{MCor}}}

\newcommand{\MSm}{\operatorname{\mathbf{MSm}}}

\newcommand{\MPST}{\operatorname{\mathbf{MPST}}}
\newcommand{\CI}{\operatorname{\mathbf{CI}}}

\newcommand{\Sq}{{\operatorname{\mathbf{Sq}}}}
\newcommand{\MSmsq}{{(\MSm)^{\Sq}}}

\newcommand{\bcube}{{\ol{\square}}}
\newcommand{\cube}{\square}

\newcommand{\M}{\mathbf{M}}

\newcommand{\ulMSm}{\operatorname{\mathbf{\underline{M}Sm}}}
\newcommand{\ulMNS}{\operatorname{\mathbf{\underline{M}NS}}}
\newcommand{\ulMPS}{\operatorname{\mathbf{\underline{M}PS}}}
\newcommand{\ulMPST}{\operatorname{\mathbf{\underline{M}PST}}}
\newcommand{\ulMNST}{\operatorname{\mathbf{\underline{M}NST}}}
\newcommand{\ulMCor}{\operatorname{\mathbf{\underline{M}Cor}}}

\newcommand{\uPic}{\underline{\Pic}}

\newcounter{spec}
{\end{list}}%

\newtheorem{lemma}{Lemma}[section]
\newtheorem{thm}[lemma]{Theorem}

\newtheorem{prop}[lemma]{Proposition}

\newtheorem{cor}[lemma]{Corollary}
\newtheorem{corollary}[lemma]{Corollary}

\newtheorem{claim}[lemma]{Claim}
\theoremstyle{definition}
\newtheorem{defn}[lemma]{Definition}

\theoremstyle{remark}
\newtheorem{qn}[lemma]{Question}

\newtheorem{remark}[lemma]{Remark}

\newtheorem{example}[lemma]{Example}

\newtheorem{warning}[lemma]{Warning}
\numberwithin{equation}{section}

\numberwithin{equation}{lemma}
 
\newcommand{\shuji}[1]{\begin{color}{green}{#1}\end{color}}
\newcommand{\shujirem}[1]{\begin{color}{red}{Shuji:\;#1}\end{color}}
\colorlet{LightRubineRed}{RubineRed!70!}

%% file: defin.tex
\def\aNis{a_{\Nis}}
\def\ulaNis{\underline{a}_{\Nis}}
\def\ulasNis{\underline{a}_{s,\Nis}}
\def\ulaNisfin{\underline{a}^{\fin}_{\Nis}}
\def\ulasNisfin{\underline{a}^{\fin}_{s,\Nis}}
\def\asNis{a_{s,\Nis}}
\def\ulasNis{\underline{a}_{s,\Nis}}
\def\qaq{\quad\text{ and }\quad}
\def\limcat#1{``\underset{#1}{\lim}"}
\def\Comp{\Comp^{\fin}}
\def\ulc{\ul{c}}
\def\ulb{\ul{b}}
\def\ulgam{\ul{\gamma}}
\def\MSm{\operatorname{\mathbf{MSm}}}
\def\MsigmaS{\operatorname{\mathbf{MsigmaS}}}
\def\ulMSm{\operatorname{\mathbf{\ul{M}Sm}}}
\def\ulMsigmaS{\operatorname{\mathbf{\ul{M}NS}}}

\def\ulMPS{\operatorname{\mathbf{\ul{M}PS}}}

\def\ulMsigmaS{\operatorname{\mathbf{\ul{M}PS}_\sigma}}
\def\ulMsigmaSTfin{\operatorname{\mathbf{\ul{M}PST}^{\fin}_\sigma}}
\def\ulMsigmaST{\operatorname{\mathbf{\ul{M}PST}_\sigma}}
\def\MsigmaS{\operatorname{\mathbf{MPS}_\sigma}}
\def\MsigmaST{\operatorname{\mathbf{MPST}_\sigma}}
\def\MsigmaSTfin{\operatorname{\mathbf{MPST}^{\fin}_\sigma}}

\def\ulMNS{\operatorname{\mathbf{\ul{M}NS}}}
\def\ulMNSTfin{\operatorname{\mathbf{\ul{M}NST}^{\fin}}}
\def\ulMNSfin{\operatorname{\mathbf{\ul{M}NS}^{\fin}}}
\def\ulMNST{\operatorname{\mathbf{\ul{M}NST}}}
\def\ulMEST{\operatorname{\mathbf{\ul{M}EST}}}
\def\MNS{\operatorname{\mathbf{MNS}}}
\def\MNST{\operatorname{\mathbf{MNST}}}
\def\MEST{\operatorname{\mathbf{MEST}}}
\def\MNSTfin{\operatorname{\mathbf{MNST}^{\fin}}}
\def\RSC{\operatorname{\mathbf{RSC}}}
\def\NST{\operatorname{\mathbf{NST}}}
\def\EST{\operatorname{\mathbf{EST}}}

\newcommand{\NS}{{\operatorname{\mathrm{NS}}}}

\def\LogRec{\operatorname{\mathbf{LogRec}}}
\def\Ch{\operatorname{\mathrm{Ch}}}

\def\MSmsq{\MSm^{\Sq}}
\def\Comp{\operatorname{\mathbf{Comp}}}
\def\uli{\ul{i}}
\def\ulis{\ul{i}_s}
\def\is{i_s}
\def\qfor{\text{ for }\;\;}
\def\CIlog{\operatorname{\mathbf{CI}}^{\mathrm{log}}}
\def\CIltr{\operatorname{\mathbf{CI}}^{\mathrm{ltr}}}
\def\CIt{\operatorname{\mathbf{CI}}^\tau}
\def\CItsp{\operatorname{\mathbf{CI}}^{\tau,sp}}
\def\ltr{\mathrm{ltr}}

\def\kX{\mathfrak{X}}
\def\kY{\mathfrak{Y}}
\def\kC{\mathfrak{C}}

\def\otCIsp{\otimes_{\CI}^{sp}}
\def\otCINissp{\otimes_{\CI}^{\Nis,sp}}

\def\hM#1{h_0^{\bcube}(#1)}
\def\hMNis#1{h_0^{\bcube}(#1)_{\Nis}}
\def\hMM#1{h^0_{\bcube}(#1)}
\def\hMw#1{h_0(#1)}
\def\hMwNis#1{h_0(#1)_{\Nis}}

\def\hetrec{h_{0, \et}}
\def\hetcube{h_{0, \et}^{\cube}}

\def\ihF#1{F^{#1}}
\def\ihFA{\ihF {\A^1}}

\def\istm{\iota_{st,m}}
\def\im{\iota_m}
\def\est{\epsilon_{st}}
\def\tL{\tilde{L}}
\def\tX{\tilde{X}}
\def\tY{\tilde{Y}}
\def\omegaCI{\omega^{\CI}}
\def\qwith{\;\text{ with} }
\def\aVNis{a^V_\Nis}
\def\ulMCorls{\ulMCor_{ls}}

\def\Zinf{Z_\infty}
\def\Einf{E_\infty}
\def\Xinf{X_\infty}
\def\Yinf{Y_\infty}
\def\Pinf{P_\infty}

\def\Lot{{\cubegm\otimes\cubegm}}

\def\Ln#1{\Lambda_n^{#1}}
\def\tLn#1{\widetilde{\Lambda_n^{#1}}}
\def\tild#1{\widetilde{#1}}
\def\otuCINis{\otimes_{\underline{\CI}_\Nis}}
\def\otCI{\otimes_{\CI}}
\def\otCINis{\otimes_{\CI}^{\Nis}}
\def\tF{\widetilde{F}}
\def\tG{\widetilde{G}}
\def\bcubered{\bcube^{\mathrm{red}}}
\def\cubegm{\bcube^{(1)}}
\def\cubegma{\bcube^{(a)}}
\def\cubegmb{\bcube^{(b)}}
\def\cubegmred{\bcube^{(1)}_{red}}
\def\cubegmreda{\bcube^{(a)}_{red}}
\def\cubegmredb{\bcube^{(b)}_{red}}

\def\LT{\bcube^{(1)}_{T}}
\def\LU{\bcube^{(1)}_{U}}
\def\LV{\bcube^{(1)}_{V}}
\def\LW{\bcube^{(1)}_{W}}
\def\LTred{\bcube^{(1)}_{T,red}}
\def\Lred{\bcube^{(1)}_{red}}
\def\LTred{\bcube^{(1)}_{T,red}}
\def\LUred{\bcube^{(1)}_{U,red}}
\def\LVred{\bcube^{(1)}_{V,red}}
\def\LWred{\bcube^{(1)}_{W,red}}
\def\PP{\P}
\def\AA{\A}

\def\LL{\bcube^{(2)}}
\def\LLred#1{\bcube^{(2)}_{#1,red}}
\def\LLredd{\bcube^{(2)}_{red}}
\def\Lredd#1{\bcube_{#1,red}}

\def\Lnredd#1{\bcube^{(#1)}_{red}}

\def\LLT{\bcube^{(2)}_T}
\def\LLTred{\bcube^{(2)}_{T,red}}

\def\LLU{\bcube^{(2)}_U}
\def\LLUred{\bcube^{(2)}_{U,red}}

\def\LLS{\bcube^{(2)}_S}
\def\LLSred{\bcube^{(2)}_{S,red}}
\def\tMCor{\Hom_{\MPST}}
\def\otHINis{\otimes_{\HI}^{\Nis}}

\def\Sh{\operatorname{\mathbf{Shv}}}
\def\Shv{\operatorname{\mathbf{Shv}}}
\def\PSh{\operatorname{\mathbf{PSh}}}
\def\Shltr{\operatorname{\mathbf{Shv}_{dNis}^{ltr}}}
\def\Shlog{\operatorname{\mathbf{Shv}_{dNis}^{log}}}
\def\Shvlog{\operatorname{\mathbf{Shv}^{log}}}
\def\SmlSm{\operatorname{\mathbf{SmlSm}}}
\def\lSm{\operatorname{\mathbf{lSm}}}
\def\lCor{\operatorname{\mathbf{lCor}}}
\def\SmlCor{\operatorname{\mathbf{SmlCor}}}
\def\PShltr{\operatorname{\mathbf{PSh}^{ltr}}}
\def\PShlog{\operatorname{\mathbf{PSh}^{log}}}
\def\lDM{\operatorname{\mathbf{logDM}^{eff}}}
\def\logDM{\operatorname{\mathbf{log}\mathcal{DM}^{eff}}}
\def\logDMone{\operatorname{\mathbf{log}\mathcal{DM}^{eff}_{\leq 1}}}
\def\logCI{\mathbf{logCI}} 

\def\DM{\operatorname{\mathbf{DM}^{eff}}}
\def\lDA{\operatorname{\mathbf{logDA}^{eff}}}
\def\DA{\operatorname{\mathbf{DA}^{eff}}}
\def\Log{\operatorname{\mathcal{L}\textit{og}}}
\def\Rsc{\operatorname{\mathcal{R}\textit{sc}}}
\def\Pro{\mathrm{Pro}\textrm{-}}
\def\pro{\mathrm{pro}\textrm{-}}
\def\dg{\mathrm{dg}}
\def\plim{\mathrm{``lim"}}
\def\ker{\mathrm{ker}}
\def\coker{\mathrm{coker}}

\def\Alb{\operatorname{Alb}}
\def\bAlb{\mathbf{Alb}}
\def\Gal{\operatorname{Gal}}

\def\hofib{\mathrm{hofib}}
\def\triv{\mathrm{triv}}
\def\ABl{\mathcal{A}\textit{Bl}}
\def\divsm#1{{#1_\mathrm{div}^{\mathrm{Sm}}}}

\def\cA{\mathcal{A}}
\def\cB{\mathcal{B}}
\def\cC{\mathcal{C}}
\def\cD{\mathcal{D}}
\def\cI{\mathcal{I}}
\def\cS{\mathcal{S}}
\def\cM{\mathcal{M}}
\def\cO{\mathcal{O}}

\def\XP{X \backslash \sP}
\def\M0a{{}^t\cM_0^a}
\newcommand{\Ind}{{\operatorname{Ind}}}

\def\Xkbar{\overline{X}_{\overline{k}}}
\def\dx{{\rm d}x}

\newcommand{\dNis}{{\operatorname{dNis}}}
\newcommand{\ABNis}{{\operatorname{AB-Nis}}}
\newcommand{\sNis}{{\operatorname{sNis}}}
\newcommand{\sZar}{{\operatorname{sZar}}}
\newcommand{\cofib}{\mathrm{Cofib}}

\newcommand{\Gmlog}{\G_m^{\log}}
\newcommand{\Gmlogred}{\overline{\G_m^{\log}}}

\newcommand{\varcolim}{\mathop{\mathrm{colim}}}
\newcommand{\varlim}{\mathop{\mathrm{lim}}}
\newcommand{\tensor}{\otimes}

\newcommand{\eq}[2]{\begin{equation}\label{#1}#2 \end{equation}}
\newcommand{\eqalign}[2]{\begin{equation}\label{#1}\begin{aligned}#2 \end{aligned}\end{equation}}

\def\varplim#1{\text{``}\varlim_{#1}\text{''}}
\def\det{\mathrm{d\acute{e}t}}

%% file: Log1Mot_final.bbl
\providecommand{\bysame}{\leavevmode\hbox to3em{\hrulefill}\thinspace}
\providecommand{\MR}{\relax\ifhmode\unskip\space\fi MR }
\providecommand{\MRhref}[2]{%
  \href{http://www.ams.org/mathscinet-getitem?mr=#1}{#2}
}
\providecommand{\href}[2]{#2}
\begin{thebibliography}{KMSY21b}

\bibitem[ABV09]{Ayoub-BV}
Joseph Ayoub and Luca Barbieri-Viale, \emph{1-motivic sheaves and the
  {A}lbanese functor}, J. Pure Appl. Algebra \textbf{213} (2009), no.~5,
  809--839.

\bibitem[AHPL16]{AHPL}
Giuseppe Ancona, Annette Huber, and Simon Pepin~Lehalleur, \emph{On the
  relative motive of a commutative group scheme}, Algebraic Geometry \textbf{3}
  (2016), 150--178.

\bibitem[BBD82]{BBD}
A.~A. Be\u{\i}linson, J.~Bernstein, and P.~Deligne, \emph{Faisceaux pervers},
  Analysis and topology on singular spaces, {I} ({L}uminy, 1981),
  Ast\'{e}risque, vol. 100, Soc. Math. France, Paris, 1982, pp.~5--171.
  \MR{751966}

\bibitem[BCKS17]{tor-div-rec}
F.~Binda, J.~Cao, W.~Kai, and R.~Sugiyama, \emph{Torsion and divisibility for
  reciprocity sheaves and 0-cycles with modulus}, J. Algebra \textbf{469}
  (2017), 437--463.

\bibitem[Ber13]{BertapelleJKT}
Alessandra Bertapelle, \emph{Remarks on 1-motivic sheaves}, J. K-Theory
  \textbf{12} (2013), no.~2, 363--380.

\bibitem[Ber14]{BertapelleJAlgebra}
A.~Bertapelle, \emph{Generalized 1-motivic sheaves}, J. Algebra \textbf{420}
  (2014), 261--268.

\bibitem[BK18]{BK}
Federico Binda and Amalendu Krishna, \emph{Zero cycles with modulus and zero
  cycles on singular varieties}, Compositio Math. \textbf{154} (2018), no.~1,
  120---187.

\bibitem[BM21]{BindaMerici}
Federico Binda and Alberto Merici, \emph{Connectivity and purity for
  logarithmic motives}, J. Inst. Math. Jussieu (2021), 1--47.

\bibitem[BM22]{BindaMericiErratum}
\bysame, \emph{Erratum to: Connectivity and purity for logarithmic motives}, J.
  Inst. Math. Jussieu (2022), 1--2.

\bibitem[BP{\O}22a]{BPOCras}
Federico Binda, Doosung Park, and Paul~Arne {\O}stv{\ae}r, \emph{Motives and
  homotopy theory in logarithmic geometry}, C. R., Math., Acad. Sci. Paris
  \textbf{360} (2022), 717--727.

\bibitem[BP{\O}22b]{BPO}
\bysame, \emph{Triangulated categories of logarithmic motives over a field},
  Ast{\'e}risque, vol. 433, Paris: Soci{\'e}t{\'e} Math{\'e}matique de France
  (SMF), 2022.

\bibitem[Bre69]{Breen}
Lawrence Breen, \emph{Extensions of abelian sheaves and {E}ilenberg-{M}aclane
  algebras}, Inventiones mathematicae \textbf{9} (1969), no.~1, 15--44.

\bibitem[BS19]{BS}
Federico Binda and Shuji Saito, \emph{Relative cycles with moduli and regulator
  maps}, J. Inst. Math. Jussieu \textbf{18} (2019), no.~6, 1233--1293.

\bibitem[BVB09]{BVBert}
Luca Barbieri-Viale and Alessandra Bertapelle, \emph{Sharp de {R}ham
  realization}, Adv. Math. \textbf{222} (2009), no.~4, 1308--1338.

\bibitem[BVK16]{BVKahn}
Luca Barbieri-Viale and Bruno Kahn, \emph{On the derived category of
  1-motives}, Ast\'erisque, vol. 381, Soc. Math. de France, 2016.

\bibitem[CS]{condensed}
Dustin Clausen and Peter Scholze, \emph{Lecture notes on condensed
  mathematics}, lectures notes for a course on condensed mathematics taught in
  the summer term 2019 at the University of Bonn.

\bibitem[Del68]{DeligneLefschetz}
Pierre Deligne, \emph{Th{\'e}or{\`e}me de {L}efschetz et crit{\`e}res de
  d{\'e}g{\'e}n{\'e}rescence de suites spectrales}, Inst. Hautes \'{E}tudes
  Sci. Publ. Math. (1968), no.~35, 107--126.

\bibitem[Del74]{DeligneHodgeIII}
\bysame, \emph{Th\'{e}orie de {H}odge. {III}}, Inst. Hautes \'{E}tudes Sci.
  Publ. Math. (1974), no.~44, 5--77. \MR{498552}

\bibitem[DG80]{DemazureGabriel}
Michel Demazure and Peter Gabriel, \emph{Introduction to algebraic geometry and
  algebraic groups}, North-Holland Mathematics Studies, vol.~39, North-Holland
  Publishing Co., Amsterdam-New York, 1980, Translated from the French by J.
  Bell.

\bibitem[ESV99]{ESV}
H{\'e}l{\`e}n Esnault, Vasudevan Srinivas, and Eckart Viehweg, \emph{The
  universal regular quotient of the chow group of points on projective
  varieties}, Invent. Math. \textbf{135} (1999), 595–664.

\bibitem[Ful98]{Fulton}
William Fulton, \emph{Intersection theory}, second ed., Ergeb. der Math.
  Grenzgeb. (3), vol.~2, Springer-Verlag, Berlin, 1998.

\bibitem[FW84]{FaltingsWustholz}
G.~Faltings and G.~W{\"u}stholz, \emph{Einbettungen kommutativer algebraischer
  {G}ruppen und einige ihrer {E}igenschaften}, J. Reine Angew. Math.
  \textbf{354} (1984), 175--205.

\bibitem[Gro85]{Gros1985}
Michel Gros, \emph{Classes de {C}hern et classes de cycles en cohomologie de
  {H}odge-{W}itt logarithmique}, M{\'e}m. Soc. Math. Fr. \textbf{21} (1985),
  1--87.

\bibitem[GSC01]{GScorr}
A~Grothendieck, J.-P. Serre, and P.~Colmez, \emph{Correspondance
  grothendieck--serre}, Documents mathématiques, vol.~2, Soci\'et\'e
  math\'ematique de France, 2001.

\bibitem[Har68]{HartshorneCohDim}
Robin Hartshorne, \emph{Cohomological dimension of algebraic varieties}, Ann.
  of Math. (2) \textbf{88} (1968), no.~3, 403--450.

\bibitem[Hin20]{Hinich}
Vladimir Hinich, \emph{So, what is a derived functor?}, Homol. Homotopy Appl.
  \textbf{22} (2020), no.~2, 279--293.

\bibitem[Hoy18]{Hoyois}
Marc Hoyois, \emph{Higher galois theory}, J. Pure Appl. Algebra \textbf{222}
  (2018), no.~7, 1859--1877.

\bibitem[Isa01]{Isaksenmodel}
Daniel~C. Isaksen, \emph{A model structure for the category of pro-simplicial
  sets}, Trans. Amer. Math. Soc. \textbf{353} (2001), 2805--2841.

\bibitem[Kat21]{FKato}
Fumiharo Kato, \emph{Integral morphisms and log blow-ups}, Isr. J. of Math.
  (2021), 1--9, to appear.

\bibitem[Kel99]{Kellercychom}
Bernhard Keller, \emph{On the cyclic homology of exact categories}, J. Pure
  Appl. Algebra \textbf{136} (1999), 1--56.

\bibitem[KMSY21a]{MotModulusI}
Bruno Kahn, Hiroyasu Miyazaki, Shuji Saito, and Takao Yamazaki, \emph{{Motives
  with modulus, I: Modulus sheaves with transfers for non-proper modulus
  pairs}}, {{\'E}pij. de G{\'e}om. Alg.} \textbf{5} (2021), 1--46.

\bibitem[KMSY21b]{MotModulusII}
\bysame, \emph{{Motives with modulus, II: Modulus sheaves with transfers for
  proper modulus pairs}}, {{\'E}pij. de G{\'e}om. Alg.} \textbf{5} (2021),
  1--31.

\bibitem[Kri06]{KrishnaThreefold}
Amalendu Krishna, \emph{Zero cycles on a threefold with isolated
  singularities}, J. Reine Angew. Math. \textbf{594} (2006), 93--115.

\bibitem[Kri10]{KrishnaArtinReese}
\bysame, \emph{An artin-rees theorem in k-theory and applications to zero
  cycles}, J. Algebraic Geom. \textbf{19} (2010), no.~3, 555--598.

\bibitem[KS83]{KatoSaitoUnCFT}
Kazuya Kato and Shuji Saito, \emph{Unramified class field theory of
  arithmetical surfaces}, Ann. of Math. (2) \textbf{118} (1983), no.~2,
  241--275. \MR{717824}

\bibitem[KS02]{KrishnaSrinivas}
Amalendu Krishna and Vasudevan Srinivas, \emph{Zero-cycles and k-theory on
  normal surfaces}, Ann. of Math. (2) \textbf{156} (2002), no.~1, 155--195.

\bibitem[KS06]{Kashiwara-Schapira}
Masaki Kashiwara and Pierre Schapira, \emph{Categories and sheaves},
  Grundlehren der Mathematischen Wissenschaften, vol. 332, Springer-Verlag,
  Berlin, 2006.

\bibitem[KS16]{KerzSaitoDuke}
Moritz Kerz and Shuji Saito, \emph{Chow group of 0-cycles with modulus and
  higher-dimensional class field theory}, Duke Math. J. \textbf{165} (2016),
  no.~15, 2811--2897.

\bibitem[KS17]{KS}
Bruno Kahn and Ramdorai Sujatha, \emph{Birational motives, {II}: Triangulated
  birational motives}, Int. Math. Res. Notices (2017), no.~22, 6778--6831.

\bibitem[KST19]{KST}
Moritz Kerz, Shuji Saito, and Georg Tamme, \emph{{$K$}-theory of
  non-{A}rchimedean rings. {I}}, Doc. Math. \textbf{24} (2019), 1365--1411.
  \MR{4012551}

\bibitem[KSY21]{KSY-RecII}
Bruno Kahn, Shuji Saito, and Takao Yamazaki, \emph{Reciprocity sheaves, {II}},
  Homology, Homotopy and Applications (2021), 25, to appear.

\bibitem[KSYR16]{KSY}
Bruno Kahn, Shuji Saito, Takao Yamazaki, and Kay R{\"u}lling, \emph{Reciprocity
  sheaves}, Compositio Mathematica \textbf{152} (2016), no.~9, 1851---1898.

\bibitem[Lur09]{HTT}
Jacob Lurie, \emph{Higher topos theory}, Ann. of Math. Studies, vol. 170,
  Princeton U. Press, 2009.

\bibitem[Lur17]{HA}
\bysame, \emph{Higher algebra},
  http://people.math.harvard.edu/~lurie/papers/HA.pdf, 2017.

\bibitem[LW85]{LW}
Marc Levine and Chuck Weibel, \emph{Zero cycles and complete intersections on
  singular varieties}, J. Reine Angew. Math. \textbf{359} (1985), 106--120.

\bibitem[Mat52]{Matsusaka}
Teruhisa Matsusaka, \emph{On a generating curve of an abelian variety}, Nat.
  Sci. Rep. Ochanomizu Univ. \textbf{3} (1952), 1--4.

\bibitem[Mil80]{MilneEtCoh}
James~S. Milne, \emph{{\'E}tale cohomology}, Princeton Math. Ser., vol.~33,
  Princeton Univ. Press, 1980.

\bibitem[Mil82]{MilneTorsion}
J.~S. Milne, \emph{Zero cycles on algebraic varieties in nonzero
  characteristic: {R}ojtman's theorem}, Compositio Math. \textbf{47} (1982),
  no.~3, 271--287.

\bibitem[Mil17]{milnegroups}
James~S. Milne, \emph{Algebraic groups: The theory of group schemes of finite
  type over a field}, Cambridge Stud. Adv. Math., vol. 170, Cambridge Univ.
  Press, 2017.

\bibitem[Mor15]{MorrowBS}
Matthew Morrow, \emph{Zero cycles on singular varieties and their
  desingularisations}, Doc. Math. \textbf{Extra Volume Merkurjev} (2015),
  465--486.

\bibitem[MVW06]{MVW}
Carlo Mazza, Vladimir Voevodsky, and Charles Weibel, \emph{Lecture notes on
  motivic cohomology}, Clay Math. Monographs, vol.~2, American Mathematical
  Society, 2006.

\bibitem[Ogu18]{ogu}
Arthur Ogus, \emph{Lectures on logarithmic algebraic geometry}, Cambridge Stud.
  Adv. Math., vol. 178, Cambridge Univ. Press, 2018.

\bibitem[Par21]{Doosung}
Doosung Park, \emph{Motivic interpretation of albanese varieties of smooth
  varieties}, arXiv preprint arxiv:1908.01582, 2021.

\bibitem[Ram01]{Ramachandran}
Niranjan Ramachandran, \emph{Duality of {A}lbanese and {P}icard 1-motives},
  $K$-Theory \textbf{22} (2001), no.~3, 271--301.

\bibitem[Roj80]{Rojtman}
A.~A. Rojtman, \emph{The torsion of the group of {$0$}-cycles modulo rational
  equivalence}, Ann. of Math. (2) \textbf{111} (1980), no.~3, 553--569.

\bibitem[RS21]{RulSaito}
Kay R{\"u}lling and Shuji Saito, \emph{Reciprocity sheaves and their
  ramification filtration}, J. Inst. Math. Jussieu (2021), 1--74.

\bibitem[Rus08]{RussellKyoto}
Henrik Russell, \emph{Generalized {A}lbanese and its dual}, J. Math. Kyoto
  Univ. \textbf{48} (2008), no.~4, 907--949.

\bibitem[Rus13]{RussellANT}
\bysame, \emph{Albanese varieties with modulus over a perfect field}, Algebra
  Number Theory \textbf{7} (2013), no.~4, 853--892.

\bibitem[RY16]{RulYama}
Kay R{\"u}lling and Takao Yamazaki, \emph{Suslin homology of relative curves
  with modulus}, J. Lond. Math. Soc. (2) \textbf{93} (2016), no.~3, 567--589.

\bibitem[Sai20]{SaitoPurity}
Shuji Saito, \emph{Purity of reciprocity sheaves}, Adv. Math. \textbf{366}
  (2020), 107067, 70.

\bibitem[Sai21]{shujilog}
\bysame, \emph{Reciprocity sheaves and logarithmic motives}, Accepted for
  publication by Compos. Math. \url{arxiv.org/abs/2107.00381}, 2021.

\bibitem[Ser60]{SerreChevalley1}
Jean-Pierre Serre, \emph{Morphismes universels et vari{\'e}t{\'e}
  d'{A}lbanese}, S{\'e}minaire {C}. {C}hevalley, 3i{\`e}me ann{\'e}e: 1958/59.
  (1960), ii+182.

\bibitem[Ser75]{SerreGACC}
\bysame, \emph{Groupes alg{\'e}briques et corps de classes}, Publ. Math. Univ.
  Nancago, vol. VII, Hermann, 1975.

\bibitem[Sri85]{SrinivasZeroCyclesII}
Vasudevan Srinivas, \emph{Zero cycles on a singular surface. ii}, J. Reine
  Angew. Math. \textbf{362} ((1985)), 4--27.

\bibitem[SS03]{SS}
Michael Spie{\ss} and Tam{\'a}s Szamuely, \emph{On the {A}lbanese map for
  smooth quasi-projective varieties}, Math. Ann. \textbf{325} (2003), no.~1,
  1--17.

\bibitem[{Sta}16]{stacks-project}
{Stacks Project Authors}, \emph{{\itshape {S}tacks {P}roject}},
  \url{http://stacks.math.columbia.edu}, 2016.

\bibitem[Voe00]{VoevTriangCat}
Vladimir Voevodsky, \emph{Triangulated categories of motives over a field},
  Cycles, transfers, and motivic homology theories, Ann. of Math. Stud., vol.
  143, Princeton Univ. Press, 2000, pp.~188--238.

\bibitem[Voe10]{Voecancel}
\bysame, \emph{Cancellation theorem}, Doc. Math. (2010), 671{\textendash}685.

\bibitem[Vol13]{Vologodsky}
Vadim Vologodsky, \emph{Some applications of weight structures of {B}ondarko},
  Int. Math. Res. Not. IMRN (2013), no.~2, 291--327. \MR{3010690}

\bibitem[Wit08]{WittenbergAlbaneseTorsors}
Olivier Wittenberg, \emph{On {A}lbanese torsors and the elementary
  obstruction}, Math. Ann. \textbf{340} (2008), no.~4, 805--838.

\bibitem[Yam17]{YamazakiLetter}
Takao Yamazaki, \emph{Personal communication}, 2017.

\end{thebibliography}
